\documentclass[12pt,reqno]{amsart}
\usepackage[margin=1in]{geometry}
\usepackage{amsmath,amssymb,amsthm,graphicx,amsxtra, setspace}
\usepackage[utf8]{inputenc}
\usepackage{mathrsfs}
\usepackage{hyperref}

\usepackage{mathtools}
\usepackage{xcolor}
\usepackage{dsfont}

\allowdisplaybreaks

\usepackage[pagewise]{lineno}

\usepackage{enumitem}
\setlist[enumerate]{leftmargin=*,widest=0}
\setlist[description]{leftmargin=*,widest=0}

\usepackage{graphicx}
\usepackage{hyperref}
\usepackage{mathtools}

\usepackage[cyr]{aeguill}

\colorlet{darkblue}{blue!50!black}

\hypersetup{
	colorlinks,%
	citecolor=blue,%
	filecolor=red,%
	linkcolor=darkblue,%
	urlcolor=blue,%
	pdfnewwindow=true,%
	pdfstartview={FitH}
}

\colorlet{darkblue}{red!100!black}

\newtheorem{theorem}{Theorem}[section]
\newtheorem{lemma}[theorem]{Lemma}

\newtheorem{corollary}[theorem]{Corollary}
\newtheorem{definition}[theorem]{Definition}

\newtheorem{remark}[theorem]{Remark}

\newtheorem{claim}{Claim}

\let\originalleft\left
\let\originalright\right
\renewcommand{\left}{\mathopen{}\mathclose\bgroup\originalleft}
\renewcommand{\right}{\aftergroup\egroup\originalright}

\theoremstyle{definition}

\def\1{\mathcal{O}}
\def\t{t\wedge\tau_N^n}
\def\tt{T\wedge\tau_N^n}
\def\d{\mathrm{d}}

\def\I{\mathrm{I}}

\def\D{\mathrm{D}}
\def\A{\mathrm{A}}
\def\W{\mathrm{W}}
\def\R{\mathbb{R}}
\def\E{\mathbb{E}}
\def\Q{\mathrm{Q}}

\def\e{\varepsilon}
\def\L{\mathrm{L}}
\def\HH{\mathrm{H}}
\def\uu{\boldsymbol{u}}  
\def\vv{\boldsymbol{v}}
\def\w{\boldsymbol{w}}

\def\f{\boldsymbol{f}}
\def\C{\mathrm{C}}

\def\P{\mathbb{P}}
\def\N{\mathbb{N}}
\def\x{\boldsymbol{x}}

\DeclareMathOperator*{\esssup}{ess\,sup}

\def\bfb{\mathbf{b}}
\def\bfB{\mathbf{B}}
\def\bfC{\mathbf{C}}
\def\bfX{\mathbf{X}}

\def\bfF{\mathbf{F}}
\def\bfG{\mathbf{G}}
\def\bfH{\mathbf{H}}
\def\bfL{\mathbf{L}}

\def\bfU{\mathbf{U}}
\def\bfS{\mathbf{S}}
\def\bfV{\mathbf{V}}
\def\bfI{\mathbf{I}}
\def\bfD{\mathbf{D}}

\def\bfG{\mathbf{G}}
\def\bfZ{\mathbf{Z}}
\def\bfe{\mathbf{e}}

\def\bfg{\mathbf{g}}
\def\bfa{\mathbf{a}}
\def\bfA{\mathbf{A}}

\def\bfX{\mathbf{X}}
\def\bfb{\mathbf{b}}
\def\bfc{\mathbf{c}}
\def\bfh{\mathbf{h}}

\def\rw{\rm{w}}

\usepackage{bm}

\def\bphi{\bm{\phi}}
\def\bfi{\boldsymbol{\varphi}}
\def\bpsi{\boldsymbol{\psi}}
\def\bxi{\boldsymbol{\xi}}

\def\bVcal{\boldsymbol{\mathcal{V}}}

\def\calS{\mathcal{S}}
\def\calA{\mathcal{A}}
\def\Mcal{\mathcal{M}}

\def\bbeta{\boldsymbol{\beta}}

\def\bz{\boldsymbol{z}}

\def\bD{\mathscr{D}}
\def\cD{\mathcal{D}}
\def\cB{\mathcal{B}}
\def\Var{\mathrm{Var}}

\DeclareMathAlphabet\mathbfcal{OMS}{cmsy}{b}{n}
\def\bcH{\mathbfcal{H}}

\usepackage{comment}
\usepackage{amsmath}
\usepackage{autobreak}

\usepackage{soul}

\newcommand{\Addresses}{{
		\footnote{
			\noindent \textsuperscript{1,3}Department of Mathematics, Indian Institute of Technology Roorkee-IIT Roorkee,
			Haridwar Highway, Roorkee, Uttarakhand 247667, INDIA.
			\\
			\textsuperscript{2}FCT - Universidade de Algarve, Faro, Portugal \& CMAFcIO - Universidade de Lisboa, Lisbon, Portugal. \par\nopagebreak
			\noindent  \textit{e-mail:} \texttt{Manil T. Mohan: maniltmohan@ma.iitr.ac.in, maniltmohan@gmail.com.}
			
			\textit{e-mail:} \texttt{Ankit Kumar: akumar14@mt.iitr.ac.in,ankitkumar.2608@gmail.com.}
			
			\textit{e-mail:} \texttt{Hermenegildo Borges de Oliveira: holivei@ualg.pt.}
			
			\noindent \textsuperscript{*}Corresponding author.
			
			\textit{Keywords:} Stochastic generalized Navier-Stokes-Voigt equations, Gaussian noise, martingale solution, strong solution.
			
			Mathematics Subject Classification (2020): Primary 60H15, 35R60; Secondary 35Q35, 76D03, 76A05.
}}}

\begin{document}

	\title[Well-posedness of weak solutions for stochastic NSV equations]{Existence and uniqueness of weak solutions for the generalized stochastic Navier-Stokes-Voigt equations\Addresses}
	
	\author[A. Kumar, H. B. de Oliveira and M. T. Mohan]
	{Ankit Kumar\textsuperscript{1}, Hermenegildo Borges de Oliveira\textsuperscript{2} and Manil T. Mohan\textsuperscript{3*}}

	\maketitle
	
	\begin{abstract}
		In this work, we consider  the incompressible generalized Navier-Stokes-Voigt equations in a bounded domain $\1\subset\R^d$, $d\geq 2$, driven by a multiplicative Gaussian noise.
The considered momentum equation is given by:
		\begin{align*}
				\mathrm{d}\left(\boldsymbol{u} - \kappa \Delta \boldsymbol{u}\right) = \left[\f +\operatorname{div} \left(-\pi\mathbf{I}+\nu|\mathbf{D}(\boldsymbol{u})|^{p-2}\mathbf{D}(\boldsymbol{u})-\boldsymbol{u}\otimes \boldsymbol{u}\right)\right]\mathrm{d} t + \Phi(\boldsymbol{u})\mathrm{d} \W(t).
		\end{align*}
In the case of $d=2,3$, $\boldsymbol{u}$ accounts for the velocity field, $\pi$ is the pressure, $\f$ is a body force and the final term stay for the stochastic forces.
Here, $\kappa$ and $\nu$ are given positive constants that account for the kinematic viscosity and relaxation time, and the power-law index $p$ is another constant (assumed $p>1$) that characterizes the flow. We use the usual notation $\mathbf{I}$ for the unit tensor and $\mathbf{D}(\boldsymbol{u}):=\frac{1}{2}\left(\nabla \boldsymbol{u} + (\nabla \boldsymbol{u})^{\top}\right)$ for the symmetric part of velocity gradient.
For $p\in\big(\frac{2d}{d+2},\infty\big)$, we first prove the existence of a \emph{martingale solution}.
Then we show the \emph{pathwise uniqueness of solutions}.
We employ the classical \emph{Yamada-Watanabe theorem} to ensure the existence of a unique \emph{probabilistic strong solution}.
	\end{abstract}
	\section{Introduction}\label{sec1}\setcounter{equation}{0}

	\subsection{The model}
	We study flows of incompressible fluids with elastic properties  governed by the incompressible generalized Navier-Stokes-Voigt (NSV) equations in a bounded domain
		$\1\subset\mathbb{R}^d$, $d\geq 2$, with its boundary denoted by $\partial\1$, during the time interval $[0, T]$, with $0<T<\infty$.
		The momentum equation is assumed here to be perturbed by a stochastic term, that is, infinite-dimensional multiplicative Gaussian noise which is the product of  $\Phi(\boldsymbol{u})$ and a Wiener process $\W(\cdot)$.
	The problem under consideration is posed by the following system of equations:
	\begin{alignat}{2}
		\label{1.1}
		\d\left(\uu - \kappa \Delta \uu\right) &= \left[\f + \operatorname{div} \left(-\pi\bfI+\nu|\bfD(\uu)|^{p-2}\bfD(\uu)-\uu\otimes \uu\right)\right]\d t + \Phi(\uu)\d \W(t), \text{ in }\1_T,  \\
		\label{1.2}
		\operatorname{div}\uu &=0, \text{ in }\1_T,  \\
		\label{1.3}
		\uu &=\uu_0, \text{ in }\1\times\{0\},  \\
		\label{1.4}
		\uu &=\boldsymbol{0}, \text{ on }\Gamma_T,
	\end{alignat}
	where $\1_T:=\1\times(0,T)$ and $\Gamma_T:=\partial\1\times[0,T]$.	Here, $\uu=(u_1,\dots,u_d)$ and $\f=(f_1,\dots,f_d)$ are vector-valued functions, $\pi $ is a scalar-valued function, whereas $\nu$ and $\kappa$  are positive constants.	In real-world applications, space dimensions of interest are $d=2$ and $d=3$, and in that case $\uu$ accounts for the velocity field, $\pi $ is the pressure,	$\f$ represents an external forcing field. Greek letters $\nu$ and $\kappa$ stay for positive constants that, in the dimensions of physical interest, account for the kinematics viscosity and for the relaxation time, that is,  the time required for a viscoelastic fluid relax from a deformed state back to its equilibrium configuration.
The power-law index $p$ is a constant that characterizes the flow and is assumed to be such that $1<p<\infty$.
	 In particular, for $1<p<2,$ the model describes shear-thinning fluids, when $p=2,$ we recover the model that governs Newtonian fluids, while for $p>2,$ the model describes shear-thickening fluids.
	The capital letters stay for tensor-valued functions, in particular $\bfI$ is the unit tensor and $\bfD(\uu)$ is the symmetric part of the velocity gradient, that is, $\bfD(\uu):=\frac{1}{2}\left(\nabla \uu + (\nabla \boldsymbol{u})^{\top}\right)$.
	The idea behind the presence of the noise $\Phi(\uu)$ is an interaction between the solution $\uu$ and the random perturbation caused by the Wiener process $\W(\cdot)$, and the product of these two terms represents the stochastic part in the momentum equation (\ref{1.1}).	This stochastic part can be interpreted in different ways. It can be understood as a turbulence in the fluid motion or it can be interpreted as a perturbation from the physical model. On the other hand, apart from the forcing term $\f,$ there might be further quantities with influence on the motion which can be characterized by this stochastic part as well.
	
	\subsection{Literature review}We start with the physical significance of  NSV equations, also named Kelvin-Voigt (KV) equations, as we have considered the generalized stochastic NSV equations \eqref{1.1}-\eqref{1.4}. KV equations are used in the applications to model the response of materials that exhibit all intermediate range of properties between an elastic solid and a viscous fluid, as, for instance, polymers. These materials have some memory in the sense that they can come back to some previous state when the shear stress is removed (see~\cite{Barnes:2000}). The deterministic KV equations were extensively studied in the article \cite{APO}, whose  author coined the name Kelvin-Voigt for the incompressible Navier-Stokes equations perturbed by the relaxation term $\kappa \partial_t\Delta \uu$ (see \cite{MTM-EECT} also). However, and as observed in the work \cite{VGZMVT}, neither Kelvin nor Voigt have suggested any stress-strain relation, or system of governing equations for viscoelastic fluids (see \cite{DDJ}).  Currently, the Navier-Stokes-Voigt name  for the associated system of equations seems to be the most accepted by the researchers in this field. Mathematically speaking, and was noticed by Ladyzhenskaya (see \cite{OAL2}), the most important property of NSV equations is that the term $\kappa\partial_t\Delta\uu$ works as a regularization of the 3D incompressible Navier-Stokes (NS) equations ensuring the unique global solvability of the corresponding problem. The existence and uniqueness of the weak and strong solutions of NSV equations have been established in the work \cite{APO} (cf. \cite{OAL3}). A local in time existence and uniqueness result of an inverse problem for NSV equations is established in \cite{PKKKMT}. Several other authors  have also studied some variants of the NSV equations,  in many settings and under different assumptions, with respect to the well-posedness and long time behavior. It is worth to mention that the works \cite{AAKAPORS,APO2} established a connection between NSV and the Oldroyd models, and proved the existence of classical global solutions for the models. Also, in the works \cite{VOPZVZ,VGZMVT,VGVO}, the authors proved the existence and uniqueness of weak solutions in case of domain varying with time and with the optimal control. The authors in \cite{SNAHBGKK2,SNAHBGKK3,SNAHBGKK4,SNAKK1,SNAKK2}  discussed the existence, uniqueness and some qualitative properties of the solutions to some variants of {NSV} equations, more specifically the case with power-laws and  anisotropic diffusion and several other properties. In \cite{SNAHBGKK1}, the authors considered the {generalized} {NSV} equations for non-homogeneous and incompressible fluid flows with the diffusion and relaxation terms described by two distinct power-laws and proved the existence of weak solutions as well as large-time behavior of the solutions.   The works \cite{YCMELEST,VKKEST,BLFREST,FRCST} established a clear connection between {NSV} equations and the turbulence modeling, to be more specific, with Bardina turbulence models. The same type of connection was firstly studied in the work \cite{WLRL} and followed by the work \cite{RLCLB}. The authors in \cite{CACLBRLBBN} connects {NSV} equations with the Prandtl-Smagorinsky and turbulent kinematic energy models in order to compute the turbulent viscosity. Several questions on {NSV} equations as a regularization of both NS and Euler equations have been tackled in  \cite{CLBLB,CLBSS,VKKEST}. Also, the author in the work \cite{HBG} considered {generalized} NS equations (NS equations with power-law) with the presence of damping term in the momentum equation which governs isothermal flows of incompressible and homogeneous  non-Newtonian fluids and established the existence of weak solutions using regularization techniques, the monotone theory, compactness arguments and truncation method.  In a recent work \cite{SNAHBGKK5}, the authors considered the classical {NSV} problem for incompressible fluids with unknown non-constant density and established the existence of weak solutions (velocity and density), unique pressure  and some regularity results of the solutions on the smoothness of the given data.
	
	We are interested in the {generalized} stochastic NSV equations \eqref{1.1}-\eqref{1.4}.
{As mentioned above, from} several physical point of view it is important to have a stochastic part in the equation of motion:
	\begin{enumerate}
		\item It gives a clear understanding of turbulence in the fluid motion (see \cite{RMBR}).
		\item It could be viewed as a deviation from the physical model.
		\item We are seeing that there may be additional values, generally minor ones, that have an impact on the motion in addition to the force $\f$.
	\end{enumerate}
Several works are available in the literature dealing with the existence of weak solutions of  stochastic NS equations starting from the article \cite{ABRT}. One can see the work \cite{FF} for an overview. However, only a few works are available {dealing with the generalized ($p\not=2$) NS problem}. In the article \cite{JCZMC}, the authors discussed  the bipolar shear thinning fluid. In that work, the presence of additional bi-Laplacian term $\Delta^2\uu$ makes the computations easy. Since,  the additional term gives enough initial regularity which leads to the proof using monotone operators without any truncation arguments.

In \cite{YTNY,NY}, the authors considered the stochastic {power-law} fluids and established the existence and uniqueness of weak solutions and strong convergence of the Galerkin approximation and energy equality, respectively. There were no interaction between the solution and Brownian motion in \cite{YTNY,NY} (additive noise). Further, in \cite{Breit}, the author extended the existence theory of the stochastic {power-law} fluids and improved the results of \cite{YTNY,NY}. He proved the existence of weak solutions (in both analytic and probabilistic sense), using the standard procedure as the pressure term disappear in the weak formulation and later it reconstructed into the formulation. He borrowed the idea from \cite{JW} (for deterministic setting), which relates each term in the equation a pressure part. Moreover, he stabilized the auxiliary equations by adding a large power of $\uu$ and the approach was based on Galekin method. Later, he extended the deterministic approach of the work \cite{JW} to the stochastic case and combined the techniques from nonlinear partial differential equations with stochastic calculus for martingales. Since the author cannot   apply test functions in this situation, he instead employed It\^o's formula to obtain energy estimates  and other stochastic tools.  Finally, in order to justify the limit in the nonlinear terms, he used the monotone operator theory together with the {$\L^\infty$--truncation}.

Let us mention some works available for stochastic NSV equations.  In \cite{HLCS}, the authors considered 3D stochastic NSV equations in bounded domains perturbed by a finite-dimensional Brownian motion and established an LDP for the solution  using weak convergence approach. The existence and long-time behavior of the solutions (more specifically mean square exponentially stability and the almost sure exponential stablility of the stationary solutions) to 3D stochastic NSV equations in bounded domains perturbed by infinite-dimensional Wiener process obtained in \cite{CTHNVT}. In a recent work \cite{QZ}, the author studied the stability of pullback random attractors for 3D stochastic NSV equations with delays.

	\subsection{Objectives and novelties of the paper}The main aim of this work is to prove the existence of martingale solution and unique probabilistic strong solution  of {generalized} stochastic NSV equations perturbed by infinite-dimensional Wiener process. {There seems to be a very scarce knowledge about the model presented in this work, and to the authors best knowledge the analysis presented here has not yet been investigated in the available literature}.
	\begin{itemize}
		\item As we have mentioned above, there are only a few works available in the literature dealing with stochastic {power-law} {fluids~\cite{Breit,YTNY,NY}. In the first work \cite{Breit}, the author proved} the existence theory for stochastic {power-law} fluids driven by an infinite-dimensional multiplicative Gaussian noise. The latter two {works~\cite{YTNY,NY} discussed the existence and uniqueness of weak solutions for stochastic {power-law} fluids driven by a} finite-dimensional Brownian motion.
		\item In this work, we are following the stochastic approach {considered in the work \cite{Breit}, which in turn was an extension to the stochastic case of the deterministic problem studied in the work \cite{JW}}. We are also considering the perturbation by an infinite-dimensional multiplicative Gaussian noise.
		
	\end{itemize}

{
In addition to the stochastic situation, the main conclusion of this work is that the Voigt term $\kappa\partial_t\Delta\boldsymbol{u}$ regularizes the power-law model of the NS equations in such a way that the existence result can be proven up to the smallest possible value (in the sense that the Gelfand triple can still be used) of the power-law exponent $p$ ($p>\frac{2d}{d+2}$), (see Theorem~\ref{thm:exist} below).
The NS power-law model, which corresponds to considering $\kappa=0$ in the momentum equation \eqref{1.1}, was considered in \cite{Breit}.
But there the author used the $\L^\infty$--truncation method, and obtained the existence result for values of $p>\frac{2(d+1)}{d+2}$.
}

\subsection{Organization}
	This article is organized as follows: We start with the basic formulation of the stochastic background, as well as  Definitions \ref{def.1} and \ref{def1.2} of martingale solution and probabilistically strong solution, respectively, and then we state our main Theorem \ref{thm:exist} in Section \ref{sec2}.
    In Section \ref{sec4}, we discuss the decomposition of pressure terms.
    {As in the weak formulation of the problem, the pressure term disappears (see Definition \ref{def.1}), later it can be recovered by using a version of de Rham's theorem.   The idea of decomposing the pressure term that we use in the sequel is borrowed from the work~\cite{Breit}, but its first application dates back to \cite{JW}}.   In Section \ref{sec5}, we consider an approximate system with a stabilization term (a large power of $\uu$), and, by using Galerkin approximations, we obtain the uniform energy estimate, followed by the existence of martingale solution for this approximate system, and for which we still have to use the Skorokhod representation theorem.  Later, we  reconstruct  the pressure term and using monotone operator theory leads to existence of martingale solution to the system \eqref{1.1}-\eqref{1.4}. We wind up the article by  establishing the pathwise uniqueness of the solution to the system \eqref{1.1}-\eqref{1.4} and hence the existence of a unique probabilistic strong solution by an application of  the classical Yamada-Watanabe theorem.
	
	
	\section{Auxiliary results}\label{sec2}\setcounter{equation}{0}
		\subsection{Function spaces}
	Let us  introduce the functions spaces used to characterize the solutions of Fluid Dynamics problems.	Let $\1$ be a domain in $\R^d$, that is,  an open and connected subset of $\R^d$.	By ${\C_0^{\infty}(\1)^d},$ we  denote the space of all infinitely differentiable $\R^d-$valued functions with compact support in $\1$.	For $1\leq p\leq \infty$, we denote by ${\L^p(\1)^d},$ the Lebesgue space consisting of all $\R^d-$valued measurable (equivalence classes of) functions that are $p-$summable over $\1$.	 {By $|\1|,$ we denote the $d$--Lebesgue measure of $\1$.} The corresponding Sobolev spaces are represented by ${\W^{k,p}(\1)^d}$, where $k\in\N$.	When $p=2$, ${\W^{k,2}(\1)^d}$ are Hilbert spaces that we denote by ${\HH^k(\1)^d}$.	We define
	\begin{align*}
		\bVcal &:=\big\{\uu\in{\C_0^{\infty}(\1)^d}:\nabla\cdot \uu=0\big\},  \\
		\bfH &:= \text{ the closure of } \bVcal \text{ in the Lebesgue space } {\L^2(\1)^d}, \\
			\L_\sigma^p(\1)^d&:= \text{ the closure of }  \bVcal \text{ in the Lebesgue space } {\L^p(\1)^d}, \\
		\W_{\sigma}^{p,k}(\1)^d&:= \text{ the closure of }  \bVcal \text{ in the Sobolev space } {\W^{k,p}(\1)^d},
	\end{align*}
	for $1\leq p<\infty$ and $k\in\N$.
	In the particular case of $p=2$, we denote the space $	\W_{\sigma}^{p,k}(\1)^d$ by $\bfV^k$, and if $k=1$, we denote it by $\bfV_p$.
	If both $p=2$ and $k=1$, we denote $	\W_{\sigma}^{p,k}(\1)^d$ solely by $\bfV$.
	Let $(\cdot,\cdot)$ stand for the inner product of the Hilbert space {$\L^2(\1)^d$}, and we denote by $\langle \cdot,\cdot\rangle ,$ the induced duality product between the space ${\W_0^{1,p}(\1)^d}$  and its dual ${\W^{-1,p'}(\1)^d}$, as well as between ${\L^p(\1)^d}$ and its dual ${\L^{p'}(\1)^d}$, where $\frac{1}{p}+\frac{1}{p'}=1$.
{In the sequel, the $\L^p$, $\W^{k,p}$ and $\W^{-k,p'}$ norms will be denoted in short by $\|\cdot\|_p$, $\|\cdot\|_{k,p}$ and $\|\cdot\|_{-k,p'}$.
On $\bfV$, we consider the equivalent norm $\|\uu\|_{\bfV}:=\|\nabla\uu\|_{2}$, $\uu\in\bfV$.
}
	
	\subsection{Stochastic setting}\label{SS}
	In this subsection, we start by recalling some notions of stochastic processes necessary to study stochastic partial differential equations.	We consider the scalar-valued case, being the vectorial case a straightforward {generalization}. Let $(\Omega,\mathscr{F},\P)$ be a probability space, where $\Omega$ denotes the set of all possible outcomes $\omega$, $\mathscr{F}$ is the set of events and $\P:\mathscr{F}\longrightarrow[0,1]$ accounts for the probability measure function.	We assume that $(\Omega,\mathscr{F},\P)$  is equipped with a filtration $\{\mathscr{F}_t\}_{t\in [0,T]}$, that is, with a nondecreasing family of sub-sigma fields of $\mathscr{F}$ ($\mathscr{F}_s\subset\mathscr{F}_t$ whenever  $s< t$), such that $\mathscr{F}_0$ contains all the null elements (elements $A\in\mathscr{F}$ with $\P(A)=0$) and $\{\mathscr{F}_t\}_{t\in [0,T]}$ is right-continuous ($\mathscr{F}_t=\mathscr{F}_{t+}=\cap_{s>t}\mathscr{F}_s$ for all $t\in [0,T]$).
	In this case, we denote a filtered probability space by $(\Omega,\mathscr{F},\{\mathscr{F}_t\}_{t\in [0,T]},\P)$.
	Letting $1\leq p\leq\infty$ and given a Banach space $\bfB$, we denote by $\L^p(\Omega,\mathscr{F},\P;\bfB),$ the space of all measurable mappings $v:(\Omega,\mathscr{F})\longrightarrow (\bfB,\mathscr{B}(\bfB))$, where $\mathscr{B}(\bfB)$ denotes the Borel $\sigma-$algebra on $\bfB$, such that
	$$\E\big[\|v\|_\bfB^p\big]<\infty,$$
	where the expectation $\E$ is taken with respect to $(\Omega,\mathscr{F},\P)$.
	A $\bfB-$valued stochastic process can be interpreted as a mapping $X:[0,T]\times\Omega\longrightarrow (\bfB,\mathscr{B}(\bfB))$, denoted by $X(t,\omega)$.
	For a fixed $\omega\in\Omega$, the mapping $t\mapsto X(t,\omega)$ is called the \emph{path or trajectory} of $X$.
	We say that a $\bfB-$valued stochastic process $X$ is \emph{adapted to a filtration} $\{\mathscr{F}_t\}_{t\in [0,T]}$, or just $\mathscr{F}_t-$adapted, if
	the mapping $X(t):\Omega\longrightarrow (\bfB,\mathscr{B}(\bfB))$, defined by $\omega\mapsto X(t,\omega)$, is measurable for all $t\in[0,T]$.
	The $\bfB-$valued stochastic process $X$ is \emph{progressively measurable} on the probability space $(\Omega,\mathscr{F},\P)$, with respect to the filtration $\{\mathscr{F}_t\}_{t\in [0,T]}$, that is, is $\{\mathscr{F}_t\}_{t\in [0,T]}-$measurable, if the mapping
	$(s,\omega)\longmapsto X(s,\omega)$
	is measurable on $\mathscr{B}([0,t])\otimes\mathscr{F}_t$ for all $t\in[0,T]$. 
	Note that every progressively measurable stochastic process is $\mathscr{F}_t-$adapted.

	We are now ready to define the notion of integral of a stochastic process with values in specific Banach spaces. 	To this end, we first consider a bounded linear operator $\Q\in{\mathcal{L}(\bfU,\bfU)}$, where {$\bfU$} is a separable Hilbert space with inner product denoted by $(\cdot,\cdot)_{\bfU}$ and associated norm by $\|\cdot\|_\bfU$.
	If  $\Q$ is supposed to be nonnegative, symmetric and with finite trace, then there exists an orthonormal basis $\{\bfe_k\}_{k\in\N}$ of  {$\bfU$} such that
	$$ \Q \bfe_k=\lambda_k\bfe_k,\qquad \lambda_k\geq 0,\quad  \forall\ k\in\N$$
	and so that $0$ is the only accumulation point of the sequence $\{\lambda_k\}_{k\in\N}$.
	For a fixed  {$\Q\in\mathcal{L}(\bfU,\bfU)$} nonnegative, symmetric and with finite trace, a $ {\bfU}-$valued stochastic process $\W$ on the probability space $(\Omega,\mathscr{F},\P)$ is a standard $\Q-$Wiener process if and only if
	$$\W(t)=\sum_{k\in\N}\sqrt{\lambda_k}\beta_k(t)\bfe_k,\qquad t\in[0,T],$$
	where $\beta_k$, with $k\in\{n\in\N:\lambda_n>0\}$, are independent real-valued Brownian motions on $(\Omega,\mathscr{F},\P)$.
	A $\Q-$Wiener process $\W$ is said to be adapted to a filtration $\{\mathscr{F}_t\}_{t\in [0,T]}$, if $\W(t)\in \mathscr{F}_t$ for every $t\in[0,T]$ and $\W(t)-\W(s)$ is independent of $\mathscr{F}_s,$ for all $0\leq s\leq t\leq T$.
	Let us now set  {$\bfU_0=\Q^{\frac{1}{2}}\bfU$}, where  $\Q^{\frac{1}{2}}$ is the operator defined by
	$$ \Q^{\frac{1}{2}}\bfe_k=\sqrt{\lambda_k}\bfe_k.$$
	Then  {$\bfU_0$} is a Hilbert space with inner product
	$$(\uu,\boldsymbol{v})_0=\big( \Q^{-\frac{1}{2}}\uu, \Q^{-\frac{1}{2}}\boldsymbol{v}),\qquad \uu,\boldsymbol{v}\in  {\bfU_0},$$
	and associated norm denoted by $\|\cdot\|_0$, where $ \Q^{-\frac{1}{2}}$ is the pseudo-inverse operator of $ \Q^{\frac{1}{2}}$.
	If $\{\bfe_k\}_{k\in\N}$ is an orthonormal basis of ${\bfU}_0$ and $\{\beta_k\}_{k\in\N}$ is a family of independent real-valued Brownian motions on $(\Omega,\mathscr{F},\P)$, then the series $\sum\limits_{k\in\N}\sqrt{\lambda_k}\beta_k(t)\bfe_k$ is convergent in $\L^2(\Omega,\mathscr{F},\P;{\bfU})$, because
	the inclusion ${\bfU}_0\subset {\bfU}$ defines a Hilbert-Schmidt embedding from ${\bfU}_0$ onto ${\bfU}$.
	Recall that by $\mathcal{L}_2({\bfU}_0,{\bfU})$, we denote the set of all linear operators {$\Q:\bfU_0\longrightarrow \bfU$} such that for every orthonormal basis $\{\bfe_k^0\}_{k\in\N}$ of ${\bfU}_0,$ one has $\sum\limits_{k\in\N} \|\Q \bfe_k^0\|_{\bfU}^2<\infty.$
	In the case that $\Q$ is no longer of finite trace, one looses this convergence, but it is possible to define the Wiener process.
	For that, we consider a further Hilbert space ${\bfU}_1$ and a Hilbert-Schmidt embedding $J:{\bfU}_0\longrightarrow {\bfU}_1$.
	Let us set $\Q_1=JJ^\ast$, where
	$J^\ast:{\bfU}_1\longrightarrow {\bfU}_0$ is the adjoint operator of $J$.
	Then $\Q_1\in \mathcal{L}({\bfU}_1,{\bfU}_1)$, $\Q_1$ is nonnegative definite and symmetric with finite trace and the series
	$\sum\limits_{k\in\N}\beta_k(t)J\bfe_k$
	converges in $\L^2(\Omega,\mathscr{F},\P;{\bfU}_1)$ and defines a $\Q_1-$Wiener process on ${\bfU}_1$,
	\begin{align}\label{2.1}
		\W(t) =\sum_{k\in\N}\beta_k(t)J\bfe_k,\qquad  t\in[0, T].
	\end{align}
	Moreover, we have that $\Q^{\frac{1}{2}}_1({\bfU}_1)=J({\bfU}_0)$ and $J:{\bfU}_0\longrightarrow \Q^{\frac{1}{2}}_1({\bfU}_1)$ is an isometry.
	Since the Hilbert space  ${\bfU}_1$ and the Hilbert-Schmidt embedding $J$ considered above always exist, (\ref{2.1}) defines a Wiener process, even {though} $\Q$ is not of finite trace. The stochastic process defined this way is usually referred to as a \emph{cylindrical Wiener process} in ${\bfU}$.
	In this case, we say that a process $\phi(t)$, $t\in[0,T]$, is integrable with respect to $\W(t)$, if it takes values
	in $\mathcal{L}_2(\Q^{\frac{1}{2}}({\bfU}_1),{\bfU})$, is predictable and if
	$$\P\left(\int_0^T\|\phi(s)\|_{\mathcal{L}_2(\Q^{\frac{1}{2}}({\bfU}_1),{\bfU})}\d s<\infty\right)=1,$$
	where $\mathcal{L}_2(\Q^{\frac{1}{2}}({\bfU}_1),{\bfU})$ denotes the set of all Hilbert-Schmidt operators from $\Q^{\frac{1}{2}}({\bfU}_1)$ to ${\bfU}$.
	Hence, we can define the stochastic integral
	$$\int_0^t\phi(s)\d\W(s)=\int_0^t\phi(s)\circ J^{-1}\d\W(s)=\sum_{k\in\N}\int_0^t\phi(s)\bfe_k\d\beta_k(s),$$
	which is independent of the choice of ${\bfU}_1$ and $J$. For more details on stochastic integration, we are referring to  \cite{DaZ,WLMR2015,PR:2007}.

	Next, we define some probabilistic spaces of time-dependent functions that we deal throughout this work.	Given $1\leq q\leq \infty$ and a separable Banach space $\bfB$, we consider the Bochner space $\L^q(0,T;\bfB)$ formed by all the $\bfB-$valued measurable functions $\uu$ on $[0,T]$ such that
	\begin{align*}
		\|\uu\|_{\L^q(0,T;\bfB)}=\left(\int_0^T\|\uu(t)\|_\bfB^q\d t\right)^{\frac{1}{q}}<\infty.
	\end{align*}
{In the particular case of $\bfB=\L^q(\1)^d$, we shall denote $\L^q(0,T;\bfB)$ in short by $\bfL^{q}(\1_T)$.} Let  $\L_{\rm w}^q(0,T;\bfB)$ denote a space $\L^q(0,T;\bfB)$ with the weak topology $\rm w$.

The space $\C([0,T]{\color{red};}\bfB)$ represents the Banach space of all continuous functions from $[0,T]$ to $\bfB$ with the norm $$\|\uu\|_{\C([0,T];\bfB)}=\sup
_{t\in[0,T]}\|\uu(t)\|_\bfB.$$
Also, $\C_{\rw}([0,T];\bfB)$ denotes the subspace of $\L^\infty(0,T;\bfB)$ formed by weakly continuous functions from $[0,T]$ into $\bfB$.

	For $1\leq p<\infty$, $1\leq q\leq \infty$ and a separable Banach space $\bfB$, we denote by $\L^p(\Omega,\mathscr{F},\P;\L^q(0,\\T;\bfB)),$
	the space of all functions $\uu=\uu(x,t,\omega)$ defined on $\1\times[0,T]\times\Omega$ and with values in $\bfB$ such that
	$\uu$ is measurable with respect to $(t,\omega)$ and $\uu$ is $\mathscr{F}_t-$ measurable for a.e. $t$, and
	\begin{align*}
		\|\uu\|_{\L^p(\Omega,\mathscr{F},\P;\L^q(0,T;\bfB))}=
		\left\{
		\begin{array}{ll}
			\displaystyle \left\{\E\left[\left(\int_0^T\|\uu(t)\|_\bfB^q\d t\right)^{\frac{p}{q}}\right]\right\}^{\frac{1}{p}}<\infty, & 1\leq q<\infty, \\
			\displaystyle \left\{\E\left[\displaystyle \esssup_{t\in[0,T]} \|\uu(t)\|_\bfB ^p\right]\right\}^{\frac{1}{p}}<\infty, & q=\infty.
		\end{array}
		\right.
	\end{align*}

	By $\L^p(\Omega,\mathscr{F},\P;\C([0,T];\bfB))$, with $1\leq p<\infty$, we denote the space of all continuous and progressively $\{\mathscr{F}_t\}_{t\in [0,T]}-$measurable $\bfB-$valued stochastic processes $\uu$ such that
	\begin{align*}
		\|\uu\|_{\L^p(\Omega,\mathscr{F},\P;\C([0,T];\bfB))}=\left\{\E\left[\displaystyle \sup_{t\in[0,T]} \|\uu(t)\|_\bfB ^p\right]\right\}^{\frac{1}{p}}<\infty.
	\end{align*}

	We recall an important result on stochastic ordinary differential equations in finite-dimensions, where we denote by $\Mcal^{d\times d}(\R),$ the vector space of all real $d\times d-$matrices.
	\begin{lemma}\label{prop:e:sde}
		Let the map {$\mathbf{M}:[0,T]\times\R^d\times\Omega\longrightarrow\Mcal^{d\times d}(\R)$, defined by $(t, x, \omega)\mapsto \mathbf{M}=\mathbf{M}(t, x, \omega)$,} and
		{$\bfb:[0,T]\times\R^d\times\Omega\longrightarrow\R^d$, defined by $(t, x, \omega)\mapsto \bfb=\bfb(t, x, \omega)$} be both continuous in $x\in\R^d$ for $t\in[0,T]$ and $\omega\in\Omega$, and be progressively measurable on the probability space $(\Omega,\mathscr{F},\P)$ with respect to the filtration $\{\mathscr{F}_t\}_{t\in [0,T]}$.
{
		Assume the following conditions are verified for all $R\in(0,\infty)$: 
		\begin{enumerate}
            \item $\displaystyle \int_0^T \sup_{|x|\leq R}\big(|\bfb(t,x)|+|\mathbf{M}(t,x)|^2\big)dt<\infty$;
			\item $2\big(x-y\big)\cdot\big(\bfb(t,x)-\bfb(t,y)\big) + \big|\mathbf{M}(t,x)-\mathbf{M}(t,y)\big|^2\leq K_t(R)$,\quad $\forall\ t\in[0,T]$, $\forall\ x,y\in\R^d : |x|,|y|\leq R$;
			\item $2x\cdot \bfb(t,x) + \big|\mathbf{M}(t,x)\big|^2\leq K_t(1)\big(1+|x|^2\big)$,\quad $\forall\ t\in[0,T]$, $\forall\ x\in\R^d : |x|\leq R$;
		\end{enumerate}
}
		where $K_t(R)$ is an $\R^+-$valued process adapted to the filtration $\{\mathscr{F}_t\}_{t\in [0,T]}$ and such that
{		$$\alpha_S(R):=\int_0^S K_t(R)\d t<\infty,\quad \forall\ S\in [0,T].$$ }
		Then for any $\mathscr{F}_0-$measurable map $X_0:\Omega\longrightarrow \R^d,$  there exists a \emph{unique solution} to the stochastic differential equation
		$$ \d X(t)=\bfb\big(t,X(t)\big)\d t + \mathbf{M}\big(t,X(t)\big)\d\W(t).$$
	\end{lemma}
	
	\begin{proof}
See \cite[Theorem 3.1.1]{PR:2007}.
	\end{proof}
	

	We recall also the following inequalities which are classical in the theory of $p$-Laplace equations.
	\begin{lemma}\label{GM}
		For all {$\mathbf{M},\ \mathbf{N}\in\R^{d\times d}$}, the following assertions hold true:
{
		\begin{alignat}{2}
			\label{2.2}
			2\leq p<\infty & \quad\Rightarrow\quad \frac{1}{2^{p-1}}|\mathbf{M}-\mathbf{N}|^p\leq \big(|\mathbf{M}|^{p-2}\mathbf{M}-|\mathbf{N}|^{p-2}\mathbf{N}\big):\big(\mathbf{M}-\mathbf{N}\big);&& \\
			\label{2.3}
			1<p<2 & \quad\Rightarrow\quad (p-1)|\mathbf{M}-\mathbf{N}|^2\leq \big(|\mathbf{M}|^{p-2}\mathbf{M}-|\mathbf{N}|^{p-2}\mathbf{N}\big):(\mathbf{M}-\mathbf{N})\big(|\mathbf{M}|^p+|\mathbf{N}|^p\big)^{\frac{2-p}{p}}.&&
		\end{alignat}
}
	\end{lemma}
	\begin{proof}
		{The proof combines \cite[Lemmas 5.1 and 5.2]{GM:1975}}.
	\end{proof}
	
	From (\ref{2.2}) and (\ref{2.3}), one easily gets
	\begin{align}\label{2.4}
		\big(|\bfD(\uu)|^{p-2}\bfD(\uu)-|\bfD(\vv)|^{p-2}\bfD(\vv)\big):\big(\bfD(\uu)-\bfD(\vv)\big)\geq 0,\  \forall \  \uu, \vv\in {\W_0^{1,p}(\1)^d},
	\end{align}
	and whenever $1<p<\infty$.
	As a consequence of (\ref{2.4}), the operator
\begin{equation}\label{op:A(u)}
\bfA(\uu):=|\bfD(\uu)|^{p-2}\bfD(\uu)
\end{equation}
is said to be \emph{monotone } for any $p$ such that $1<p<\infty$.
	
	\subsection{Probabilistic weak formulation}
	
	We are interested in martingale solutions, that is, weak solutions in the probabilistic sense, to the stochastic problem (\ref{1.1})-(\ref{1.4}).
	This means that when seeking  martingale solutions of (\ref{1.1})-(\ref{1.4}), constructing a  filtered probability space $(\Omega,\mathscr{F},\{\mathscr{F}_t\}_{t\in [0,T]},\P)$ and a cylindrical Wiener process $\W$ on it are both part of the problem.
	We shall prove that a martingale solution typically exists for given Borel measures $\Lambda_0$ and $\Lambda_{\f}$ that account for initial and forcing laws as follows,
	\begin{alignat}{2}
		\label{3.1}
		& \Lambda_0=\P\circ\uu_0^{-1},\quad \mbox{i.e., }\ \P\big(\uu_0\in \mathrm{U}\big) = \Lambda_0(\mathrm{U}),\qquad \forall\ \mathrm{U}\in \mathscr{B}(\bfV), \\
		\label{3.2}
		&  \Lambda_{\f}=\P\circ\f^{-1},\quad \mbox{i.e., }\ \P\big(\f\in \mathrm{U}\big) = \Lambda_{\f}(\mathrm{U}),\qquad \forall\ \mathrm{U}\in \mathscr{B}(\bfL^2(\1_T)).
	\end{alignat}
	It should be noted that even if the initial datum $\uu_0$ and the forcing term $\f$ are given, they might live on different probability spaces, and therefore $\uu_0$ and $\uu(0)$ from one hand, and $\f_t$ and $\f(t)$ on the other, can only coincide in law.
	In particular, the underlying probability space is not a priori known but becomes part of the solution.
	Next, we define martingale solutions for the stochastic problem (\ref{1.1})-(\ref{1.4}) starting with an initial law defined on $\bfV$ and a forcing law defined on $\bfL^2(\1_T)$.
	
	\begin{definition}\label{def.1}
		Let $\Lambda_0$, $\Lambda_{\f}$ be Borel probability measures on $\bfV$ and $\bfL^2(\1_T)$, respectively.
		We say that
		$$\big((\Omega,\mathscr{F},\{\mathscr{F}_t\}_{t\in [0,T]},\P),\uu,\uu_0,\f,\W\big)$$
		is a \emph{martingale solution} to the stochastic problem (\ref{1.1})-(\ref{1.4}), with initial datum $\Lambda_0$ and forcing term $\Lambda_{\f}$,
		if:
		\begin{enumerate}
			\item $\big(\Omega,\mathscr{F},\P)$ is a stochastic basis with a complete right-continuous filtration $\{\mathscr{F}_t\}_{t\in [0,T]}$;
			\item $\W$ is a cylindrical $\{\mathscr{F}_t\}_{t\in [0,T]}-$adapted Wiener process;
			\item $\uu$ is a progressively $\{\mathscr{F}_t\}_{t\in [0,T]}-$measurable stochastic process  with $\P-$a.s. paths
			$t\mapsto\uu(t,\omega)\in \L^\infty(0,T;\bfV)\cap \L^p(0,T;{\W_{0}^{1,p}(\1)^d}),$ with a continuous modification having paths in $\C([0,T];\bfV)$;
			\item $\uu(0)\ (:=\uu_0)$ is progressively $\{\mathscr{F}_t\}_{t\in [0,T]}-$measurable on the probability space \\$(\Omega,\mathscr{F},\P)$,  with $\P-$a.s. paths $\uu(0,\omega)\in \bfV$ and $\Lambda_0=\P\circ\uu_0^{-1}$ in the sense of (\ref{3.1});
			\item $\f$ is an $\{\mathscr{F}_t\}_{t\in [0,T]}-$adapted stochastic process $\P-$a.s. paths $\f(t,\omega)\in \bfL^2(\1_T)$ and
			$\Lambda_{\f}=\P\circ\f^{-1}$ in the sense of (\ref{3.2});
			\item for every $\bfi\in {\C_0^{\infty}(\1)^d}$ with $\operatorname{div}\bfi=0$ and all $t\in[0,T],$ the following identity holds $\P-$a.s.
			\begin{align}\label{3p3}\nonumber
				& \int_\1\uu(t)\cdot\bfi \,\d\x + \kappa\int_\1\nabla\uu(t):\nabla\bfi\,\d\x - \int_0^t\int_\1 {\uu\otimes\uu}:\mathds{\nabla}\bfi\,\d\x \d s \\&\nonumber\quad +
					\nu\int_0^t\int_\1{|\bfD(\uu)|^{p-2}\bfD(\uu)}:\bfD\bfi\,\d\x \d s
					\\&=
				\int_\1\uu_0\cdot\bfi\,\d\x + \kappa\int_\1\nabla\uu_0:\nabla\bfi\,\d\x +
					\int_0^t\int_\1{\f}\cdot\bfi\,\d\x\d s + \int_0^t\int_\1{\Phi(\uu)\d\W(s)}\cdot\bfi\,\d\x.
			\end{align}
		\end{enumerate}
	\end{definition}
	We are now interested  to know whether it is possible to find  a probabilistically strong solution for the problem (\ref{1.1})-(\ref{1.4}), for a given initial velocity $\uu_0$ and forcing  $\f$ (which are random variables rather than  probability laws) on a given probability space.
\begin{definition}[Probabilistically strong solution]\label{def1.2}
	We are given a stochastic basis \\ $(\Omega,\mathscr{F},\{\mathscr{F}_t\}_{t\geq0},\P)$, initial datum $\uu_0$ and a forcing term $\f$. Then, the problem \eqref{1.1}-\eqref{1.4}  has a pathwise \emph{strong probabilistic solution} if and only if there exists a $\uu:[0,T]\times \Omega\to \bfV$ with $\P-$a.s., paths
	\begin{align*}
		\uu(\cdot,\omega) \in \L^\infty(0,T;\bfV)\cap \L^p(0,T;{\W_0^{1,p}(\1)^d}),
	\end{align*} with a continuous modification having $\P-$a.s. paths in $\C([0,T];\bfV)$, and \eqref{3p3} holds for all $\bphi\in\bfV.$ 	
\end{definition}
		\begin{definition}[Pathwise uniqueness]\label{def1.3}
			For $i=1,2$, let  $\uu_i$ be any solution on the stochastic basis $(\Omega,\mathscr{F},\{\mathscr{F}_t\}_{t\geq0},\P)$  to the system \eqref{1.1}-\eqref{1.4} with initial datum $\uu_0$ and forcing term ${\f}$. Then, the solutions of the system \eqref{1.1}-\eqref{1.4} are \emph{pathwise unique} if and only if
			\begin{align*}
				\P\big\{\uu_1(t)=\uu_2(t),\ \forall \ t\geq 0\big\}=1.
			\end{align*}
	\end{definition}
	
	\subsection{Assumptions and main result}
	The existence of martingale solutions to the stochastic problem (\ref{1.1})-(\ref{1.4}) shall be carried out in the forthcoming pages under suitable assumptions on the initial datum $\uu_0$, forcing term $\f$ and on the noise coefficient $\Phi(\uu)$, and for a suitable range of the summability Lebesgue exponent $p$.	On the initial and forcing laws, we assume that for some constant $\gamma=\gamma(p,d)$ (to be determined further on)
	\begin{alignat}{2}
		& \displaystyle \int_\bfV\|\bz\|_{\bfV}^\gamma \d\Lambda_0(\bz)<\infty, &&  \label{3.4}\\
		& \displaystyle\int_{\bfL^{2}(\1_T)}\|\bfg\|_{\L^{2}(\1_T)}^\gamma \d\Lambda_{\f}(\bfg)<\infty. \label{3.5} &&
	\end{alignat}
	
	Let $\bfU$ be a Hilbert space with orthonormal basis $\{\bfe_k\}_{k\in\N}$ and let $\mathcal{L}_2(\bfU,\L^2(\1)^d)$ be the space of Hilbert-Schmidt operators form $\bfU$ to $\L^2(\1)^d$ (to simplify writing, in the sequel we denote the operator norm $\|\cdot\|_{\mathcal{L}_2(\bfU,\L^2(\1)^d)}$ solely by $\|\cdot\|_{\mathcal{L}_2}$). Moreover, we define an auxiliary space $\bfU_0\supset \bfU$ as
	\begin{align*}
		\bfU_0:=\bigg\{\vv =\sum_{k\in\N} \alpha_k \bfe_k: \sum_{k\in\N}\frac{\alpha_k^2}{k^2}<\infty\bigg\}
	\end{align*}equipped with the norm
\begin{align*}
	\|\vv\|_{\bfU_0}^2:=\sum_{k\in\N}\frac{\alpha_k^2}{k^2},  \ \ \vv=\sum_{k\in\N} \alpha_k\bfe_k.
\end{align*}Throughout the sequel, we consider a cylindrical $\{\mathscr{F}_t\}-$Wiener process $\W=\{\W(t)\}_{t\geq 0}$ which has the form
\begin{align*}
	\W(t)=\sum_{k\in\N}\bfe_k \beta_k(t),
\end{align*}where $\{\beta_k\}_{k\in\N}$ are independent real-valued Brownian motions.
The embedding $\bfU\hookrightarrow \bfU_0$ is Hilbert-Schmidt and trajectories of $\W$ are $\P$-a.s. continuous with values in $\bfU_0$ (see Subsection \ref{SS}).
	In this article, we suppose that the noise coefficient $\Phi(\uu)$ satisfies  linear growth and Lipschitz conditions.	More precisely, we assume that for each $\w\in{\L^2(\1)^d}$ there is a mapping
	\begin{align*}
			\Phi(\w): &\bfU\longrightarrow {\L^2(\1)^d}  \\
			& \bfe_k\longmapsto  \Phi(\w)\bfe_k=\phi_k(\w),
	\end{align*}
	where $\{\bfe_k\}_{k\in\N}$ is an orthonormal basis of $\bfU$, such that $\phi_k\in \C(\R^d)$ and the following conditions hold for some constants $K,L> 0$:
	\begin{align}\label{3.6a}
		\displaystyle \sum_{k\in\N}|\phi_k(\bxi)|\leq K(1+|\bxi|),\  \mbox{ and } \
		\displaystyle \sum_{k\in\N}| \phi_k(\bxi)-\phi_k(\boldsymbol{\zeta})|\leq L|\bxi-\boldsymbol{\zeta}|,\ \ {\bxi, \boldsymbol{\zeta}\in\R^d}.
	\end{align}
Moreover, we are assuming that the following condition holds for some constant $C>0$,
\begin{align}\label{3.6b}
	\sup_{k\in\N}k^2|\phi_k(\bxi)|^2\leq C(1+|\bxi|^2),\qquad {\bxi\in\R^d}.
\end{align}

	\begin{theorem}\label{thm:exist}
		Let $\1\subset\R^d$ be a bounded domain with a smooth boundary $\partial\1$ of class $\C^2$, and assume that conditions
	\eqref{3.4} and \eqref{3.5} hold for

 \begin{equation}\label{gamma}
\gamma\geq \max\bigg\{\frac{pd}{d-2},2+\frac{2p}{d-2}, \frac{2d}{d-2}\bigg\} \ (d\neq 2)\ \text{ and }\ \gamma\geq 2\ (d=2),
\end{equation}

 and \eqref{3.6a}, \eqref{3.6b}  are fulfilled.
		If $2\leq d\leq 4$ and
\begin{align}\label{3.7}
			p> \frac{2d}{d+2},
		\end{align}
		then there exists, at least, a \emph{martingale solution}
		$$\big((\overline{\Omega},\overline{\mathscr{F}},\{\overline{\mathscr{F}}_t\}_{t\in [0,T]},\overline{\P}),\overline{\uu},\overline{\uu}_0,\overline{\f},\overline{\W}\big)$$
	in the sense of Definition \ref{def.1}	to the stochastic problem \eqref{1.1}-\eqref{1.4}. 		
	\end{theorem}
	
	 \begin{theorem}\label{thm2.7}
		Under the assumptions of Theorem \ref{thm:exist}, there exists a unique \emph{probabilistically strong solution}  of the system \eqref{1.1}-\eqref{1.4} in the sense of Definition \ref{def1.2}.
	\end{theorem}
	The proof of the above theorems are presented in the subsequent sections.

	\section{Pressure decomposition} \label{sec4}\setcounter{equation}{0}
	In this section, our focus will be on the pressure term coming in our model and we decompose it in such a way that each part of the pressure term corresponds to one term in the equation.   The following theorem generalizes the idea of \cite[Theorem 2.6]{JW} to the stochastic case and a similar result has been established in  \cite[Section 3]{Breit} for the stochastic case. 

	\begin{theorem}\label{pr:th:1}
		Let us consider a stochastic basis $(\Omega,\mathscr{F},\{\mathscr{F}_t\}_{t\geq 0},\P)$, $\uu\in \L^2({\Omega,\mathscr{F},\P};\L^\infty(0,T;\bfV))$,  ${\bcH\in\bfL^r(\1_T)}$ for some {$1<r\leq 2$}, both adapted to $\{\mathscr{F}_t\}_{t\geq0}$. Moreover, if the initial data $\uu_0\in \L^2({\Omega,\mathscr{F},\P};\bfV)$ and $\Phi\in \L^2({\Omega,\mathscr{F},\P};\L^\infty(0,T;{\mathcal{L}_2(\bfU,\L^2(\1)^d)}))$ progressively measurable such that
		\begin{align}			
					&\int_\1\uu(t)\cdot\bphi\,\d\x+\kappa	\int_\1\nabla\uu(t):\nabla\bphi\,\d\x+ 	\int_0^t\int_\1 {\bcH}:\nabla \bphi\,\d\x\d s \nonumber \\
& =	\int_\1 \uu_0\cdot\bphi\,\d\x+\kappa	\int_\1\nabla\uu_0:\nabla\bphi\,\d\x+ 	\int_0^t\int_\1{\Phi(\uu)} \d\W(s)\cdot \bphi \,\d\x, \label{eq:0:pr:th:1}
				\end{align}
for all {$\bphi\in\bVcal$ and all $t\in[0,T]$}. Then there are functions $\pi_{\bcH}$, $\pi_\Phi$ and $\pi_h$ adapted to $\{\mathscr{F}_t\}_{t\geq 0}$ such that
		\begin{enumerate}
			\item We have $\Delta\pi_h=0$ and there holds for $m:=\min\{2,r\}$
			\begin{align*}
				& \E\bigg[\int_0^{T}\|\pi_{\bcH}(t)\|^r_{r}\d t\bigg] \leq C\E\bigg[\int_0^T\|{\bcH}(t)\|^r_{r}\d t\bigg] ,\\
				&	\E\bigg[\sup_{t\in[0,T]}\|\pi_\Phi(t)\|^2_{2}\bigg] \leq C\E\bigg[\sup_{t\in[0,T]}{\|\Phi(\uu(t))\|^2_{\mathcal{L}_2}}\bigg]  ,\\
				&	\E\bigg[\sup_{t\in[0,T]}\|\pi_h(t)\|^m_{m}\bigg] \leq C\E\bigg[1+\sup_{t\in[0,T]}\Big\{\|\uu(t)\|^2_{2}+\kappa\|\nabla\uu(t)\|^2_{2}\Big\} \\
&+\|\uu_0\|^2_{2}+\kappa\|\nabla\uu_0\|^2_{2} + {\sup_{t\in[0,T]}\|\Phi(\uu(t))\|^2_{\mathcal{L}_2}} +\int_0^T\|{\bcH}(t)\|^r_{r}\d t\bigg] .		
			\end{align*}
			\item There holds
			\begin{align*}
				&\int_\1\big(\uu(t) {+} \nabla \pi_h(t)\big)\cdot\bphi\,\d\x+\kappa\int_\1\nabla\uu(t):\nabla\bphi\,\d\x +\int_0^t\int_\1{\bcH}:\nabla \bphi\,\d\x\d s\\&\quad {-} \int_0^t\int_\1 \pi_{\bcH}\operatorname{div} \bphi\,\d\x\d s \\&= \int_\1 \uu_0\cdot\bphi\,\d\x+\kappa\int_\1\nabla\uu_0:\nabla\bphi\,\d\x {+} \int_\1\pi_\Phi(t)\operatorname{div} \bphi \,\d\x+\int_0^t\int_\1{\Phi(\uu)} \d\W(s)\cdot\bphi \,\d\x,
			\end{align*}
for all {$\bphi\in \C_{0}^\infty(\1)^d$}. Moreover $\pi_h(0)=\pi_{\bcH}(0)=\pi_\Phi(0)=0,\; \P-$a.s.
		\end{enumerate}
	\end{theorem}
	\begin{remark}
		If we place all the decomposed terms of pressure together, we have
		\begin{align*}
			\pi(t)=\pi_h(t)+\pi_\Phi(t)+\int_{0}^{t}\pi_{\bcH}(s) \d s,
		\end{align*}then there holds $\pi \in \L^m(\Omega,{\mathscr{F},\P};\L^\infty(0,T;\L^m(\1)))$.
	\end{remark}

Before proceeding with the proof of Theorem~\ref{pr:th:1}, let us recall some important auxiliary results.
The following lemma is a variant of the well-known de~Rham's theorem.

\begin{lemma}\label{lemm:BP}
Assume $1<m<\infty$.
\begin{enumerate}
  \item For each $\vv^\ast\in{\W^{-1,m}(\1)^d}$ such that
\begin{equation*}
\big< \vv^{\ast},\vv\big>=0, \quad\forall\ \vv\in
\mathbf{V}_{m'},
\end{equation*}
 there exists a unique $\pi\in \mathrm{L}^{m}(\1)$, with $\int_{\1}\pi\,\,\d\x=0$, such that
\begin{equation*}
\left\langle \vv^\ast,\vv\right\rangle =
\int_{\1}\pi \operatorname{div}\vv\,\,\d\x,\quad\forall\,\vv\in\W_{0}^{1,m'}(\1)^d.
\end{equation*}
Moreover, there exists a positive constant $C$ such that
\begin{equation*}
\|\pi\|_{m}\leq C \|\vv^\ast\|_{-1,m}.
\end{equation*}
\item For each
$\xi\in\mathrm{L}^{m}(\Omega)$ with $\int_{\Omega}\xi\,\,\d\x=0$, there exists, at least, one solution $\w\in\W^{1,m}_0(\1)^d$ to the problem
\begin{equation*}
\operatorname{div} \w=\xi\quad\mbox{in}\quad \1,
\end{equation*}
and such that for some positive constant $C$,
\begin{equation*}
\|\nabla\w\|_{m}\leq C\|\xi\|_{m}.
\end{equation*}
\end{enumerate}
\end{lemma}

\begin{proof}
The proof is due to Bogovski\v{\i}~\cite{MEB} (see also \cite[Theorems~III.3.1 and III.5.3]{Galdi:2011} and \cite[Proposition I.1.1]{Te}).
\end{proof}

	\begin{proof}
\textbf{Step 1:}
Define a bilinear form $\calS$ with the domain $\cD(\calS)\subseteq \L^2(\1)^d$ by setting $\cD(\calS)=\W_0^{1,2}(\1)^d$ and
\begin{equation*}
\calS(\uu,\vv)=\int_{\1}\uu\cdot\vv\,\,\d\x + \kappa \int_{\1}\nabla\uu:\nabla\vv\,\,\d\x.
\end{equation*}
Since $\W_0^{1,2}(\1)^d$ is complete with respect to the norm
\begin{equation*}
\calS(\uu,\uu)^{\frac{1}{2}}=\left(\|\uu\|_{2}^2+\kappa\|\nabla\uu\|^2_{2}\right)^{\frac{1}{2}},
\end{equation*}
the form $\calS$ is closed.
As $\kappa>0$, it is easy to see that $\calS$ is positive and symmetric.
Therefore by \cite[Lemma~II.3.2.1]{HS}, there exists a uniquely determined positive self-adjoint operator $\calA:\cD(\calA)\longrightarrow \L^2(\1)^d$, with dense domain
$\cD(\calA)\subseteq\W_0^{1,2}(\1)^d$, such that
\begin{equation*}
\int_{\1}\calA(\uu)\cdot\vv\,\,\d\x=\calS(\uu,\vv)=\int_{\1}\uu\cdot\vv\,\,\d\x + \kappa \int_{\1}\nabla\uu:\nabla\vv\,\,\d\x.
\end{equation*}
Setting $\vv\in \C_0^\infty(\1)^d$, we can see that $\calA(\uu)=(\bfI-\kappa\Delta)(\uu)=(\bfI-\kappa\operatorname{div}\nabla)(\uu)$ holds in the distribution sense, and so we can set
$\calA=\bfI-\kappa\Delta$, where $\bfI$ denotes the identity operator.
In this way, the operator $\bfI-\kappa\Delta: \cD(\bfI-\kappa\Delta)\longrightarrow \L^2(\1)^d$ is defined by
\begin{equation}\label{w:I-kD}
\int_{\1}(\bfI-\kappa\Delta)(\uu)\cdot\vv\,\,\d\x=\calS(\uu,\vv)=\int_{\1}\uu\cdot\vv\,\,\d\x + \kappa \int_{\1}\nabla\uu:\nabla\vv\,\,\d\x,
\end{equation}
for $\uu\in\cD(\bfI-\kappa\Delta)$, $\vv\in \W_0^{1,2}(\1)^d$, and
\begin{equation*}
\cD(\bfI-\kappa\Delta)=\left\{\uu\in \W_0^{1,2}(\1)^d: \vv\longmapsto \calS(\uu,\vv)\ \mbox{is continuous in the $\L^2$-norm}\right\}.
\end{equation*}
Given the above, and observing that $m\leq \min\{r,2\}$, we can write \eqref{eq:0:pr:th:1} in the following form:
		\begin{align}			
					& \int_\1(\bfI-\kappa\Delta)(\uu(t))\cdot\bphi\,\d\x + 	\int_0^t\int_\1 {\bcH}:\nabla \bphi\,\d\x\d s \nonumber \\
& =	\int_\1 (\bfI-\kappa\Delta)(\uu_0)\cdot\bphi\,\d\x + \int_0^t\int_\1\Phi(\uu) \d\W(s)\cdot \bphi \,\d\x, \label{eq:01:pr:th:1}
				\end{align}
for all {$\bphi\in \W_0^{1,m'}(\1)^d$ with $\operatorname{div}\bphi=0$ in $\1$, and all $t\in[0,T]$}.

\vspace{2mm}
\noindent
\textbf{Step 2:}
Using \eqref{eq:01:pr:th:1} and Lemma~\ref{lemm:BP}-(1), we can argue as in the proof of \cite[Theorem~2.6]{JW} to prove the existence of a unique pressure $\pi\in \C_{\rw}([0,T];\L^m(\1))$, with
\begin{equation}\label{0:av:pr}
\int_{\1}\pi(t)\,\d\x=0,\qquad \forall\ t\in[0,T],
\end{equation}
such that

		\begin{align}\label{PM}\nonumber
				&	\int_\1\uu(t)\cdot{(\bfI-\kappa\Delta)(\bphi)}\,\d\x +\int_0^t\int_\1 {\bcH}:\nabla \bphi \,\d\x\d s \\
&=\int_\1 \uu_0\cdot{(\bfI-\kappa\Delta)(\bphi)}\,\d\x {+} \int_\1 \pi(t)\operatorname{div} \bphi \,\d\x +\int_0^t\int_\1 {\Phi(\uu)} \d\W(s)\cdot\bphi \,\d\x,
				\end{align}
for all $\bphi\in \W_0^{1,m'}(\1)^d$ (note that $m=\min\{2,r\}$), and all $t\in[0,T]$.

Now, we will prove that
		\begin{align}\label{4.1}
			\pi\in \L^m(\Omega,{\mathscr{F},\P};\L^\infty(0,T;\L^m(\1))).
		\end{align}
From \eqref{PM}, we can say that $\pi$ is a progressively measurable process, since all other terms in this equation are.
{By Lemma~\ref{lemm:BP}-(2), with $m'$ in the place of $m$, given $\psi\in\L^{m'}(\1),$ there exists $\bphi\in \W_0^{1,m'}(\1)^d$ such that
\begin{equation}\label{eq:Bog}
\operatorname{div}\bphi = \psi -\psi_{\1},\qquad \psi_{\1}:=\frac{1}{|\1|}\int_{\1}\psi\,\d\x.
\end{equation}
Combining \eqref{0:av:pr} and \eqref{eq:Bog} with \eqref{PM}, one has}
		\begin{align*}
			\int_\1\pi(t)\psi\,\d\x &= \int_\1\pi(t)(\psi-\psi_{\1})\,\d\x \\
&=\int_\1(\uu(t)-\uu_0)\cdot {(\bfI-\kappa\Delta)\cB(\psi)}\,\d\x {+}\int_0^t\int_\1 {\bcH}:\nabla  {\cB(\psi)}\,\d\x \d s\\
&\quad {-} \int_0^t\int_\1{\Phi(\uu)}\d\W(s)\cdot {\cB(\psi)}\,\d\x,
				\end{align*}
for all $\psi\in\L^{m'}(\1)$, and
where $\cB:\L^{m'}(\1)\longrightarrow  \W_0^{1,m'}(\1)^d$ is the Bogovski\u{\i} operator (see \cite[Appendix 4.2]{MP}).
Denoting by $\cB^\ast:\W_0^{1,m'}(\1)^d\longrightarrow \L^{m'}(\1),$ the Bogovski\u{\i} adjoint operator, this implies
		\begin{align}\label{ref:Pi(t)}
			\pi(t)=\big((\bfI-\kappa\Delta)\cB\big)^\ast(\uu(t)-\uu_0)+\int_{0}^{t}(\nabla\cB)^*{\bcH} \d s-\int_{0}^{t}\cB^*\Phi(\uu) \d\W(s),
		\end{align}
with respect to the $\L^2(\1)$--inner product.
From \eqref{ref:Pi(t)}, one immediately has $\pi(0)=0$.
Averaging \eqref{ref:Pi(t)} in the $\L^m(\1)$--norm, we get
\begin{align}\label{4.2}\nonumber
&	
\E\bigg[\sup_{t\in[0,T]}\|\pi(t)\|^m_{m}\bigg]  \\&=
\nonumber \E \bigg[\sup_{t\in[0,T]}
\bigg\|\big((\bfI-\kappa\Delta)\cB\big)^\ast(\uu(t)-\uu_0)-\int_{0}^{t}(\nabla\cB)^*{\bcH}(\uu(s))\, \d s +
\int_{0}^{t}\cB^*\Phi (\uu(s))\d\W(s)\bigg\|_{m}^m\bigg]  \\&\nonumber\leq

 C\bigg( 1+
\E \bigg[\sup_{t\in[0,T]}\|\big((\bfI-\kappa\Delta)\cB\big)^\ast(\uu(t)-\uu_0)\|_{2}^2\bigg] + \E \bigg[\int_0^T\|(\nabla\cB)^*{\bcH}(\uu(t))\|_{m}^m \d t\bigg] \\
&\qquad + \E \bigg[\sup_{t\in[0,T]}\bigg\|\int_{0}^{t}\cB^*\Phi(\uu(s)) \d\W(s)\bigg\|_{2}^2\bigg]\bigg)=:C\bigg(1+ \sum_{i=1}^{3} I_i\bigg),
\end{align}
where we have used the Minkowski, Hölder and Young inequalities, together with the fact that $m\leq 2$.
Using the continuity of the operator $\cB^\ast$ from $\L^2(\1)^d$ to $\L^2(\1)$, observing first that $(\bfI-\kappa\Delta)$ is self-adjoint and therefore $\big((\bfI-\kappa\Delta)\cB\big)^\ast=\cB^\ast(\bfI-\kappa\Delta)$, and still using \eqref{w:I-kD} and Minkowski's inequality, we have
		\begin{align*}
	I_1& \leq C\E \bigg[\sup_{t\in[0,T]} \|(\bfI-\kappa\Delta)(\uu(t)-\uu_0)\|_{2}^2\bigg] \\
& = C \E\bigg[\sup_{t\in[0,T]} \|\uu(t)-\uu_0\|_{2}^2+\kappa\|\nabla(\uu(t)-\uu_0)\|^2_{2}\bigg] \\
& \leq  C \E\bigg[\sup_{t\in[0,T]}\Big\{\|\uu(t)\|_{2}^2 + \kappa\sup_{t\in[0,T]}\|\nabla\uu(t)\|^2_{2} \Big\}  +\|\uu_0\|_{2}^2  + \kappa\|\nabla\uu_0\|^2_{2}\bigg].
		\end{align*}
Since $\bcH\in \bfL^r(\1_T)$, $m=\min\{2,r\}$ and $1<r\leq 2$, and still using the fact that  $(\nabla\cB)^*$ is a continuous operator from $\L^{r}(\1)^{d\times d}$ to $\L^{r}(\1)$,  a similar calculation yields
\begin{align*}
	I_2 \leq C \E\bigg[1+\int_0^T\|{\bcH}(t)\|_{r}^r \d t \bigg].
\end{align*}
The  stochastic integral term  $I_3$ appearing in \eqref{4.2} can be handled with the help of Burkholder-Davis-Gundy (BDG) inequality, {using also here the continuity of
$\cB^\ast$ from $\L^2(\1)^d$ to $\L^2(\1)$},
\begin{align*}
	I_3\leq  C\E\bigg[\int_0^T\|\Phi(\uu(t))\|^2_{\mathcal{L}_2}\d t\bigg].
\end{align*}
Combining the above estimates in \eqref{4.2}, we obtain
	\begin{align}\label{4.002}\nonumber
 \E\bigg[\sup_{t\in[0,T]}\|\pi(t)\|_{m}^m\bigg] &\leq  \E\bigg[1+\sup_{t\in[0,T]}\Big\{\|\uu(t)\|^2_{2}+\kappa\|\nabla\uu(t)\|^2_{2}\Big\}+\|\uu_0\|_{2}^2+\kappa\|\nabla\uu_0\|^2_{2} \\
&\qquad  +\int_0^T\|{\bcH}(t)\|_{r}^r\d t+\int_0^T\|\Phi(\uu(t))\|^2_{\mathcal{L}_2}\d t \bigg].
\end{align}
On the other hand, by \cite[Corollary~2.5]{JW}, there exist unique functions
\begin{alignat*}{2}
& \pi_0\in\C_{\rw}([0,T];\Delta\W_0^{2,m}(\1)),\qquad \Delta\W_0^{2,m}(\1):=\left\{\Delta v: v\in \W_0^{2,m}(\1)\right\}, && \\
& \pi_h\in \C_{\rw}([0,T];\L^m_{\Delta}(\1)),\qquad \L^m_{\Delta}(\1):=\left\{v\in \L^m(\1): \Delta v=0,\ \ \int_{\1}v\,\d \x=0\right\}, &&
\end{alignat*}
such that
\begin{alignat}{2}
& \pi:=\pi_0+\pi_h\qquad \mbox{in}\ \ \1_T, \nonumber \\
& \sup_{t\in[0,T]}\|\pi_0(t)\|_{m}+ \sup_{t\in[0,T]}\|\pi_h(t)\|_{m} \leq C\sup_{t\in[0,T]}\|\pi(t)\|_{m}, \label{est:po:ph}
\end{alignat}
for some positive constant $C$. Therefore, we can decompose the pressure term pointwise  on $\1_T$ as follows:
		\begin{align*}
			\pi:=\pi_0+\pi_h,\qquad
			\pi_0:=\Delta\Delta_\1^{-2}\Delta\pi,\qquad \pi_h:=\pi-\pi_0,
		\end{align*}
where $\Delta_\1^{-2}$
 denotes the solution operator to the bi-Laplace equation with respect to zero boundary values for the function and gradient.
{As a consequence of \eqref{est:po:ph}, inequality \eqref{4.002} holds true for $\pi_0$ and $\pi_h$ as well.
Into \eqref{PM}, replacing $\pi$ by $\pi_0+\pi_h$, inserting $\bphi=\nabla\psi$, $\psi\in\W_0^{2,m'}(\1)$, and observing that $\operatorname{div}\uu =0$, $\operatorname{div}\uu_0 =0$ and $\Delta\pi_h=0$, one has
		\begin{align}\label{4.3}
		\int_\1\pi_0(t)\Delta \psi\,\d\x  = \int_0^t\int_\1 {\bcH}:\nabla^2\psi\,\d\x\d s-\int_0^t\int_\1\Phi(\uu) \d\W(s)\cdot\nabla\psi\,\d\x,
		\end{align}
for all $\psi\in\W_0^{2,m'}(\1)$. From \eqref{4.3}, it is clear that $\pi_0(0)=0$. And once that $\pi(0)=0$, we also have $\pi_h(0)=0$. As $\operatorname{div}^2 {\bcH}=\operatorname{div}(\operatorname{div} {\bcH})\in \W^{-2,m}(\1)$, by \cite[Lemma 2.4]{JW},
there exists a function $\pi_{\bcH}(t)\in \Delta \W_0^{2,m}(\1)$ such that
		\begin{align*}
			\int_\1\pi_{\bcH}(t)\Delta\psi \,\d\x =\int_\1{\bcH}(t):\nabla^2\psi\,\d\x,
		\end{align*}
for all $\psi\in \C_0^\infty(\1)$.
The measurability of $\pi_{\bcH}$} follows from the measurability of the right hand side.
{On account of the solvability of the bi-Laplace equation, one has
		\begin{align*}
			\|\pi_{\bcH}(t)\|_{m}^m \leq C \|{\bcH}(t)\|_{m}^m \qquad \P\otimes\lambda-\text{a.e.}
		\end{align*}
for some positive constant $C$, where $\lambda$ is the Lebesgue measure on $(0,T)$.
For more details, see \cite{RM}, or  \cite[Lemma 2.1]{JW} and \cite[Lemma 2.2]{DBLDSS}.
This in turn gives
		\begin{align*}
			\int_{\Omega\times\1_T}\|\pi_{\bcH}(t)\|_{m}^m\d t\d\P \leq C\int_{\Omega\times\1_T}\|{\bcH}(t)\|_{m}^m\d t\d\P.
		\end{align*}
Next, we define
$$\pi_\Phi(t):=\pi_0(t) + \int_{0}^{t}\pi_{\bcH}(s) \d s,$$
as the unique solution to the following integral equation
		\begin{align}\label{4.4}
		\int_\1 \pi_\Phi(t)\Delta\psi\,\d\x=\int_0^t\int_\1\Phi(\uu(s)) \d\W(s)\cdot\nabla\psi \,\d\x,\qquad \ \psi\in \C_0^\infty(\1).
		\end{align}
Since $\pi_\Phi(t)\in \Delta \W_0^{2,m}(\1)$, \eqref{4.4} can be written as follows
		\begin{align}\label{ref:A:1}
		\int_\1\pi_\Phi(t)\psi\,\d\x=\int_0^t\int_\1\Phi(\uu) \d\W(s)\cdot\nabla(\Delta^{-2}_{\mathcal{O}}\Delta\psi)\,\d\x,\qquad \psi\in \C_0^\infty(\1).
		\end{align}
From here, one immediately has $\pi_\Phi(0)=0$. Defining the operator $\bD:=\nabla\Delta_\1^{-2}\Delta:\W_0^{2,m}(\1)\longrightarrow\W_0^{1,m}(\1)^d$, we can write \eqref{4.4} as
		\begin{align*}
		\int_\1\pi_\Phi(t)\psi\,\d\x=\int_0^t\int_\1\bD^*\Phi(\uu) \d\W(s)\psi\,\d\x,\qquad \psi\in \C_0^\infty(\1),
		\end{align*}
and hence
$$\pi_\Phi(t)=\int_{0}^{t}\bD^*\Phi(\uu(s)) \d\W(s),\qquad \mbox{$\P\times \lambda^{d+1}$-a.e.},$$ where $\lambda^{d+1}$ is the Lebesgue measure on $\mathcal{O}\times(0,T)$.
Using the BDG inequality, together with the fact that $\bD^*$ is a continuous operator from $\L^2(\1)^d$ to $\L^2(\1)$, we can show that
		\begin{align*}
\E\bigg[\sup_{t\in[0,T]}\|\pi_\Phi(t)\|_{2}^2\bigg]& = \E\bigg[\sup_{t\in[0,T]}\bigg\|\int_0^t\bD^*\Phi(\uu(s))\d\W( s)\bigg\|_{2}^2\bigg] \\
& \leq
C\E\bigg[\int_0^T\|\Phi(\uu(t))\|_{\mathcal{L}_2}^2\d t\bigg].
\end{align*}
Finally, we can see that
$$\bar{\pi}_0(t):=\pi_\Phi(t)+\int_{0}^{t}\pi_{\bcH}(s) \d s$$
solves \eqref{4.3}, and $\bar{\pi}_0(t)\in \Delta \W_0^{2,m}(\1)$.
This implies}
		\begin{align*}
			\pi_0(t)=\pi_\Phi(t)+\int_{0}^{t}\pi_{\bcH} (s)\,\d s,
		\end{align*}
		and hence we obtained the required result.
	\end{proof}

For the following result, it is worth keeping in mind that we are assuming a smooth boundary $\partial\1$, in fact of class $\C^2$. See \cite[Section 2.2]{PhCGS1} and \cite[Section 2.2]{PhCGS2} in the case of non-smooth boundary $\partial\1$.

	\begin{corollary}\label{p:dc:cor:1}
Assume the conditions of Theorem~\ref{pr:th:1} are satisfied.
Then there exists $\Phi_\pi\in \L^2(\Omega,\mathscr{F},\P;\L^\infty(0,T;\mathcal{L}_2(\bfU,\L^2(\1))))$ progressively measurable such that
		\begin{align}\label{p:dc:cor:1:eq}
		\int_\1\pi_\Phi(t)\operatorname{div} \bphi \,\d\x =	\int_0^t\int_\1\Phi_\pi \d\W(s)\cdot\bphi \,\d\x, \quad  \forall\ \bphi \in \C_0^\infty(\1).
		\end{align}
and $\|\Phi_\pi \bfe_j\|_{2}\leq C({\1})\|\Phi\bfe_j\|_{2}$, for all $j$, that is, we have $\P\otimes \lambda$, a.e.
		\begin{align*}
			\|\Phi_\pi\|_{\mathcal{L}_2} \leq C({\1}) \|\Phi\|_{\mathcal{L}_2}.
		\end{align*}
Furthermore, if $\Phi$ satisfies \eqref{3.6a}, then there holds
		\begin{align*}
			\|\Phi_\pi (\uu_1)-\Phi_\pi (\uu_2)\|_{\mathcal{L}_2}  \leq C(L,{\1}) \|\uu_1-\uu_2\|_{{2}},
		\end{align*}for all $\uu_1,\uu_2\in \L^2(\1)^d$.
	\end{corollary}
	\begin{proof}
		From {\eqref{ref:A:1} of the proof of Theorem \ref{pr:th:1}, we have for any $\bphi \in {\C_0^\infty(\1)^d}$}
		\begin{align*}
				\int_\1\pi_\Phi(t)\operatorname{div} \bphi \,\d\x&= 	\int_0^t\int_\1\Phi (\uu)\d\W(s)\cdot\nabla (\Delta^{-2}_{\mathcal{O}}\Delta\operatorname{div} \bphi)\,\d\x \\& =\sum_j 	\int_0^t\int_\1\Phi(\uu)\bfe_j \d\beta_j(s)\cdot \nabla (\Delta^{-2}_{\mathcal{O}}\Delta\operatorname{div} \bphi)\,\d\x \\& =\sum_j	\int_0^t\int_\1 \nabla \Delta\Delta^{-2}_{\mathcal{O}}\operatorname{div} \Phi(\uu)\bfe_j \d\beta_j(s)\cdot \bphi\,\d\x\\& =	\int_0^t\int_\1 \nabla \Delta\Delta^{-2}_{\mathcal{O}}\operatorname{div} \Phi (\uu)\d\W(s)\cdot \bphi \,\d\x,
			\end{align*}
{which proves \eqref{p:dc:cor:1:eq}} by setting $\Phi_\pi=\nabla \Delta\Delta^{-2}_{\mathcal{O}}\operatorname{div} \Phi$.
For  {the} second part, we use {the} local regularity theory for the bi-Laplace equation in the  following manner:
		\begin{align*}
		\big|(\Phi_\pi\bfe_k,\bpsi)\big| = &
 \big| (\nabla\Delta\Delta^{-2}_{\mathcal{O}}\operatorname{div} \Phi \bfe_k,\bpsi)\big|\leq  \|\nabla\Delta\Delta^{-2}_{\mathcal{O}}\operatorname{div} \Phi \bfe_k\|_{{2}}\|\bpsi\|_{{2}} \\
 \leq &
 C({\1}) \|\Delta\Delta^{-2}_{\mathcal{O}}\operatorname{div} \Phi \bfe_k\|_{-1,2}\|\bpsi\|_{{2}}\leq C({\1}) \|\Delta^{-2}_{\mathcal{O}}\operatorname{div} \Phi \bfe_k\|_{3,2}\|\bpsi\|_{{2}} \\
\leq &  C({\1}) \|\operatorname{div} \Phi \bfe_k\|_{-1,2}\|\bpsi\|_{{2}}\leq C({\1}) \|\Phi \bfe_k\|_{{2}}\|\bpsi\|_{{2}}
\leq C({\1}) \|\Phi \bfe_k\|_{{2}}\|\bpsi\|_{{2}},
		\end{align*}
for $\bpsi\in \L^2(\1)^d$. Similarly, we can conclude the final part as follows:
	\begin{align*}
		\|(\Phi_\pi(\uu_1)-\Phi_\pi(\uu_2))\bfe_k\|_{{2}} & =	\|\nabla\Delta\Delta^{-2}_{\mathcal{O}}\operatorname{div} (\Phi(\uu_1)-\Phi(\uu_2))\bfe_k\|_{{2}} \\
& \leq C({\1})\|(\Phi(\uu_1)-\Phi(\uu_2))\bfe_k\|_{{2}} = C({\1})\|\phi_k(\uu_1)-\phi_k(\uu_2)\|_{{2}} \\
& \leq C({L,\1})\|\uu_1-\uu_2\|_{{2}},
	\end{align*}where we have used \eqref{3.6a}.
	\end{proof}

	\begin{corollary}\label{p:dc:cor:2}
		Let the {conditions of Theorem \ref{pr:th:1} be satisfied}. Then, for all $\gamma\in[1,\infty)$
		\begin{align*}
 \E\bigg[\sup_{t\in[0,T]}\|\pi_h(t)\|_{{m}}^m\bigg]^{\gamma}
&\leq C\E\bigg[\sup_{t\in[0,T]}\big\{\|\uu(t)\|^2_{2}+\kappa{\|\nabla\uu(t)\|^2_{2}} \big\}+\sup_{t\in[0,T]}\|\Phi(t)\|^2_{\mathcal{L}_2}\bigg]^\gamma \\
&\quad  +C\E\bigg[1+\|\uu_0\|^2_{2}+\kappa{\|\nabla\uu_0\|^2_{2}}+\int_0^T{\|\bcH(t)\|^r_{r}}\d t\bigg]^\gamma,
			\end{align*}provided the right hand side of above inequality is finite.
	\end{corollary}
\begin{proof}
	The proof of this corollary is {straightforward} from Theorem \ref{pr:th:1}-(1).
\end{proof}


	\begin{corollary}\label{p:dc:cor:3}
	Let  the {conditions of Theorem \ref{pr:th:1} be satisfied, and assume that the following decomposition holds:}
		\begin{align*}
			{\bcH}={\bcH}_1+{\bcH}_2,
		\end{align*}
where
$\bcH_1 \in \L^{r_1}(\Omega,\mathscr{F},\P;\L^{r_1}(0,T;{\L^{r_1}(\1)^{d\times d}}))$,  ${\bcH}_2 \in \L^{r_2}(\Omega,\mathscr{F},\P;\L^{r_2} (0,T;{\L^{r_2}(\1)^{d\times d}}))$ and $\operatorname{div} {\bcH}_2\in \L^{r_2}(\Omega,\mathscr{F},\P;\L^{r_2}(0,T;{\L^{r_2}(\1)^d}))$. Then, we have
		\begin{align*}
			\pi_{\bcH} =\pi_1+\pi_2,
		\end{align*}
and {there holds for all $\gamma\in[1,\infty)$,}
		\begin{align*}
				\E\bigg[\int_0^T\|\pi_1(t)\|^{r_1}_{r_1}\d t\bigg]^\gamma &\leq C\E\bigg[\int_0^T{\|\bcH_1(t)\|^{r_1}_{r_1}}\d t\bigg]^\gamma,\\
				\E\bigg[\int_0^T\|\pi_2(t)\|^{r_2}_{r_2}\d t\bigg]^\gamma & \leq C\E\bigg[\int_0^T{\|\bcH_2(t)\|^{r_2}_{r_2}}\d t\bigg]^\gamma, \\
				\E\bigg[\int_{0}^{T}\|\nabla \pi_2(t)\|_{r_2}^{r_2}\d t\bigg]^\gamma &\leq C\E\bigg[\int_0^T\big\{\|{\bcH}_2(t)\|^{r_2}_{r_2}+\|\operatorname{div} {\bcH}_2(t)\|^{r_2}_{r_2}\big\}\d t \bigg]^\gamma.
			\end{align*}
	\end{corollary}
	\begin{proof}
		Here $\pi_1$ and $\pi_2$ are the unique solutions of the following equations defined on $\P\otimes \lambda$, a.e.:
		\begin{align*}
				\int_\1 \pi_1(t)\Delta\psi \,\d\x &=-	\int_\1 {\bcH}_1(t):\nabla^2\psi\,\d\x,\qquad {\pi_1\in\Delta\W_0^{2,r_1}(\1)}, \\
				\int_\1 \pi_2(t)\Delta\psi \,\d\x &=-	\int_\1{\bcH}_2(t):\nabla^2\psi \,\d\x,\qquad {\pi_2\in\Delta\W_0^{2,r_2}(\1)}.
				\end{align*}
By \cite[Lemma 2.3]{JW}, and once that $\bcH_1 \in \L^{r_1}(\Omega,\mathscr{F},\P;\L^{r_1}(0,T;{\L^{r_1}(\1)^{d\times d}}))$,
${\bcH}_2 \in \L^{r_2}(\Omega,\mathscr{F},\break\P;\L^{r_2}(0,T;{\L^{r_2}(\1)^{d\times d}}))$ and
$\operatorname{div} {\bcH}_2\in \L^{r_2}(\Omega,\mathscr{F},\P;\L^{r_2}(0,T;{\L^{r_2}(\1)^{d}}))$, the estimates are straightforward.
	\end{proof}
	
	\section{Solvability of The Approximate System}\label{sec5}\setcounter{equation}{0}
	In this section, we establish the solvability of an auxiliary problem that regularizes the problem \eqref{1.1}-\eqref{1.4} with the following stabilization term in the momentum equation:
\begin{equation}\label{ref:a(u)}
\alpha\bfa(\uu),\qquad \bfa(\uu):=|\uu|^{q-2}\uu,\qquad \alpha>0,\qquad 1<q<\infty.
\end{equation}
Given $\alpha>0$, we consider the problem
	\begin{equation}\label{AS1}
		\left\{
			\begin{aligned}
				& \d(\I-\kappa\Delta) \uu  =\big\{\operatorname{div}  \bfA(\uu)-\operatorname{div}  (\uu\otimes\uu)+\nabla\pi -\alpha\bfa(\uu)
				+\f \big\}\d t+\Phi(\uu)\d \W(t),\\
				& \uu(0) =\uu_0,
			\end{aligned}
		\right.
	\end{equation}
depending on the initial $\Lambda_0$ and forcing $\Lambda_{\f},$ laws in the conditions of \eqref{3.4} and \eqref{3.5}, respectively, and for
the operator $\bfA(\uu)$ defined in \eqref{op:A(u)}.
The exponent $q$ in \eqref{ref:a(u)} is chosen in such a way that the convective term becomes a compact perturbation (see \eqref{5.012} below).
For that purpose, we choose
\begin{equation}\label{hyp:q}
q\geq \max\{2p',3\},
\end{equation}
and thus a solution $\uu$ is expected in the following space:
	\begin{align*}
		\bVcal_{p,q}:=\L^2({\Omega,\mathscr{F},\P};\L^\infty(0,T;\bfV))\cap \L^p({\Omega,\mathscr{F},\P};\L^p(0,T;{\W_{0}^{1,p}(\1)^d})\cap \L^q({\Omega,\mathscr{F},\P};\L^q(0,T;{\L^q(\1)^d})).
		\end{align*}
		\subsection{Approximate solutions}\label{Subs:GA}
		We construct a solution to the problem {(\ref{AS1})} as a limit of suitable Galerkin approximations.
		Let $s$ be the smallest positive integer such that ${\W^{s,2}_0(\1)^d}\hookrightarrow {\W^{1,\infty}_0(\1)}^d$, for $s>1+\frac{d}{2}$, and let us consider the space $\bfV_s$ associated to ${\W^{s,2}_0(\1)^d}$ defined in Section \ref{sec2}. 
		By means of separability, there exists a basis $\big\{\mathds{\bpsi}_{k}\big\}_{k\in\N}$ of $\bfV_s$, formed by the eigenfunctions of a suitable spectral problem, that is, orthogonal in ${\L^2(\1)^d}$ and that can be made orthonormal in ${\W^{s,2}_0(\1)^d}$ (see~\cite[Theorem~A.4.11]{MNRR:1996}).
		Given
		$n\in\N$, let us consider the $n-$dimensional space $\bfX^{n}=\mathrm{span}\{\bpsi_{1}, \dots, \bpsi_{n}\}$.
		For each $n\in\N$, we search for approximate solutions of the form
		\begin{align}\label{5.1}
			\uu_n(x,t)= \sum_{k=1}^nc_k^n(t)\bpsi_{k}(x),\quad\bpsi_{k}\in \bfX^n,
		\end{align}
		where the coefficients $c_{1}^n(t),\dots,c_{n}^n(t)$
		are solutions of the following $n$ stochastic ordinary differential equations
		\begin{align}\nonumber
			\label{5.2}%
				&
				\d\Big[\big(\uu_n(t),\bpsi_{k}\big)+\kappa\big(\nabla\uu_n(t),\nabla\bpsi_{k}\big)\Big]
				\\
				&\nonumber=
				\Big[\big(\uu_n(t)\otimes\uu_n(t):\nabla\bpsi_{k}\big)-\nu\big<|\bfD(\uu_n(t))|^{p-2}\bfD(\uu_n(t)):\bfD(\bpsi_{k})\big>
				+
				\big(\f(t),\bpsi_{k}\big) \Big]\d t  \\&\qquad-\alpha \big(\bfa(\uu_n(t)),\bpsi_l\big)+{\Phi(\uu_n(t))}\d\W_n(t),\bpsi_{k}\big), \qquad k=1,\dots,n,
			\end{align}
		supplemented with the initial conditions
		\begin{align}\label{5.3}
			\uu_n(0)=\uu_0^n,\quad\mbox{in}\ \1,
		\end{align}
		where $\uu_{0}^{n}=P^n(\uu _{0})$, with $P^n$ denoting the orthogonal projection $P^n:\bfV\longrightarrow \bfX^n$ so that
		\begin{align*}
			\uu_n(0,x)=\sum_{k=1}^n c_k^n(0)\bpsi_k(x),\quad c_k^n(0)=c_{k,0}^n:=(\uu _{0},\bpsi_k),\quad k=1,\dots,n.
		\end{align*}
		In (\ref{5.2}), we assume that the approximate cylindrical Wiener process $\big\{\W_n(\cdot)\big\}$ has the form
		$$\W_n(t)=\sum_{k=1}^{n}\bfe_k\beta_k(t),$$
		and note that (\ref{5.2}) is to be understood $\P-$a.s., and for all $t\in(0,T]$.
		Observe that the stochastic system (\ref{5.2})-(\ref{5.3}) can be written in the matrix form as follows,
		\begin{align}\label{5.4}
			\bfC \d\bfc(t)=\bfb(t)\d t + \bfG(t)\d\bbeta(t), \quad \bfc(0)=\bfc_{n,0},
		\end{align}
		where
		$\bfC=\left\{a^n_{lm}\right\}_{l,m=1}^n$, $\bfb(t)=\{b_{l}^n(t)\}_{l=1}^n$, $\bfc(t)=\{c_m^n(t)\}_{m=1}^n$,
		$\bfG(t)=\left\{g^n_{lm}(t)\right\}_{l,m=1}^n$ and $\bbeta(t)=\{\beta_m^n(t)\}_{m=1}^n$,  with
		\begin{alignat*}{5}
			 a_{lm}^n&:=
			\left(\bpsi_{l},\bpsi_{m}\right)+\kappa\left(\nabla\bpsi_{l},\nabla\bpsi_{m}\right), \\
			 b_{l}^n(t)&:=
			\big(\f(t),\bpsi_{l}\big) + \big(\uu^n_n(t)\otimes\uu^n_n(t):\nabla\bpsi_{l}\big)
			-\nu\big<|\bfD(\uu_n(t))|^{p-2}\bfD(\uu_n(t)):\bfD(\bpsi_{l})\big>\\&\quad -\alpha \big(\bfa(\uu_n(t)),\bpsi_l\big),\\
			 \bfc(t)&:=(c_{1}^n(t),\dots,c_{n}^n(t))\qquad\mbox{and}\qquad
			\bfc_{0}={(c_{1,0}^n,\dots,c_{n,0}^n)}, \\
			 g^n_{lm}(t)&:=\big(\Phi(\uu_n(t))\bfe_l,\bpsi_m\big), \\
			 \bbeta(t)&:=(\beta_{1}^n(t),\dots,\beta_{n}^n(t)).
		\end{alignat*}
Taking into account that the family $\big\{\bpsi_{k}\big\}_{k\in\N}$ is linearly independent in $\bfV$,
		$\big(\bfC\bxi,\bxi\big)>0$ for all $\bxi\in\R^n\setminus\{0\}$, we can write (\ref{5.4}) in the form
		\begin{align}\label{5.5}
			\d\bfc(t)=\overline{\bfb(t)}\d t + \overline{\bfG(t)}\d\bbeta(t), \quad \bfc(0)=\bfc_{n,0},
		\end{align}
		where $\overline{\bfb(t)}=\bfC^{-1}\bfb(t)$ and $\overline{\bfG(t)}=\bfC^{-1}\bfG(t)$.
		
		If both the coefficients $\bfb(\cdot)$ and $\bfG(\cdot)$ are globally {Lipschitz-continuous}, one can use the classical existence results for stochastic differential equations. {Our case does not fall in that category, so we have to use the monotonicity method and verify the conditions of the} vectorial version of Lemma \ref{prop:e:sde} as follows: 
		\begin{align*}
				&\langle \bfb(\cdot,\uu_n)-\bfb(\cdot,\vv_n),\uu_n-\vv_n\rangle \\&=\big( \uu_n\otimes\uu_n-\vv_n\otimes \vv_n: \nabla(\uu_n-\vv_n)\big) \\ &
				\quad-\nu\langle |\bfD(\uu_n)|^{p-2}\bfD(\uu_n)-|\bfD(\vv_n)|^{p-2}\bfD(\vv_n):\bfD(\uu_n-\vv_n)\rangle \\&\quad -\alpha\big( |\uu_n|^{q-2}\uu_n-|\vv_n|^{q-2}\vv_n,\uu_n-\vv_n\big)
				\\& \leq
				\big( \uu_n\otimes\uu_n-\vv_n\otimes \vv_n: \nabla(\uu_n-\vv_n)\big),
				\end{align*}where we have used \eqref{2.4}. If $\|\uu_n\|_{2} \leq R$ and $\|\vv_n\|_{2} \leq R$, there exists a constant $C(R,n)>0$ such that
		\begin{align*}
				\langle \bfb(t,\uu_n)-\bfb(t,\vv_n),\uu_n-\vv_n\rangle \leq C(R,n)\|\uu_n-\vv_n\|^2_{2}.
			\end{align*}
This, together with the Lipschitz-continuity of $\bfG$ (cf. \eqref{3.6a}), gives us the weak monotonicity property in the sense of Lemma~\ref{prop:e:sde}-(2). Using the fact $\langle \uu_n\otimes\uu_n:\nabla \uu_n\rangle =0$, we find
		\begin{align*}
	   \langle \bfb(\cdot,\uu_n),\uu_n\rangle
&=   -\nu \langle|\bfD(\uu_n)|^{p-2}\bfD(\uu_n):\bfD(\uu_n)\rangle -\alpha\big( |\uu_n|^{q-2}\uu_n,\uu_n\big)+\big( \f,\uu_n\big) \\
&	\leq 			C(1+\|\f\|_{2}\|\uu_n\|_{2}) \leq C(1+\|\f\|_{2})(1+\|\uu_n\|^2_{2}).			
		\end{align*}
Using the linear growth of $\bfG$ cf. \eqref{3.6a}, we find
		\begin{align*}
				\langle \bfb(t,\uu_n),\uu_n\rangle+\|\bfG\|_{\mathcal{L}_2}^2  \leq C(1+\|\f\|_{2})(1+\|\uu_n\|^2_{2}).
				\end{align*}
Since the term $\int_{0}^{T}(1+\|\f\|_{2})\d t$ is finite, $\P-$a.s., {this yields the weak coercivity property in the sense of Lemma~\ref{prop:e:sde}-(3).
By Lemma~\ref{prop:e:sde}, we obtain the existence of a unique strong solution $\bfC\bfc\in \L^2(\Omega,\mathscr{F},\P;\C([0,T^*];\R^n))$ to the stochastic system \eqref{5.4}}, which implies
		\begin{align*}
			\sup_{t\in[0,T^*]}\|\bfC\bfc(t)\|_{\R^n} <+\infty.
		\end{align*}
Our {aim is} to show that $\sup\limits_{t\in[0,T^*]}\| \bfc(t)\|_{\R^n} <+\infty$, which we obtain as follows:
		\begin{align*}			
				\sup_{t\in[0,T^*]}\| \bfc(t)\|_{\R^n} &=\sup_{t\in[0,T^*]}\|\bfC^{-1}\bfC\bfc(t)\|_{\R^n} \leq \|\bfC^{-1}\|_{\mathcal{L}(\R^n;\R^n)}\sup_{t\in[0,T^*]}\|\bfC\bfc(t)\|_{\R^n}
				<+\infty,
		\end{align*}since $\bfC^{-1}$ is a bounded operator.
		Thus, we obtain the existence of {a unique strong solution to the stochastic system \eqref{5.3}}.

		\subsection{Uniform estimates}\label{uestimate}
In this subsection, we establish the uniform energy estimate for the solution of the finite-dimensional approximate system corresponding to the system \eqref{AS1}.

		\begin{theorem}\label{thrmUE}
Let $p\in(1,\infty)$ and assume that \eqref{3.6a} and \eqref{hyp:q} are verified.
Assume, in addition, that \eqref{3.4} and \eqref{3.5} hold with $\gamma=2$.
Then there exists a positive constant $C,$ neither depending on $n$ nor on $\alpha$ such that
		\begin{align}\label{5.9}\nonumber
		&
		\E\left[\sup_{t\in(0,T)}\left\{\|\uu_n(t)\|^2_{2} + 2\kappa{\|\nabla\uu_n(t)\|_{2}^2}\right\} +4C({p,\1})\nu
		\int_0^T{\|\nabla\uu_n(t)\|_{p}^p}\d t+2\alpha \int_0^{T}\|\uu_n(t)\|^q_{q}\d t
		\right]  \\
		&\leq
		C(\kappa)\left\{\bigg(\frac{1}{\eta_1}+\kappa\bigg)\int_\bfV\|\bz\|_{\bfV}^2\d\Lambda_0(\bz)+\int_{\bfL^{2}(\1_T)}\|\bfg\|_{\bfL^{2}(\1_T)}^2\d\Lambda_{\f}(\bfg)+C(K)T\right\}e^{\frac{C(K)T}{\eta_1}},
\end{align}where $C$ is independent of $\alpha$ and $\eta_1$ is the first eigenvalue of the Dirichlet Laplacian.
	\end{theorem}
\begin{proof}	Define  a sequence of stopping times $\{\tau_N^n\}_{N\in\mathbb{N}}$ as follows:
	\begin{align}\label{ST}
		\tau_N^n:= \inf_{t\in[0,T]} \{t : \|\nabla \uu_n(t)\|_2 \geq N\}.
	\end{align}
	Applying finite-dimensional  It\^o's formula to the process $\|(\I-\kappa\Delta)^{\frac{1}{2}}\uu_n(\cdot)\|^2_{2}$, we obtain, $\P-$a.s.,
		\begin{align}\label{5.6}\nonumber			
&	\|\uu_n(\t)\|^2_{2}+\kappa\|\nabla\uu_n(\t)\|_{2}^2 	\\\nonumber
&=
\|\uu_n(0)\|^2_{2}+\kappa\|\nabla\uu_n(0)\|_{2}^2-2\nu\int_{0}^{\t} \int_\1 \bfA(\uu_n):\bfD(\uu_n)\,\d\x \d s
\\&\quad\nonumber -2\alpha\int_0^{\t}\int_\1\bfa(\uu_n)\cdot\uu_n\,\d\x \d s
+2\int_{0}^{\t}\int_\1\Phi(\uu_n)\d\W_n(s)\cdot\uu_n\,\d\x\\&\nonumber\quad
\quad+\int_0^{\t}\int_\1\d \bigg\langle\int_0^{\cdot} \Phi(\uu_n)\d\W_n\bigg\rangle_s\,\d\x\\
&\nonumber
\leq \bigg(\frac{1}{\eta_1}+\kappa\bigg)\|\nabla\uu_n(0)\|_{2}^2-2\nu\int_{0}^{t}\|\bfD(\uu_n(s))\|_{p}^p \d s-2\alpha\int_0^{\t}\|\uu_n(s)\|^q_{q}\d s
				\\
&\nonumber
\quad+2 \int_{0}^{\t}\int_\1 \f\cdot\uu_n\,\d\x\d s +2\int_{0}^{\t}\int_\1\Phi(\uu_n)\d\W_n(s)\cdot\uu_n\d\x \\&\quad + \int_{0}^{\t}\int_\1\d \bigg\langle\int_0^{\cdot} \Phi(\uu_n)\d\W_n\bigg\rangle_s\,\d\x,
				\end{align}
where we have used  Poincar\'e's inequality. Taking expectation on both sides and then using Korn's inequality (cf. \cite{VAK} or \cite[Lemma 2.1]{SNAHBGKK1}), we find
		\begin{align*}
&\E\bigg[	\|\uu_n(\t)\|^2_{2}+\kappa\|\nabla\uu_n(\t)\|_{2}^2+2C({p,\1})\nu\int_{0}^{\t}\|\nabla\uu_n(s)\|_{p}^p \d s \bigg]\\&\qquad+2\alpha \E\bigg[\int_0^{\t}\|\uu_n(s)\|^q_{q}\d s \bigg]\\
& \leq \E\bigg[\bigg(\frac{1}{\eta_1}+\kappa\bigg)\|\nabla\uu_n(0)\|_{2}^2 +
2 \underbrace{\int_{0}^{\t}\int_\1 \f\cdot\uu_n\,\d\x\d s}_{:=I_1(\t)}  +
\underbrace{\int_{0}^{\t}\int_\1\d \bigg\langle\int_0^{\cdot} \Phi(\uu_n)\d\W_n\bigg\rangle_s\,\d\x}_{:=I_2(\t)} \\
&\quad+2\underbrace{\int_{0}^{\t}\int_\1\Phi(\uu_n)\d\W_n(s)\cdot\uu_n\d\x}_{:=I_3(\t)}\bigg]
				\\
& \leq  \bigg(\frac{1}{\eta_1}+\kappa\bigg)\E\big[\|\nabla\uu_n(0)\|_{2}^2\big]+ 2\E\big[I_1(\t)\big]+ \E\big[I_2(\t)\big]+2 \E\big[I_3(\t)\big].
		\end{align*}
We estimate $\E\big[|I_1(\t)|\big]$ using Young's and Poincaré's inequalities as follows:
		\begin{align*}
			\E\big[|I_1(\t)|\big] &\leq  \e\E\bigg[\int_{0}^{\t}\|\uu_n(s)\|^2_{2}\d s\bigg]+C(\e)\E\bigg[\int_{0}^{\t}\|\f(s)\|^2_{2}\d s\bigg]\\& \leq \frac{\e}{\eta_1} \E\bigg[\int_{0}^{\t}{\|\nabla\uu_n(s)\|_{2}^2}\d s\bigg]+C(\e)\E\bigg[\int_{0}^{\t}\|\f(s)\|^2_{2}\d s\bigg].
		\end{align*}
We know that $\E\big[I_3(\t)\big]=0$, since the term $I_3(\t)$ is a martingale with zero expectation.
	Now, we consider the term $\E\big[I_2(\t)\big]$ and estimate it with the help of  It\^o's isometry, \eqref{3.6a} {and Poincaré's inequality} as,
	\begin{align*}
		\E\big[I_2(\t)\big] &=\E\bigg[\sum_{k=1}^n\int_0^{\t}\int_\1\big|\Phi(\uu_n)\bfe_k\big|^2\d\x\d s\bigg]\\&\leq \E\bigg[\sum_{k\in\N}\int_0^{\t}\int_\1|\phi_k(\uu_n)|^2\,\d\x\d s\bigg]\leq C(K) \E\bigg[T+{\frac{1}{\eta_1}\int_0^{\t}\|\nabla\uu_n(s)\|^2_{2}}\d s\bigg].
	\end{align*}
	Combining the above estimates and using it in \eqref{5.6}, we find
		\begin{align*}			
				&\E\bigg[	\|\uu_n(\t)\|^2_{2}+\kappa{\|\nabla\uu_n(\t)\|_{2}^2}+2C({p,\1})\nu\int_{0}^{\t}{\|\nabla\uu_n(s)\|_{p}^p} \d s +2\alpha \int_0^{\t}\|\uu_n(s)\|^q_{q}\d s
				\bigg]\\
& \leq \bigg(\frac{1}{\eta_1}+\kappa\bigg)\E\big[{\|\nabla\uu_n(0)\|_{2}^2}\big]+C(K,T)+\frac{1}{\eta_1}(C(K)+\e) \E\bigg[\int_{0}^{\t}{\|\nabla\uu_n(s)\|_{2}^2}\d s\bigg]\\&\quad+C(\e)\E\bigg[\int_{0}^{\t}\|\f(s)\|^2_{2}\d s\bigg].
				\end{align*}
An application of Gronwall's inequality with a proper choice of $\e$  in the above inequality yields
		\begin{align}\label{5.7}\nonumber
				&\E\bigg[{\|\nabla\uu_n(\t)\|_{2}^2}\bigg] \\&\leq C(\kappa)\bigg\{\bigg(\frac{1}{\eta_1}+\kappa\bigg)\E\big[{\|\nabla\uu_n(0)\|_{2}^2}\big]+C\E\bigg[\int_{0}^{T}\|\f(s)\|^2_{2}\d s\bigg]+C(K,T)\bigg\}e^{\frac{C(K)T}{\eta_1}}.
		\end{align}Taking limit $N\to\infty$ in \eqref{5.7} and using the monotone convergence theorem, we get
		\begin{align}\label{5.007}\nonumber
		&\E\bigg[{\|\nabla\uu_n(t)\|_{2}^2}\bigg] \\&\leq C(\kappa)\bigg\{\bigg(\frac{1}{\eta_1}+\kappa\bigg)\E\big[{\|\nabla\uu_0\|_{2}^2}\big]+C\E\bigg[\int_{0}^{T}\|\f(s)\|^2_{2}\d s\bigg]+C(K,T)\bigg\}e^{\frac{C(K)T}{\eta_1}},
	\end{align}for all $t\in[0,T]$.

Taking supremum from 0 to $\tt$ and then expectation in \eqref{5.6}, we find
		\begin{align}\label{5.8}\nonumber
			&	
\E\bigg[\sup_{t\in[0,\tt]}\big\{\|\uu_n(t)\|^2_{2}+\kappa{\|\nabla\uu_n(t)\|_{2}^2}\big\}+2C({p,\1})\nu\int_{0}^{\tt}{\|\nabla\uu_n(s)\|_{p}^p} \d s \bigg] \\
&
\qquad\nonumber+2\alpha \E\bigg[\int_0^{\tt}\|\uu_n(s)\|^q_{q}\d s \bigg] \\
&\nonumber \leq \bigg(\frac{1}{\eta_1}+\kappa\bigg)\E\big[{\|\nabla\uu_0\|_{2}^2}\big]+C(\e)\E\bigg[ \int_{0}^{T}\|\f(s)\|^2_{2}\d s\bigg]+C(K)T\\&\quad+\frac{1}{\eta_1}(C(K)+\e) \E\bigg[\int_{0}^{\tt}{\|\nabla\uu_n(s)\|_{2}^2}\d s\bigg]+2\E\bigg[\sup_{t\in[0,\tt]}I_3     (t)\bigg].
		\end{align}
We estimate the final term from the right hand side of  \eqref{5.8} using the BDG (see \cite[Theorem 1.1]{DLB}), H\"older's and Young's inequalities as follows:
		\begin{align}\label{BDG}\nonumber
					\E\bigg[\sup_{t\in[0,\tt]}|I_3(t)|\bigg]&\nonumber=	\E\bigg[\sup_{t\in[0,\tt]}\bigg|\int_{0}^{t}\int_\1\Phi(\uu_n)\d\W_n(s)\cdot\uu_n(s)\,\d\x\bigg|\bigg]\\
&\nonumber
=					\E\bigg[\sup_{t\in[0,\tt]}\bigg|\int_{0}^{t}\sum_{k=1}^n\int_\1\Phi(\uu_n)\bfe_k\d\beta_k(s)\cdot\uu_n\,\d\x\bigg|\bigg]
\\&\nonumber
=					\E\bigg[\sup_{t\in[0,\tt]}\bigg|\int_{0}^{t}\sum_{k=1}^n\int_\1\phi_k(\uu_n)\d\beta_k(s)\cdot\uu_n\,\d\x\bigg|\bigg] \\
&\nonumber
\leq C\E \bigg[\int_0^{\tt} \sum_{k=1}^n\bigg(\int_\1\phi_k(\uu_n)\cdot\uu_n\,\d\x\bigg)^2\d s\bigg]^\frac{1}{2} \\&\nonumber
\leq C\E \bigg[\int_0^{\tt} \sum_{k=1}^n\bigg(\int_\1|\uu_n|^2\d \x \int_\1|\phi_k(\uu_n)|^2\d\x\bigg)\d s\bigg]^\frac{1}{2} \\&\nonumber
\leq C(K)\E \bigg[1+\int_0^{\tt}\bigg(\int_\1|\uu_n|^2\d \x\bigg)^2\d s\bigg]^\frac{1}{2} \\&\leq \e \E \bigg[\sup_{t\in[0,\tt]} \|\uu_n(s)\|^2_{2}\bigg]+C(\e,K)\E\bigg[T+\int_0^{\tt}\|\uu_n(s)\|^2_{2}\d s\bigg].
					\end{align}
Substituting the above estimate in \eqref{5.8} with a proper choice of $\e$, and an application of Gronwall's inequality yields
		\begin{align*}
			&\E\bigg[\sup_{t\in[0,\tt]}\big\{\|\uu_n(t)\|^2_{2}+2\kappa{\|\nabla\uu_n(t)\|_{2}^2}\big\}+4C({p,\1})\nu\int_{0}^{\tt}{\|\nabla\uu_n(t)\|_{p}^p}\d t\bigg]\\&\qquad+2\alpha \E\bigg[\int_0^{\tt}\|\uu_n(t)\|^q_{q}\d t
			 \bigg]\\& \leq
			C(\kappa) \bigg\{\bigg(\frac{1}{\eta_1}+\kappa\bigg)\int_{\bfV}\|\bz\|_{\bfV}^2\d \Lambda_0(\bz)+2C\int_{\bfL^2(\1_T)}\|\bfg\|_{\bfL^2(\1_T)}^2\d \Lambda_{\f}(\bfg)+C(K,T)\bigg\}e^{\frac{C(K)T}{\eta_1}}.
		\end{align*}
	Finally, passing $N\to \infty$ in the above inequality, we arrive at
		\begin{align*}
			&
				\E\bigg[\sup_{t\in[0,T]}\left\{\|\uu_n(t)\|^2_{2} + 2\kappa{\|\nabla\uu_n(t)\|_{2}^2}\right\} +4C({p,\1})\nu
				\int_0^T{\|\nabla\uu_n(t)\|_{p}^p}\d t+2\alpha \int_0^{T}\|\uu_n(t)\|^q_{q}\d t
				 \bigg]  \\
				&\leq
				C(\kappa)\left\{\bigg(\frac{1}{\eta_1}+\kappa\bigg)\int_\bfV\|\bz\|_{\bfV}^2\d\Lambda_0(\bz)+\int_{\bfL^{2}(\1_T)}\|\bfg\|_{\bfL^{2}(\1_T)}^2\d\Lambda_{\f}(\bfg)+C(K,T)\right\}e^{\frac{C(K)T}{\eta_1}},
		\end{align*}
		by assuming the right-hand side is finite, which is true due to (\ref{3.4}) and (\ref{3.5}).\end{proof}
	
	\begin{theorem}\label{thrmUE2}
Let $p\in(1,\infty)$ and assume the conditions \eqref{3.6a} and \eqref{hyp:q} are satisfied.
Assume, in addition, that \eqref{3.4} and \eqref{3.5} hold with $\gamma=2$.
Then there exists a \emph{martingale solution} to the approximate system \eqref{AS1},
		\begin{align*}
			((\overline{\Omega},\overline{\mathscr{F}},\{\overline{\mathscr{F}}_t\}_{t\in[0,T]},\overline{\P}),\overline{\uu},\overline{\uu}_0,\overline{\f},\overline{\W}),
		\end{align*}
in the following sense:
	\begin{enumerate}
		\item $ (\overline{\Omega},\overline{\mathscr{F}},\{\overline{\mathscr{F}}_t\}_{t\in[0,T]},\overline{\P})$ is {a} stochastic basis, with a complete right-continuous filtration $\{\overline{\mathscr{F}}_t\}_{t\in [0,T]}$;
		\item $\overline{\W}$ is a cylindrical $\{\overline{\mathscr{F}}_t\}_{t\in [0,T]}-$adapted Wiener process;
		\item $\overline{\uu}$ is a progressively $\{\overline{\mathscr{F}}_t\}_{t\in [0,T]}-$measurable stochastic process  with $\overline{\P}-$a.s. paths	$t\mapsto\overline{\uu}(t,\omega)\in \L^\infty(0,T;\bfV)\cap \L^p(0,T;{\W_{0}^{1,p}(\1)^d})\cap \L^q(0,T;{\L^q(\1)^d})$, with a continuous modification having paths in $\C([0,T];\bfV)$;
		\item $\overline{\uu}_0$ is progressively $\{\overline{\mathscr{F}}_t\}_{t\in [0,T]}-$measurable on the probability space $(\overline{\Omega},\overline{\mathscr{F}},\overline{\P})$,  with $\overline{\P}-$a.s. paths $\overline{\uu}_0(\omega)\in \bfV$ and $\Lambda_0=\overline{\P}\circ\overline{\uu}_0^{-1}$ in the sense of (\ref{3.1});
		\item $\overline{\f}$ is a $\{\overline{\mathscr{F}}_t\}_{t\in [0,T]}-$adapted stochastic process{, with} $\overline{\P}-$a.s. paths $\overline{\f}(t,\omega)\in \bfL^2(\1_T)$ and	$\Lambda_{\overline{\f}}=\overline{\P}\circ\overline{\f}^{-1}$ in the sense of (\ref{3.2});
		\item for every $\bfi\in {\C_0^{\infty}(\1)^d}${, with $\operatorname{div}\bfi=0$ in $\1_T$, and for} all $t\in[0,T],$ the following identity holds $\overline{\P}-$a.s.:
		\begin{align}\label{3.3}\nonumber
			& \int_\1\overline{\uu}(t)\cdot\bfi\,\d\x + \kappa\int_\1\nabla\overline{\uu}(t):\nabla\bfi\,\d\x - \int_0^t\int_\1\overline{\uu}\otimes\overline{\uu}:\mathds{\nabla}\bfi\,\d\x\d s \\
&\nonumber\quad +\alpha \int_0^t\int_\1|\overline{\uu}|^{q-2}\overline{\uu}\cdot{\bfi}\,\d\x\d s	+	\nu\int_0^t\int_\1|\bfD(\overline{\uu})|^{p-2}\bfD(\overline{\uu}):\bfD\bfi \,\d\x\d s
			\\&\nonumber=
			 \int_\1\overline{\uu}_0\cdot\bfi\,\d\x + \kappa \int_\1\nabla\overline{\uu}_0:\nabla\bfi\,\d\x +
		\int_0^t\int_\1 \overline{\f}\cdot\bfi\,\d\x\d s \\&\quad+ \int_0^t\int_\1\Phi(\overline{\uu})\d\overline{\W}(s)\cdot\bfi\,\d\x.
		\end{align}	\end{enumerate}
{In addition, there exists a positive constant $C$, not depending on $\alpha$, such that}
		\begin{align}\label{UE2}\nonumber
		&
		\E\left[\sup_{t\in[0,T]}\left\{\|\overline{\uu}(t)\|^2_{2} + 2\kappa{\|\nabla\overline{\uu}(t)\|_{2}^2}\right\} +4\nu
		\int_0^T{\|\nabla\overline{\uu}(t)\|_{p}^p}\d t+2\alpha \int_0^{T}\|\overline{\uu}(t)\|^q_{q}\d t
		\right]  \\
		&\leq
		C\left\{\bigg(\frac{1}{\lambda_1}+\kappa\bigg)\int_\bfV\|\overline{\bz}\|_{\bfV}^2\d\Lambda_0(\overline{\bz})+\int_{\bfL^{2}(\1_T)}\|\overline{\bfg}\|_{\bfL^{2}(\1_T)}^2\d\Lambda_{\overline{\f}}(\overline{\bfg})+C(K,T)\right\}.
	\end{align}
	\end{theorem}

\begin{proof}
For the sake of better comprehension, the proof of Theorem~\ref{thrmUE2} shall be split through the following sections:
\subsection{Weak convergence}\label{WC}
		In view of (\ref{5.9}), the Banach-Alaoglu theorem gives the existence of functions $\uu$, $\w$, $\bfS$ and $\Psi$ such that, {for some subsequences still labeled by the same subscript,}
		\begin{alignat}{2}
			\label{5.10}
			\uu_n&\xrightharpoonup[n\to\infty]{} \uu,\quad && \text{in}\quad \L^2\big({\Omega,\mathscr{F},\P};\L^\infty(0,T;\bfV)\big),  \\
			\label{5.11}
			\uu_n&\xrightharpoonup[n\to\infty]{} \uu,\quad  &&\text{in}\quad \L^p\big({\Omega,\mathscr{F},\P};\L^{p}(0,T;{\W_0^{1,p}(\1)^d}\big),  \\
				\uu_n &\xrightharpoonup[n\to\infty]{} \uu,\quad && \text{in} \quad 	\L^q\big({\Omega,\mathscr{F},\P};\L^{q}(0,T;{\L^q(\1)^d})\big),	\\
		\bfa(\uu_n)&\xrightharpoonup[n\to\infty]{} \bfa,\quad && \text{in} \quad  \L^{q'}\big({\Omega,\mathscr{F},\P};\L^{q'}(0,T;{\L^{q'}(\1)^d})\big), \\
				\label{5.12}
			\uu_n\otimes\uu_n&\xrightharpoonup[n\to\infty]{} \w,\quad  &&\text{in}\quad {\L^{\frac{q}{2}}\big({\Omega,\mathscr{F},\P};\L^{\frac{q}{2}}(0,T;\L^{\frac{q}{2}}(\1)^{d\times d}))}, \\
			\label{5.012}
			\bfA(\uu_n)&\xrightharpoonup[n\to\infty]{} \bfS,\quad && \text{in}\quad \L^{p'}\big({\Omega,\mathscr{F},\P};\L^{p'}(0,T;{\L^{p'}(\1)^{d\times d}})\big),
			 \\
			\label{5.13}
			\bfA(\uu_n)&\xrightharpoonup[n\to\infty]{} \bfS,\quad  &&\text{in}\quad \L^{p'}\big({\Omega,\mathscr{F},\P};\L^{p'}(0,T;{\W^{-1,p'}(\1)^{d\times d}})\big),\\
			\label{5.14}
			\Phi(\uu_n)&\xrightharpoonup[n\to\infty]{} \Psi,\quad && \text{in}\quad \L^{2}\big({\Omega,\mathscr{F},\P};\L^{2}(0,T;{\mathcal{L}_2(\bfU,\L^2(\1)^d)})\big).
		\end{alignat}Our aim is to establish {that}
	\begin{align*}		\w = \uu\otimes \uu, \ \ \bfS=\A(\uu), \ \text{ and }\  \Psi =\Phi(\uu).
	\end{align*}
		\subsection{Compactness}
		From the Subsection \ref{WC}, we have the weak {convergence results} \eqref{5.10}-\eqref{5.14}, which are not enough to pass {to} the limit in the nonlinear terms and {in} the noise coefficient appearing in our model. In order to pass these terms to the limit, we need some compactness arguments.
		
		We test (\ref{5.2}) with $\bphi\in\bVcal$, so that $\P-$a.s.,
		\begin{align}\nonumber
			\label{5.15}%
			&
			\int_\1\uu_n(t)\cdot\bphi\,\d\x+\kappa\int_\1\nabla\uu_n(t):\nabla\bphi\,\d\x
				\equiv
				\int_\1\uu_n(t)\cdot {P^n_s}(\bphi)\,\d\x+\kappa\int_\1\nabla\uu_n(t):\nabla P^n_s(\bphi)\,\d\x \\
				&\nonumber=
				\int_\1\uu_0\cdot P^n_s(\bphi)\,\d\x+\kappa\int_\1\nabla\uu_0(t):\nabla P^n_s(\bphi)\,\d\x +
			\int_0^t\int_\1\bfG_n:\nabla P^n_s(\bphi)\,\d\x\d s
				\\
				&\quad+
			\int_0^t\int_\1\Phi(\uu_n)\d\W_n(s)\cdot P^n_s(\bphi)\,\d\x,
				\end{align}
		where $P^n_s$ denotes the projection into the $n-$dimensional space $\bfX^{n}$ with respect to the {$\bfV^s$} inner product, and
		\begin{align}\label{5.16}
			\bfG_n:=\uu_n\otimes\uu_n+\nabla \Delta^{-1}\bfa(\uu_n)
			-\nu\bfA(\uu_n)+\bfF,
		\end{align}
		with $\bfF$ chosen in {$\L^2(0,T;\W^{1,2}(\1)^{d\times d})$ in such a way} that $\operatorname{div} \bfF=-\f$ in the weak sense.

\begin{claim}\label{claim:1}
$\w=\uu\otimes\uu$ and $\Psi=\Phi(\uu)$.
\end{claim}		
\begin{proof}[Proof of Claim~\ref{claim:1}]
Using the energy estimate obtained in \eqref{5.9} and the definitions of {$\bfA$ and $\bfa$ (see \eqref{op:A(u)} and \eqref{ref:a(u)}), and observing \eqref{hyp:q},}  we find
		\begin{align}\label{5.17}
			\bfG_n\in \L^{q_0}({\Omega,\mathscr{F},\P};\L^{q_0}(0,T;{\L^{q_0}(\1)^{d\times d}})),\qquad {q_0:=\min\left\{p',q'\right\}>1},
		\end{align}
uniformly in $n$.
{Note that in fact it should be $q_0:=\min\left\{p',q',\frac{q}{2}\right\}>1$, but once that $q\geq 3$ (see assumption \eqref{hyp:q}), we have $q_0=\min\left\{p',q'\right\}$.}
Let us define the functional
		\begin{align*}
			\mathcal{H}(t,\bphi):=\int_0^t\int_\1\bfG_n:\nabla P^n_s(\bphi)\,\d\x\d s,\qquad {\bphi\in\bVcal}.
		\end{align*}
As  $1+\frac{2}{q_0}>\frac{1}{q_0}-\frac{1}{2}$ implies the embedding $\W^{\tilde{s},q_0}(\1)\hookrightarrow\W^{s,2}(\1)$ for $\tilde{s}\geq s+d\big(1+\frac{2}{q_0}\big)$, we can use \eqref{5.17} to show that
		\begin{align}\label{IP}
			\E\bigg[\|\mathcal{H}\|_{\W^{1,q_0}(0,T;\W_{\sigma}^{-\tilde{s},q_0}(\1))}\bigg]\leq C.
		\end{align}
The above estimate can be justified as follows (cf. \cite[Section 4]{DBFG}):

	\begin{align*}
& \bigg\| \frac{\d}{\d t }\mathcal{H} (t,\cdot)\bigg\|_{\L^{q_0}(0,T;\W_{\sigma}^{-\tilde{s},q_0}(\1))}  \\
& = \bigg\|\sup_{\|\bphi\|_{\tilde{s},q_0'}\leq 1}\frac{\d }{\d t} \mathcal{H} (t,\bphi)\bigg\|_{\L^{q_0}(0,T)} =
\bigg\|\sup_{\|\bphi\|_{\tilde{s},q_0'}\leq 1}\int_{\1}\bfG_n(t):\nabla P_{\tilde{s}}^n(\bphi)\,\d\x\bigg \|_{\L^{q_0}(0,T)}  \\
&	\leq	\bigg\|\sup_{\|\bphi\|_{\tilde{s},q_0'}\leq 1}\| \bfG_n(t)\|_{q_0}\|\nabla P_{\tilde{s}}^n(\bphi)\|_{q_0'}\bigg \|_{\L^{q_0}(0,T)}
		\leq C\bigg(\int_0^T\| \bfG_n(t)\|_{q_0}^{q_0}\d t \bigg)^{\frac{1}{q_0}}.
	\end{align*}
For all $0\leq s<t\leq T$, we have
\begin{align*}
&	\E\bigg[\bigg\|\int_0^t\Phi (\uu_n(\ell))\d \W_n(\ell)-\int_0^s\Phi (\uu_n(\ell))\d \W_n(\ell) \bigg\|_{2}^{q}\bigg]  \\
&= \E\bigg[\bigg\|\int_s^t\Phi(\uu_n(\ell))\d\W_n(\ell)\bigg\|_{2}^{q}\bigg]
=\E\bigg[\bigg\|\int_s^t\sum_{j=1}^{\infty}\Phi(\uu_n(\ell))\bfe_i\d\beta_j^n\bigg\|_{2}^{q}\bigg]  \\
&\leq C\E\bigg[\bigg(\int_s^t\sum_{j=1}^{\infty}\|\phi_k(\uu_n(\ell))\|^2_{2}\d \ell \bigg)^{\frac{q}{2}}\bigg]\leq
C|t-s|^{\frac{q}{2}}\bigg\{\E\bigg[\sup_{\ell \in[0,T]}\big\{1+\|\uu_n(\ell)\|^2_{2}\big\}\bigg]\bigg\}^{\frac{q}{2}}
\leq C|t-s|^{\frac{q}{2}},
\end{align*}
where we have used the BDG and Young's inequalities and the energy estimate \eqref{5.9}.
By the Kolmogorov continuity criterion, there exists a modification which has $\P-$a.s. Hölder continuous paths such that
\begin{align}\label{STphi1}
			\E\bigg[\bigg\|\int_{0}^{t}\Phi(\uu_n(s))\d\W_n(s)\bigg\|_{\C^\mu([0,T];\L^2(\1)^d)}\bigg] \leq C,\quad \mu:=\theta-\frac{1}{q},
\end{align}
for $\frac{1}{q}<\theta<\frac{1}{2}$ if $q>2$, which is the case due to assumption \eqref{hyp:q}.
Observing that $\W_0^{\tilde{s},q_0'}(\1)^d\hookrightarrow \W_{0}^{1,2}(\1)^d\hookrightarrow\L^2(\1)^d$,  $1<q_0<\infty$, implies $\L^2(\1)^d\hookrightarrow\W^{-\tilde{s},q_0}(\1)^d$, by the above inequality, we have
	\begin{align}\label{Phi1}	\E\bigg[\bigg\|\int_{0}^{t}\Phi(\uu_n(s))\d\W_n(s)\bigg\|_{\C^\mu([0,T];\W_{0}^{-\tilde{s},q_0}(\1)^d)}\bigg] \leq C.
		\end{align}
	 Collecting the {information from \eqref{IP} and \eqref{Phi1}} in \eqref{5.15}, and still using \eqref{3.4}, we arrive at
		\begin{align*}
			\E\left[\|(\I-\kappa\Delta)\uu_n\|_{\C^\mu([0,T];\W_{\sigma}^{-\tilde{s},q_0}(\1)^d)}\right]\leq C,
		\end{align*}
{for some positive constant $C$ that does not depend on $n$.
This in turn implies  for some $\eta>0$ that
		\begin{align}\label{5.19}
			\E\left[\|(\I-\kappa\Delta)\uu_n\|_{\W^{\eta,q_0}(0,T;\W_{\sigma}^{-\tilde{s},q_0}(\1)^d)}\right]\leq C.
		\end{align}From \eqref{5.19}, we conclude that
		\begin{align}\label{5.019}
		\E\left[\|\uu_n\|_{\W^{\eta,q_0}(0,T;\W_{\sigma}^{2-\tilde{s},q_0}(\1)^d)}\right]\leq C.
	\end{align}Observing that $\tilde{s}\geq s+d\big(1+\frac{2}{q_0}\big)$ implies $2-\tilde{s}<0$.
	
Let
\begin{equation*}
\bfZ_{p,q}:=\L^\infty(0,T;\bfV)\cap \L^p(0,T;\W^{1,p}_0(\1)^d)\cap \L^q(0,T;\L^q(\1)^d).
\end{equation*}
By the Sobolev embedding theorem, we infer
\begin{equation*}
\bfZ_{p,q}\hookrightarrow\L^\infty(0,T;\L_{\sigma}^{\rho}(\1)^d)\cap \L^p(0,T;\W^{1,p}_0(\1)^d),\qquad \rho\leq \min\left\{2^\ast,q\right\}.
\end{equation*}
By a version of the Aubin-Lions compactness lemma (see~\cite[Theorem 2.1]{FFDG}), we further have
\begin{equation*}
\W^{\eta,q_0}(0,T;\W_{\sigma}^{2-\tilde{s},q_0}(\1)^d)\cap\L^\infty(0,T;\L_{\sigma}^{\rho}(\1)^d)\cap \L^p(0,T;\W^{1,p}_0(\1)^d)\hookrightarrow\hookrightarrow \L^\rho(0,T;\L_{\sigma}^\rho(\1)^d),
\end{equation*}
for $q_0\leq \rho <p^\ast$.
In view of this, and attending to the definition of $q_0$ in \eqref{5.17}, one has
		\begin{align}\label{5.20}
			\W^{\eta,q_0}(0,T;\W_{\sigma}^{2-\tilde{s},q_0}(\1)^d)\cap \bfZ_{p,q} \hookrightarrow\hookrightarrow \L^\rho(0,T;\L_{\sigma}^\rho(\1)^d),
		\end{align}
for
\begin{equation}\label{rho,p,q}
\min\left\{p',q'\right\}<\rho<\min\left\{2^\ast,p^\ast,q\right\},\qquad p>\frac{2d}{d+2}.
\end{equation}
}
Now, we define {the space
		\begin{align*}
			\mathfrak{V}:=\L^\rho(0,T;\L_{\sigma}^\rho(\1)^d)\otimes \C([0,T];\bfU_0)\otimes \bfV\otimes \bfL^2(\1_T).
		\end{align*}
In the} sequel, we use the following notations:
		\begin{enumerate}
			\item $\varrho_{\uu_n}$ denotes the law of $\uu_n$ on {$\L^\rho(0,T;\L_{\sigma}^\rho(\1)^d)$};
			\item $\varrho_{\W}$ denotes the law of $\W$ on $\C([0,T];\bfU_0)$;
			\item $\varrho_n$ denotes the joint law of $\uu_n,\W,\uu_0$ and $\f$ on the space $\mathfrak{V}$.
		\end{enumerate}
Let us consider a ball $B_N$ in the space $\W^{\eta,q_0}([0,T];\W_{\sigma}^{2-\tilde{s},q_0}(\1)^d)\cap \bfZ_{p,q}$ and denote its complement by $B_N^c$. Using the uniform estimates  \eqref{5.9} and \eqref{5.019}, we find
		\begin{align*}
				\varrho_{\uu_n}(B_N^c)&=\P\left(\|\uu_n\|_{\W^{\eta,q_0}(0,T;\W_{\sigma}^{2-\tilde{s},q_0}(\1)^d)}+\|\uu_n\|_{\bfZ_{p,q}}\geq N\right)\\& \leq \frac{1}{N}\E\left[\|\uu_n\|_{\W^{\eta,q_0}(0,T;\W_{\sigma}^{2-\tilde{s},q_0}(\1)^d)}+\|\uu_n\|_{\bfZ_{p,q}}\right] \leq \frac{C}{N}.
				\end{align*}
For any fixed $\xi>0$, we can find $N(\xi)$ such that
		\begin{align*}
			\varrho_{\uu_n}(B_{N(\xi)})\geq 1-\frac{\xi}{4}.
		\end{align*}
Since, the law $\varrho_{\W}$ is tight as being a Radon measure on the Polish space $\C([0,T];\bfU_0)$, then there exists a compact subset $K_\xi\subset \C([0,T];\bfU_0)$ such that $	\varrho_{\uu_n}(K_\xi)\geq 1-\frac{\xi}{4}$.
By the same reasoning,  we are able {to} find compact subsets of $\bfV$ and $\bfL^2(\1_T)$, such that their measures $\Lambda_0$ and $\Lambda_{\f}$ are greater than $1-\frac{\xi}{4}$.  Therefore, we can find a compact subset $\mathfrak{V}_\xi\subset \mathfrak{V}$ such that $\varrho_n(\mathfrak{V}_\xi)\geq 1-\xi$. Hence, $\{\varrho_n\}_{n\in\N}$ is tight in the same space. An application of Prokhorov's theorem (see  \cite[Theorem 2.6]{NISW}) yields $\varrho_n$ is relatively weakly compact, which implies that $\varrho_n$ has a weakly convergent subsequence with  weak limit $\varrho$.
		\subsection{Existence of a probability space $(\overline{\Omega},\overline{\mathscr{F}},\overline{\P})$}
		
		Applying Skorohod's representation theorem (see  \cite[Theorem 2.7]{NISW}) to ensure the existence of another probability space $(\overline{\Omega},\overline{\mathscr{F}},\overline{\P})$, a random sequence $(\overline{\uu}_n,\overline{\W}_n,\overline{\uu}_0^n,\overline{\f}_n)$ and a random variable $(\overline{\uu},\overline{\W},\overline{\uu}_0,\overline{\f})$ on the probability space $(\overline{\Omega},\overline{\mathscr{F}},\overline{\P})$, taking values in $\mathfrak{V}$ such that the following holds:
		\begin{enumerate}
			\item The laws of sequence of random variables $(\overline{\uu}_n,\overline{\W}_n,\overline{\uu}_0^n,\overline{\f}_n)$  and the random variable $(\overline{\uu},\overline{\W},\overline{\uu}_0,\overline{\f})$ under the new probability measure $\overline{\P}$ coincide with $\varrho_n$ and $\varrho:=\lim\limits_{n\to\infty}\varrho_n$, respectively;
			\item {The} following convergence {results hold true
			\begin{alignat*}{2}
				\overline{\uu}_n & \xrightarrow[n\to\infty]{} \overline{\uu},\quad &&\text{in} \quad  \L^\rho(0,T;\L_{\sigma}^{\rho}(\1)^d), \\
				\overline{\W}_n & \xrightarrow[n\to\infty]{} \overline{\W}, \quad&&\text{in} \quad  \C([0,T];\bfU_0),\\
				\overline{\uu}_0^n & \xrightarrow[n\to\infty]{}\overline{\uu}_0, \quad &&\text{in}\quad \bfV,\\
				\overline{\f}_n & \xrightarrow[n\to\infty]{}\overline{\f}, \quad &&\text{in} \quad \L^2(0,T;\L^2(\1)^d),
			\end{alignat*}
}$\overline{\P}-$a.s.;
			\item {The convergence results \eqref{5.11} and \eqref{5.12}} still hold for the functions defined on the newly constructed probability space $(\overline{\Omega},\overline{\mathscr{F}},\overline{\P})$. Moreover, for any finite $\beta,$ we have
			\begin{align*}
			\overline{\E}\bigg[\sup_{t\in[0,T]}\|\overline{\W}_n(t)\|_{\bfU_0}^\beta\bigg] = \E\bigg[\sup_{t\in[0,T]}\|\overline{\W}(t)\|_{\bfU_0}^\beta\bigg].
			\end{align*}
		\end{enumerate}Applying Vitali's convergence theorem, {we have for some subsequences still labeled by the same subscript}
		\begin{alignat}{2}
			\label{5.21}
			\overline{\W}_n&\xrightarrow[n\to \infty]{} \overline{\W}, \quad &&\text{in}\quad \L^2({\overline{\Omega},\overline{\mathscr{F}},\overline{\P}};\C([0,T];\bfU_0)),\\
			\label{5.22}
			\overline{\uu}_n&\xrightarrow[n\to \infty]{}  \overline{\uu},\quad  &&\text{in}\quad  {\L^\rho({\overline{\Omega},\overline{\mathscr{F}},\overline{\P}};\L^\rho(0,T;\L_{\sigma}^\rho(\1)^d))}, \\
			\label{5.23}
			\overline{\uu}_0^n&\xrightarrow[n\to \infty]{} \overline{\uu}_0, \quad  &&\text{in}\quad  \L^2({\overline{\Omega},\overline{\mathscr{F}},\overline{\P}};\bfV),  \\
			\label{5.24}
			\overline{\f}_n&\xrightarrow[n\to \infty]{}\overline{\f},\quad  &&\text{in}\quad  \L^2({\overline{\Omega},\overline{\mathscr{F}},\overline{\P}};\L^2(0,T;{\L^2(\1)^d})),
		\end{alignat}
for {$\rho$ and $p$ in the conditions of \eqref{rho,p,q}.}
Now, we need to define the filtration on the newly constructed  probability space.
First, we define an operator of restriction to the interval $[0,T]$ denoted by $\bfh_t$  acting on various path spaces.
More precisely, let $\bfX$ denote any of the {spaces $\L^\rho(0,T;\L^\rho(\1)^d)$, $\L^2(0,T;\L^2(\1)^d)$, or $\C([0,T];\bfU_0)$,} and for $t\in[0,T]$, we define
		\begin{align}\label{5.25}
			\bfh_t:\bfX\to\bfX \big|_{[0,t]}, \quad g \mapsto g\big|_{[0,t]}.
		\end{align}
The mapping ${\bf h}_t$ is continuous.
Let us define {the} $\overline{\P}-$augmented canonical {filtration} by $\{\overline{\mathscr{F}}_t\}_{t\in[0,T]}$ of the process $(\overline{\uu},\overline{\f},\overline{\W})$, that is,
		\begin{align*}
			\overline{\mathscr{F}}_t=\sigma\big(\sigma(\bfh_t\overline{\uu},\bfh_t\overline{\f},\bfh_t\overline{\W})\big)\cup\{M\in \overline{\mathscr{F}};\overline{\P}(M)=0\}), \quad t\in[0,T].
		\end{align*}Next, we are going to show that our approximate equations also {hold} in the new probability space.
		
		\subsection{Validity of the approximate equations in $(\overline{\Omega},\overline{\mathscr{F}},\overline{\P})$} The method we employ in this subsection has been already discussed in several works {(see for e.g. \cite{Breit,MH,MO})}. The main aim of this method is to identify the quadratic variation of the martingale  as well as cross variation with the limit Wiener process obtained through compactness. Note that the laws of $\overline{\W}_n$ and $\overline{\W}$ are {the} same. As a consequence of {the} new probability space, we can find a collection of mutually independent real-valued $\{\overline{\mathscr{F}}_t\}_{t\in[0,T]}-$Wiener processes $\{\overline{\beta}_j^n\}_{j\in\N}$ such that $\overline{\W}_n=\sum\limits_{j\in\N}\bfe_j\overline{\beta}_j^n$, that is, there exists a collection of mutually independent real-valued $\{\overline{\mathscr{F}}_t\}_{t\in[0,T]}-$Wiener processes $\{\overline{\beta}_j\}_{j\in\N}$ such that $\overline{\W}=\sum\limits_{j\in\N}\bfe_j\overline{\beta}_j$. We used the abbreviation $\overline{\W}_{n,n}$ for $\sum\limits_{j=1}^{n}\bfe_j\overline{\beta}_j^n$.
Let us consider $\bphi\in {\bVcal}$ and define the functionals for $t\in[0,T]$
		\begin{align*}
				N(\uu_n,\uu_0,\f)_t&=\int_\1\uu_n(t)\cdot\bphi\,\d\x+\kappa\int_\1\nabla\uu_n(t):\nabla\bphi\,\d\x-\int_\1\uu_0^n\cdot\bphi\,\d\x-\kappa\int_\1\nabla\uu_0^n:\nabla\bphi\,\d\x\\
&\quad+\int_0^t\int_\1\f\cdot P^n(\bphi)\,\d\x\d s +\int_0^t\int_\1 \uu_n\otimes\uu_n:\nabla P^n(\bphi) \,\d\x\d s\\&\quad+ \nu\int_0^t\int_\1 \bfA(\uu_n):\bfD(P^n(\bphi))\,\d\x\d s,\\
				M(\uu_n)_t&=\sum_{j=1}^n\int_{0}^{t}\bigg(\int_\1\phi_j(\uu_n)\cdot P^n(\bphi)\,\d\x\bigg)^2\d s,\\
				M_j(\uu_n)_t&=\int_{0}^{t}\int_\1\phi_j(\uu_n)\cdot P^n(\bphi)\,\d\x\d s.
				\end{align*}Let us denote the increment $N(\uu_n,\uu_0,\f)_t-N(\uu_n,\uu_0,\f)_s$ by $N(\uu_n,\uu_0,\f)_{s,t}$ and similarly for $M(\uu_n)_{s,t}$ and $M_j(\uu_n)_{s,t}$. Observe that our proof will be over once we prove that the process $N(\overline{\uu}_n)$ is an $\{\overline{\mathscr{F}}_t\}_{t\in[0,T]}-$martingale and its quadratic and cross variations satisfy
		\begin{align}\label{5.26}
			\langle N(\overline{\uu}_n,\overline{\uu}_0,\overline{\f})\rangle=M(\overline{\uu}_n), \quad\text{ and  } \quad\langle N(\overline{\uu}_n,\overline{\uu}_0,\overline{\f}),\overline{\beta}_j\rangle=M_j(\overline{\uu}_n),
		\end{align}respectively. In that situation, we have
		\begin{align}\label{5.27}
			\bigg\langle N(\overline{\uu}_n,\overline{\uu}_0,\overline{\f})-\int_{0}^{\cdot}\int_\1\Phi(\overline{\uu}_n)\d\overline{\W}_{n,n}\cdot P^n(\bphi)\,\d\x \bigg\rangle=0,
		\end{align}which implies the required result in the new probability space. Let us establish \eqref{5.26}.
To finish this proof, we use the uniform estimate and we claim that the mappings
		\begin{align*}
			(\uu_n,\uu_0,\f)\mapsto N(\uu_n,\uu_0,\f)_t,\quad \uu_n\mapsto M(\uu_n)_t, \quad \text{and}\quad \uu_n\mapsto M_j(\uu_n)_t,
		\end{align*}
are well defined and measurable on a subspace of the path space, where the joint law of $(\overline{\uu}_n,\overline{\uu}_0,\overline{\f})$ is supported, that is, {the uniform estimate \eqref{5.9} of Theorem~\ref{thrmUE} holds}.
In the case of $M(\uu_n)_t,$ we have by \eqref{3.6a} and the continuity of $P^n$ in the space {$\L^2(\1)^d$}
		\begin{align*}
				\sum_{j=1}^n\int_{0}^{t}\bigg(\int_\1\phi_j(\uu_n)\cdot P^n(\bphi)\,\d\x \bigg)^2\d s&\leq C(\bphi)	\sum_{j=1}^{n}\int_{0}^{t}\int_\1|\phi_j(\uu_n)|^2\,\d\x\d s\\&
				\leq C(\bphi,K)\bigg(T+\int_0^{T}\|\uu_n(t)\|^2_{2}\d t\bigg),
				\end{align*}
which is finite {in view of the} uniform estimate \eqref{5.9}.
The rest two mappings $N(\uu_n,\uu_0,\f)_t$ and $M_j(\uu_n)_t$ can be tackled in the {same manner. Therefore} the following random variables have the same laws:
		\begin{align*}
				N(\uu_n,\uu_0,\f) &\sim N(\overline{\uu}_n,\overline{\uu}_0,\overline{\f}),\\
				M(\uu_n)&\sim M(\overline{\uu}_n),\\
				M_j(\uu_n)&\sim M_j(\overline{\uu}_n).
		\end{align*}For any fixed times $s,t\in[0,T]$ with $s<t$, let us define a continuous function by
		\begin{align*}
			h:\mathfrak{V}\big|_{[0,s]}\to [0,1].
		\end{align*} Since
		\begin{align*}
			N(\uu_n,\uu_0,\f)_t=\int_{0}^{t}\int_\1\Phi(\uu_n)\d\W_n(s)\cdot P^n(\bphi)\,\d\x=\sum_{j=1}^{n}\int_{0}^{t}\int_\1\phi_j(\uu_n)\d\beta_j\cdot P^n(\bphi)\,\d\x,
		\end{align*}is a square integrable $\{\mathscr{F}_t\}-$martingale, we conclude that
		\begin{align*}
			\big[N(\uu_n,\uu_0,\f)\big]^2-M(\uu_n),\quad \text{and}\quad N(\uu_n)\beta_j-M_j(\uu_n),
		\end{align*}are $\{\mathscr{F}_t\}-$martingales. Again, consider the restriction of a function to the interval $[0,s]$ which is denoted by $\bfh_s$. Due to the equality of laws, we find
		\begin{align*}			
				\overline{\E}&\bigg[h\big(\bfh_s\overline{\uu}_n,\bfh_s\overline{\W}_n,\bfh_s\overline{\f},\overline{\uu}_0\big) N(\overline{\uu}_n,\overline{\uu}_0,\overline{\f})_{s,t}\bigg]\\&=	\E\bigg[h\big(\bfh_s\uu_n,\bfh_s\W,\bfh_s\f,\uu_0\big) N(\uu_n,\uu_0,\f)_{s,t}\bigg]=\text{0},\\
				\overline{\E}&\bigg[h\big(\bfh_s\overline{\uu}_n,\bfh_s\overline{\W}_n,\bfh_s\overline{\f},\overline{\uu}_0\big) \big(\big[N(\overline{\uu}_n,\overline{\uu}_0,\overline{\f})\big]_{s,t}^2-M(\overline{\uu}_n)_{s,t}\big)\bigg]\\&=	\E\bigg[h\big(\bfh_s\uu_n,\bfh_s\W,\bfh_s\f,\uu_0\big) \big(\big[N(\uu_n,\uu_0,\f)\big]_{s,t}^2-M(\uu_n)_{s,t}\big)\bigg]=\text{0},\\
				\overline{\E}&\bigg[h\big(\bfh_s\overline{\uu}_n,\bfh_s\overline{\W}_n,\bfh_s\overline{\f},\overline{\uu}_0\big) \big(\big[N(\overline{\uu}_n,\overline{\uu}_0,\overline{\f})\overline{\beta}_j^n\big]_{s,t}-M_j(\overline{\uu}_n)_{s,t}\big)\bigg]\\&=	\E\bigg[h\big(\bfh_s\uu_n,\bfh_s\W,\bfh_s\f,\uu_0\big) \big(\big[N(\uu_n,\uu_0,\f)\beta_j\big]_{s,t}-M_j(\uu_n)_{s,t}\big)\bigg]=\text{0}.
				\end{align*}
Thus, we proved  {\eqref{5.26} and hence} \eqref{5.27}, which implies that on the new probability space $(\overline{\Omega},\overline{\mathscr{F}},\overline{\P})$, we have the equations for $j=1,\ldots,n,$
		\begin{align}\label{5.28}\nonumber
			\int_\1 \d&\overline{\uu}_n\cdot\bpsi_j\,\d\x+\kappa\int_\1\d \nabla \overline{\uu}_n:\nabla \bpsi_j\,\d\x +\nu\int_\1\bfA(\overline{\uu}_n):\bfD(\bpsi_j)\,\d\x\d t+\alpha \int_\1\bfa (\overline{\uu}_n)\cdot \bpsi_j\,\d\x\d t
			\\&=\int_\1 \overline{\uu}_n\otimes\overline{\uu}_n:\nabla\bpsi_j\,\d\x\d t+\int_\1\overline{\f}\cdot\bpsi_j \,\d\x\d t+\int_\1\Phi(\overline{\uu}_n)\d\overline{\W}_{n,n}\cdot\bpsi_j \,\d\x,\\
				\overline{\uu}_n(0)&=P^n\overline{\uu}_0,
				\end{align}
and the following {convergence results}:
		\begin{alignat}{2}
		\label{5.010}
		\overline{\uu}_n&\xrightharpoonup[n\to\infty]{} \overline{\uu},\quad &&\text{in}\quad \L^2\big({\overline{\Omega},\overline{\mathscr{F}},\overline{\P}};\L^\infty(0,T;\bfV)\big),  \\
		\label{5.0011}
		\overline{\uu}_n&\xrightharpoonup[n\to\infty]{} \overline{\uu},\quad &&\text{in}\quad \L^p\big({\overline{\Omega},\overline{\mathscr{F}},\overline{\P}};\L^{p}(0,T;{\W_0^{1,p}(\1)^d}\big),  \\
			\label{5.00011}	\overline{\uu}_n &\xrightharpoonup[n\to\infty]{} \overline{\uu},\quad && \text{in}\quad	\L^q\big({\overline{\Omega},\overline{\mathscr{F}},\overline{\P}};\L^{q}(0,T;{\L^q(\1)^d})\big),	\\
		\label{5.000011}	\bfa(\overline{\uu}_n)&\xrightharpoonup[n\to\infty]{} \bfa(\overline{\uu}),\quad &&\text{in}\quad \L^{q'}\big({\overline{\Omega},\overline{\mathscr{F}},\overline{\P}};\L^{q'}(0,T;{\L^{q'}(\1)^d})\big) \\
		\label{5.0012}
		\overline{\uu}_n\otimes\overline{\uu}_n&\xrightharpoonup[n\to\infty]{} \overline{\uu}\otimes\overline{\uu},\quad &&\text{in}\quad {\L^{\frac{q}{2}}}\big({\overline{\Omega},\overline{\mathscr{F}},\overline{\P}};\L^{\frac{q}{2}}(0,T;\L^{\frac{q}{2}}(\1)^{d\times d})\big),  \\
		\label{5.00012}
		\bfA(	\overline{\uu}_n)&\xrightharpoonup[n\to\infty]{} \overline{\bfS},\quad &&\text{in}\quad \L^{p'}\big({\overline{\Omega},\overline{\mathscr{F}},\overline{\P}};\L^{p'}(0,T;{\L^{p'}(\1)^{d\times d}})\big),   \\
		\label{5.013}
		\bfA(	\overline{\uu}_n)&\xrightharpoonup[n\to\infty]{} \overline{\bfS},\quad &&\text{in}\quad \L^{p'}\big({\overline{\Omega},\overline{\mathscr{F}},\overline{\P}};\L^{p'}(0,T;{\W^{-1,p'}(\1)^{d\times d}})\big),    \\
		\label{5.014}
		\Phi(	\overline{\uu}_n)&\xrightharpoonup[n\to\infty]{} \Phi(\overline{\uu}),\quad &&\text{in}\quad \L^{2}\big({\overline{\Omega},\overline{\mathscr{F}},\overline{\P}};\L^{2}(0,T;\mathcal{L}_2(\bfU;{\L^2(\1)^d}))\big).
	\end{alignat}
Using \eqref{5.21}-\eqref{5.24} and  \eqref{5.010}-\eqref{5.014}, we obtain the following limit equation, for all $t\in[0,T]$,  $\overline{\P}-$a.s.:		\begin{align}\label{5.35}\nonumber
			\int_\1	&\overline{\uu}(t)\cdot \bphi \,\d\x +\kappa\int_\1\nabla\overline{\uu}{(t)}: \nabla\bphi \,\d\x +\nu \int_0^t\int_\1 \overline{\bfS}:\bfD(\bphi)\,\d\x\d s + \alpha \int_0^t\int_\1 \bfa(\overline{\uu})\cdot\bphi\,\d\x\d s
			\\&\nonumber=\int_\1\overline{\uu}_0\cdot \bphi \,\d\x +\kappa\int_\1\nabla\overline{\uu}_0:\nabla\bphi\,\d\x+\int_0^t\int_\1\overline{\uu}\otimes\overline{\uu}:\nabla\bphi \,\d\x\d s+\int_0^t\int_\1\overline{\f}\cdot\bphi \,\d\x\d s\\&\quad+\int_\1\int_0^t\Phi(\overline{\uu})\d\overline{\W}(s)\cdot\bphi\,\d\x,
				\end{align}
				for all $\bphi\in \C_{0,\sigma}^\infty(\1)^d$.
The stochastic terms {need} some justification while passing {to} the limit.
{From \eqref{5.21}, from one hand, and \eqref{3.6a}
		and \eqref{5.22}, on the other, we have}
		\begin{alignat*}{2}			
				\overline{\W}_n &\xrightarrow[n\to \infty]{} \overline{\W},\quad  &&\text{in}\quad  \C([0,T];\bfU_0),\\
				\Phi(\overline{\uu}_n) &\xrightarrow[n\to \infty]{}  \Phi(\overline{\uu}),\quad&& \text{in}\quad  \L^2(0,T;\mathcal{L}_2(\bfU;{\L^2(\1)^d})),
		\end{alignat*}
		in probability. {By \cite[Lemma 2.1]{ADNGRT}, these convergence results imply}
		\begin{align*}
			\int_{0}^{t}\Phi(\overline{\uu}_n)\d \overline{\W}_n(s)\xrightarrow[n\to \infty]{}	\int_{0}^{t}\Phi(\overline{\uu})\d \overline{\W}(s), \quad \text{in}\quad  \L^2(0,T;\L^2(\1)^d),
		\end{align*}
in probability.
Thus, we are able to pass {to} the limit in {the} stochastic term. 
{This concludes the proof of Claim~\ref{claim:1}.}
\end{proof}

The proof of the next claim shall be done by using the theory of monotone operators.
\begin{claim}\label{claim:2}
\begin{align}\label{5.36}
			\overline{\bfS}=\bfA(\overline{\uu}).
		\end{align}
\end{claim}
\begin{proof}[Proof Claim~\ref{claim:2}]
Applying infinite-dimensional It\^o's formula (see \cite[Theorem 2.1]{IGDS} (see Subsection \ref{PU} below for a proper justification) to the process $\|(\I-\kappa\Delta)^{\frac{1}{2}}\overline{\uu}(\cdot)\|^2_{2}$ and using the fact $\int_\1\overline{\uu}\otimes\overline{\uu}:\nabla\overline{\uu} \d x =0$ in \eqref{5.35}, we get $\P-$a.s.,
		\begin{align*}
				&\|\overline{\uu}(t)\|^2_{2}+\kappa{\|\nabla\overline{\uu}(t)\|_{2}^2}+2\nu \int_{0}^{t}\int_\1 \overline{\bfS}:\bfD(\overline{\uu})\,\d\x\d s+\alpha \int_0^t \int_\1 \bfa(\overline{\uu})\cdot\overline{\uu}\,\d\x\d s
				 \\&=\|\overline{\uu}_0\|^2_{2}+\kappa{\|\nabla\overline{\uu}_0\|_{2}^2}+2\int_{0}^{t}\int_\1\overline{\f}\cdot\overline{\uu}\,\d\x\d s+2\int_0^t\int_\1\Phi(\overline{\uu})\d \overline{\W}(s)\cdot\overline{\uu}\,\d\x\\&\quad+\int_0^t\int_\1\d\bigg\langle\int_0^{\cdot}\Phi(\overline{\uu})\d\overline{\W}\bigg\rangle_s\,\d\x,
			\end{align*}for {all} $t\in[0,T]$.
Applying finite-dimensional It\^o's formula {now} to the process $\|(\I-\kappa\Delta)^{\frac{1}{2}}\overline{\uu}_n(\cdot)\|^2_{2}$, {subtracting the former from the later, and then taking expectation to the obtained equation}, we get
		\begin{align*}
				&2\nu\overline{\E}\bigg[\int_0^T\int_\1(\bfA(\overline{\uu}_n)-\bfA(\overline{\uu})):\bfD\big(\overline{\uu}_n-\overline{\uu}\big)\,\d\x\d s\bigg]  \\
&\quad+2\alpha \overline{\E}\bigg[\int_0^T\int_\1(\bfa(\overline{\uu}_n)-\bfa(\overline{\uu}))\cdot(\overline{\uu}_n-\overline{\uu})\,\d\x\d s\bigg]    \\
& = \overline{\E}\bigg[-\big\{\|\overline{\uu}_n(T)\|^2_{2}+\kappa{\|\nabla\overline{\uu}_n(T)\|_{2}^2}\}+\|\overline{\uu}(T)\|^2_{2}+\kappa{\|\nabla\overline{\uu}(T)\|_{2}^2}\\&\qquad+
\|P^n\overline{\uu}_0^n\|^2_{2}+\kappa{\|\nabla P^n\overline{\uu}_0^n\|_{2}^2}-
\|\overline{\uu}_0\|^2_{2}-\kappa{\|\nabla \overline{\uu}_0\|_{2}^2}\bigg] \\
&\quad+2\nu\overline{\E}\bigg[\int_0^T\int_\1(\overline{\bfS}-\bfA(\overline{\uu}_n)):\bfD(\overline{\uu})\,\d\x\d s-\int_0^T\int_\1\bfA(\overline{\uu}):\bfD\big(\overline{\uu}_n-\overline{\uu})\,\d\x\d s\bigg] \\ &\quad+2\alpha\overline{\E}\bigg[\int_0^T\int_\1(\bfa(\overline{\uu})-\bfa(\overline{\uu}_n))\cdot\overline{\uu}\,\d\x\d s-\int_0^T\int_\1\bfa(\overline{\uu})\cdot(\overline{\uu}_n-\overline{\uu})\,\d\x\d s\bigg]				\\
&\quad +\overline{\E}\bigg[2\int_0^T\int_\1\overline{\f}\cdot(\overline{\uu}_n-\overline{\uu})\,\d\x\d s+\int_0^T\int_\1\d\bigg\langle\int_0^{\cdot}\Phi(\overline{\uu}_n)\d\overline{\W}_{n,n}\bigg\rangle_s\,\d\x \\ &\qquad-\int_0^T\int_\1\d\bigg\langle\int_0^{\cdot}\Phi(\overline{\uu})\d\overline{\W}\bigg\rangle_s\,\d\x\bigg].
				\end{align*}
{
By \eqref{5.010} and the lower semicontinuity of the norm, one has}
\begin{equation*}
\liminf\limits_{n\to\infty}\overline{\E}\left[\|\overline{\uu}_n(T)\|^2_{2}+\kappa{\|\nabla\overline{\uu}_n(T)\|_{2}^2}
-\|\overline{\uu}(T)\|^2_{2}-\kappa{\|\nabla\overline{\uu}(T)\|_{2}^2}\right]\geq 0.
\end{equation*}

Using this, together with \eqref{5.0011}, \eqref{5.000011}, \eqref{5.013} and the monotonicity of the operator $\bfa(\cdot)$, we can pass to the limit $n\to\infty$ in the previous equation to show  that
		\begin{align*}
				&2\nu\lim_{n\to\infty}\overline{\E}\bigg[\int_0^T\int_\1 \big(\bfA(\overline{\uu}_n)-\bfA(\overline{\uu})\big):\bfD\big(\overline{\uu}_n-\overline{\uu}\big)\,\d\x\d s\bigg] \\& \leq \lim_{n\to\infty} \overline{\E}\bigg[\int_0^T\int_\1\d\bigg(\bigg\langle\int_0^{\cdot}\Phi(\overline{\uu}_n)\d\overline{\W}_{n,n}\bigg\rangle_s-\bigg\langle\int_0^{\cdot}\Phi(\overline{\uu})\d\overline{\W}\bigg\rangle_s\bigg)\,\d\x\bigg].
				\end{align*}
For the remaining integral, we use	\eqref{5.21} and \eqref{5.22}, together with \eqref{3.6a}, to find that
		\begin{align*}
			\overline{\E}\bigg[\int_0^T\int_\1\d\bigg\langle\int_0^{\cdot}\Phi(\overline{\uu}_n)\d\overline{\W}_{n,n}\bigg\rangle_s\,\d\x\bigg] \xrightarrow[n\to \infty]{} \overline{\E}\bigg[\int_0^T\int_\1\d\bigg\langle\int_0^{\cdot}\Phi(\overline{\uu})\d\overline{\W}\bigg\rangle_s\,\d\x\bigg].
		\end{align*}Finally, we arrive at
		\begin{align*}
			\lim_{n\to\infty}\overline{\E}\bigg[\int_0^T\int_\1\big( \bfA(\overline{\uu}_n)-\bfA(\overline{\uu})\big):\bfD\big(\overline{\uu}_n-\overline{\uu}\big) \d t\bigg]\leq 0.
		\end{align*}
Using {the} monotonicity of the operator $\bfA(\cdot)$, we find (cf. \cite[Eqn. (1.6)]{BLFREST} or \cite[Appendix A]{JW})
		\begin{align*}
			\bfD(\overline{\uu}_n) \xrightarrow[n\to\infty]{} \bfD(\overline{\uu}), \quad \; \overline\P\otimes\lambda^{d+1}-\text{a.e}.
		\end{align*}
{
As a consequence of this, we obtain \eqref{5.36}, which proves Claim~\ref{claim:2}.}
\end{proof}
{
This finishes the proof of Theorem~\ref{thrmUE2}.}
\end{proof}

\begin{corollary}\label{corllary}
{Let the hypothesis of  Theorem \ref{thrmUE2} be verified.
In addition, assume that \eqref{3.4} and \eqref{3.5} hold with $\gamma\geq2$.}
Then there {exists} a  \emph{martingale solution} to the system  \eqref{AS1} such that
	\begin{align*}
	&\hspace{-.1cm}\E\bigg[\sup_{t\in[0,T]} \big\{\|\overline{\uu}(t)\|^2_{2}+\kappa{\|\nabla\overline{\uu}(t)\|_{2}^2}\big\}\bigg]^{\frac{\gamma}{2}}+
C({p,\1})\nu\E\bigg[\int_{0}^{T}{\|\nabla\overline{\uu}(t)\|_{p}^p}\d t\bigg]^{\frac{\gamma}{2}} \\
&\hspace{-.3cm}\nonumber\quad +C\alpha\E\bigg[\int_{0}^{T}\|\overline{\uu}(t)\|^q_{q}\d t\bigg]^{\frac{\gamma}{2}}
\\&\hspace{-.1cm} 	\leq 	C_1\bigg\{\bigg(\frac{1}{\eta_1}+\kappa\bigg)^{\frac{\gamma}{2}}\bigg(\int_{\bfV}\|\overline{\bz}\|_{\bfV}^2\d\Lambda_0(\overline{\bz})\bigg)^{\frac{\gamma}{2}}+\bigg(\int_{\bfL^2(\1_T)}\|\overline{\bfg}\|_{\bfL^2(\1_T)}^2\d \Lambda_{\overline{\f}}(\overline{\bfg})\bigg)^{\frac{\gamma}{2}}+C(\gamma,K,T)\bigg\},
			\end{align*}where the constant $C_1$ is independent of $\alpha$.
\end{corollary}
\begin{proof}Applying infinite-dimensional  It\^o's formula to the process $\|(\I-\kappa\Delta)^{\frac{1}{2}}\overline{\uu}(\cdot)\|^2_{2}$, and
	taking the supremum from $0$ to $T$ in the resultant, next raising to the power $\frac{\gamma}{2}$ and then taking expectation, we find
	\begin{align}\label{5.38}\nonumber
		&	\E\bigg[\sup_{t\in[0,T]}\big\{\|\overline{\uu} (t)\|^2_{2}+\kappa{\|\nabla\overline{\uu}(t)\|_2^2}\big\}+C({p,\1})\nu\int_{0}^{T}{\|\nabla\overline{\uu}(t)\|_{p}^p} \d t+2\alpha\int_0^T\|\overline{\uu}(t)\|^q_{q}\d t
	\bigg]^{\frac{\gamma}{2}} \\& 	\nonumber\leq C(\gamma)\Bigg\{\bigg(\frac{1}{\eta_1}+\kappa\bigg)^{\frac{\gamma}{2}}\E\big[{\|\nabla\overline{\uu}_0\|_{2}^2}\big]^{\frac{\gamma}{2}}
+\E\bigg[\int_{0}^{T}\int_\1\overline{\f}\cdot\overline{\uu}\,\d\x\d t\bigg]^{\frac{\gamma}{2}}\\&\hspace{-.3cm} \qquad+\E\bigg[\sup_{t\in[0,T]}\bigg|\int_0^t\int_\1\Phi(\overline{\uu})\d\overline{\W}(s)\cdot\overline{\uu}\,\d\x\bigg|\bigg]^{\frac{\gamma}{2}}+\E\bigg[\int_{0}^{T}\int_\1\d\bigg\langle\int_0^{\cdot}\Phi(\overline{\uu})\d\overline{\W}\bigg\rangle_s\,\d\x\bigg]^{\frac{\gamma}{2}}\Bigg\}.
			\end{align}
To estimate the second term in the right hand side of \eqref{5.38}, we use the Cauchy-Schwarz and Young's inequalities  as follows:
	\begin{align*}
		\E\bigg[\int_{0}^{T}\int_\1\overline{\f}\cdot\overline{\uu}\,\d\x\d t\bigg]^{\frac{\gamma}{2}} \leq \e\E\bigg[\int_{0}^{T}\|\overline{\uu}(t)\|^2_{2}\d t\bigg]^{\frac{\gamma}{2}}+C(\e,\gamma)\E\bigg[\int_0^T\|\overline{\f}(t)\|^2_{2}\d t\bigg]^{\frac{\gamma}{2}}.
	\end{align*}
Let us consider the penultimate term of the right hand side of \eqref{5.38}.
Using the similar calculations to \eqref{BDG}, we deduce
	\begin{align*}
	&	\E\bigg[\sup_{t\in[0,T]}\bigg|\int_{0}^{t}\int_\1\Phi(\overline{\uu})\d\overline{\W}(s)\cdot\overline{\uu}\,\d\x\bigg|\bigg]^{\frac{\gamma}{2}}\\& \leq \e\E\bigg[\sup_{t\in[0,T]}\|\overline{\uu}(t)\|^2_{2}\bigg]^\frac{\gamma}{2}
+C(\e,\gamma,K)\E\bigg[T+\int_0^{T}{\|\overline{\uu}(t)\|^2_{2}\d t}\bigg]^\frac{\gamma}{2}.
			\end{align*}
Choosing $\e$ small enough and using {the} above estimates in \eqref{5.38}, we find
	\begin{align}\label{5.39}\nonumber
		&\E\bigg[\sup_{t\in[0,T]} \big\{\|\overline{\uu}(t)\|^2_{2}+2\kappa{\|\nabla\overline{\uu}(t)\|_{2}^2}\big\}\bigg]^{\frac{\gamma}{2}}
+C({p,\1})\nu\E\bigg[\int_{0}^{T}{\|\nabla\overline{\uu}(t)\|_{p}^p}\d t\bigg]^{\frac{\gamma}{2}}\\
&\nonumber\quad +C\alpha\E\bigg[\int_{0}^{T}\|\overline{\uu}(t)\|^q_{q}\d t\bigg]^{\frac{\gamma}{2}} \\& \nonumber\leq C(\gamma)\bigg\{\bigg(\frac{1}{\eta_1}+\kappa\bigg)^{\frac{\gamma}{2}}\E\big[{\|\nabla\overline{\uu}_0\|_{2}^2}\big]^{\frac{\gamma}{2}}
+C(\gamma,\eta_1,K)\E\bigg[\int_{0}^{T}{\|\nabla\overline{\uu}(t)\|_{2}^2}\d t\bigg]^{\frac{\gamma}{2}}\\&\qquad+\E\bigg[\int_0^T\|\overline{\f}(t)\|^2_{2}\d t\bigg]^{\frac{\gamma}{2}}+C(\gamma,K,T)\bigg\}.
			\end{align}
An application of Gronwall's inequality in \eqref{5.39} yields
	\begin{align*}
			&	\E\bigg[\sup_{t\in[0,T]}{\|\nabla\overline{\uu}(t)\|_2^2}\bigg]^{\frac{\gamma}{2}} \\
&\leq
			C_1\bigg\{\bigg(\frac{1}{\eta_1}+\kappa\bigg)^{\frac{\gamma}{2}}\bigg(\int_{\bfV}\|\overline{\bz}\|_{\bfV}^2\d\Lambda_0(\overline{\bz})\bigg)^{\frac{\gamma}{2}}+\bigg(\int_{\bfL^2(\1_T)}\|\overline{\bfg}\|_{\bfL^2(\1_T)}^2\d \Lambda_{\overline{\f}}(\overline{\bfg})\bigg)^{\frac{\gamma}{2}}+C(\gamma,K,T) \bigg\},
		\end{align*}where $C_1:=C_1(\kappa,\gamma,\eta_1,K,T)$.
Collecting the obtained information in \eqref{5.39}, we obtain the required result.
\end{proof}

\subsection{Non-stationary flows}
In this subsection, we prove the main result of this work, that is, the \emph{existence of martingale solutions} to the stochastic problem \eqref{1.1}-\eqref{1.4}.

\begin{proof}[Proof of Theorem \ref{thm:exist}]
The proof of the theorem is lengthy, so we have divided {it into several steps.
We start by approximating the original problem by an auxiliary problem under the conditions of the} previous sections.
By Theorems \ref{thrmUE} and \ref{thrmUE2}, we have the existence of solutions to the approximate system \eqref{AS1}.
Later, we obtain the uniform estimates, {followed by weak convergence of subsequences as a direct application} of the Banach-Alaoglu theorem.
In the second part, we establish some compactness arguments for the solution to {the} approximate system \eqref{AS1}.
We pass to the limit in the viscous term with the help of the monotone operator theory.

\vspace{2mm}
\noindent
\textbf{Step (1):} ({\bf A priori estimate and weak convergence}). We consider the following system:
\begin{equation}\label{5.40}
	\left\{
	\begin{aligned}
		\d(\I-\kappa\Delta)\uu_n(t)&= \big[\operatorname{div}(\bfA(\uu_n(t)))-\operatorname{div}(\uu_n(t)\otimes\uu_n(t))-\frac{1}{n}|\uu_n(t)|^{q-2}\uu_n(t)
		\\&\qquad+\nabla\pi (t)+\f_n(t)\big]\d t+\Phi(t,\uu_n(t))\d\W(t),\\ \uu_n(0)&=\uu_0^n.
		\end{aligned}
	\right.
\end{equation}Using Theorems \ref{thrmUE} and \ref{thrmUE2}, $\alpha=\frac{1}{n}$,
we have the existence of a martingale solution
\begin{align*}
	\big((\Omega,\mathscr{F},\{\mathscr{F}_t\}_{t\geq0},\P),\uu_n,\uu_0^n,\f_n,\W\big)
\end{align*}to \eqref{5.40} with $\uu_n\in\bVcal_{p,q},\;\Lambda_0=\P\circ(\uu_0^n)^{-1}$ and $\Lambda_{\f_n}=\P\circ(\f_n)^{-1}$.
For the sake of writing, we have removed the underline bars.
In view of the aforementioned results, we can write $\P-$a.s.,
\begin{align*}
		&\int_\1\uu_n(t)\cdot\bphi\,\d\x+\kappa\int_\1\nabla\uu_n(t):\nabla\bphi\,\d\x+\nu\int_{0}^{t}\int_\1 \bfA(\uu_n):\bfD(\bphi)\,\d\x\d s\\&\quad+\frac{1}{n}\int_0^t\int_\1 \bfa(\uu_n)\cdot\bphi\,\d\x \d s
		 \\&=\int_\1\uu_0^n\cdot\bphi\,\d\x+\kappa\int_\1\nabla\uu_0^n:\nabla\bphi\,\d\x+\int_{0}^{t}\int_\1\uu_n\otimes\uu_n:\nabla \bphi\,\d\x\d s+\int_{0}^{t}\int_\1\f_n\cdot\bphi \,\d\x\d s\\&\quad+\int_0^t\int_\1\Phi(\uu_n)\d\W(s)\cdot\bphi \,\d\x, \  \text{ for all }\  \bphi\in \bVcal.
	\end{align*} We can choose the probability space independently of $n$.
The same holds for the Wiener process $\W(\cdot)$. On the other hand, Theorem \ref{thrmUE} gives us the uniform estimates for $\uu_n$ in the space
\begin{align*}
	\L^2({\Omega,\mathscr{F},\P};\L^\infty(0,T;\bfV))\cap \L^p({\Omega,\mathscr{F},\P};\L^p(0,T;{\W_0^{1,p}(\1)^d}).
\end{align*}
By Corollary \ref{corllary} and {assumptions \eqref{3.4}-\eqref{3.5} on $\Lambda_0$ and $\Lambda_{\f_n}$, with $\gamma\geq 2$}, we find
\begin{align}\label{5.41}\nonumber
&	\E\bigg[\sup_{t\in(0,T)}\big\{\|\uu_n(t)\|^2_{2}+\kappa{\|\nabla\uu_n(t)\|_2^2}\big\}^{\frac{\gamma}{2}} \bigg]+C({p,\1})\nu\E\bigg[\int_{0}^{T}{\|\nabla\uu_n(t)\|_{p}^p} \d t\bigg]^{\frac{\gamma}{2}}\\
&\quad+\frac{C}{n}\E\bigg[\int_{0}^{T}\|\uu_n(t)\|^q_{q} \d t\bigg]^{\frac{\gamma}{2}}  \leq C_1(\kappa,\gamma,\eta_1,K,T).
\end{align}

\begin{claim}\label{claim:3}
For any $p\geq1$ and
\begin{equation*}
\gamma\geq\max\left\{\frac{pd}{d-2},2+\frac{2p}{d-2}\right\}
\end{equation*}
there holds
\begin{equation}\label{est:E:un2}
\E\bigg[\int_{0}^{T}\|\uu_n(t)\|^{\rho}_{\rho}\d t\bigg]\leq C, \qquad
\mbox{$\displaystyle\rho:=\frac{pd}{d-2}$ if $d\not=2$,\ and  any  $\rho\in[1,\infty)$ if $d=2$.}
\end{equation}
\end{claim}
\begin{proof}[Proof of Claim~\ref{claim:3}]
Assume that $d\not=2$ (for $d=2$ is easier).
For $1\leq p\leq 2$, we can use {Sobolev's} and Hölder's inequalities, together with \eqref{5.41}, so that for $\rho\geq\max\left\{\frac{pd}{d-2},2\right\},$ one has
\begin{equation*}
\begin{split}
\E\bigg[\int_{0}^{T}\|\uu_n(t)\|^{\frac{pd}{d-2}}_{\frac{pd}{d-2}}\d t\bigg]\leq &
C\E\bigg[\int_{0}^{T}\|\nabla\uu_n(t)\|_2^{\frac{pd}{d-2}}\d t\bigg]=
C\E\bigg[\int_{0}^{T}\left[\left(\|\nabla\uu_n(t)\|^{2}_2\right)^{\frac{\gamma}{2}}\right]^{\frac{pd}{\gamma(d-2)}}\d t\bigg] \\
\leq & C(T)\left\{\E\bigg[\sup_{t\in[0,T]}\left\{\|\uu_n(t)\|^{2}_2+\kappa\|\nabla\uu_n(t)\|^{2}_2\right\}^{\frac{\gamma}{2}}\bigg]\right\}^{\frac{pd}{\gamma(d-2)}}\leq C.
\end{split}
\end{equation*}
If $p>2$, we can use the embedding \begin{align*}
	\L^\infty(0,T;\bfV)\cap\L^p(0,T;\W_0^{1,p}(\1)^d)\hookrightarrow \L^{\frac{pd}{d-2}}(0,T;{\L^{\frac{pd}{d-2}}(\1)^d}),
\end{align*}
together with the Sobolev and H\"older's inequalities,  and with \eqref{5.41}, to show that
\begin{equation*}
\begin{split}
& \E\bigg[\int_{0}^{T}\|\uu_n(t)\|^{\frac{pd}{d-2}}_{\frac{pd}{d-2}}\d t\bigg]\leq C\E\bigg[\int_{0}^{T}\|\nabla\uu_n(t)\|_2^{\frac{2p}{d-2}}\|\nabla\uu_n(t)\|_p^{p}\d t\bigg],\quad p\geq 2  \\
& =
C\E\bigg[\int_{0}^{T}\left[\left(\|\nabla\uu_n(t)\|^{2}_2\right)^{\frac{\gamma}{2}}\right]^{\frac{2p}{\gamma(d-2)}}\|\nabla\uu_n(t)\|_p^{p}\d t\bigg]  \\
& \leq
C\left\{\E\bigg[\sup_{t\in[0,T]}\left\{\|\nabla\uu_n(t)\|^{2}_2\right\}^{\frac{\gamma}{2}}\bigg]\right\}^{\frac{2p}{\gamma(d-2)}}
\left\{\E\bigg[\int_0^T\|\nabla\uu_n(t)\|_p^{p}\d t\bigg]^{\frac{\gamma(p-2)}{\gamma(d-2)-2p}}\right\}^{\frac{\gamma(d-2)-2p}{\gamma(d-2)}},\ \gamma\geq \frac{2p}{d-2} \\
& \leq
C(T) \left\{\E\bigg[\sup_{t\in[0,T]}\left\{\|\uu_n(t)\|^{2}_2+\kappa\|\nabla\uu_n(t)\|^{2}_2\right\}^{\frac{\gamma}{2}}\bigg]\right\}^{\frac{2p}{\gamma(d-2)}}
\left\{\E\bigg[\int_0^T\|\nabla\uu_n(t)\|_p^{p}\d t\bigg]^{\frac{\gamma}{2}}\right\}^{\frac{2}{\gamma}}\leq C.
\end{split}
\end{equation*}
The penultimate inequality holds provided $\gamma\geq2+\frac{2p}{d-2}$.
\end{proof}

\begin{claim}\label{claim:4}
 For some $\gamma\geq \frac{2d}{d-2}$, there holds
 \begin{align}\label{est:E:un:div}
	\E\bigg[\int_{0}^{T}\|\uu_n(t)\otimes\uu_n(t)\|^{q_0}_{q_0}\d t+\int_{0}^{T}\|\operatorname{div}(\uu_n(t)\otimes\uu_n(t))\|^{q_0}_{q_0}\d t\bigg]\leq C,
\end{align}
for $1\leq q_0\leq\frac{d}{d-1}$.
\end{claim}

\begin{proof}[Proof of Claim~\ref{claim:4}]
Assume also here that $d\not=2$. By Hölder's and Sobolev's inequalities, along with \eqref{5.41}, one immediately has
\begin{align*}
	\E\bigg[\int_0^T\|\uu_n(t)\otimes \uu_n(t)\|_{q_0}^{q_0}\d t\bigg]&\leq \E\bigg[\int_0^T\| \uu_n(t)\|_{2q_0}^{2q_0}\d t\bigg] \leq C\E\bigg[\int_0^T\|\nabla \uu_n(t)\|_2^{2q_0}\d t\bigg]\\&\leq C T\E\bigg[\sup_{t\in[0,T]}\|\nabla \uu_n(t)\|_{2}^{2q_0}\bigg] \leq C,
\end{align*}for $q_0\leq \frac{d}{d-2}$ and $\gamma\geq 2q_0$.

	In turn, using Hölder's and Sobolev's inequalities, together with \eqref{5.41}, one has
	\begin{align*}
&	\E\bigg[\int_{0}^{T}\|\operatorname{div}(\uu_n(t)\otimes\uu_n(t))\|^{q_0}_{q_0}\d t\bigg] \leq \E\bigg[\int_0^T\||\uu_n(t)|\nabla \uu_n|\|_{q_0}^{q_0}\d t \bigg]\\&\leq \E\bigg[\int_0^T\|\uu_n(t)\|_{\frac{2q_0}{2-q_0}}^{q_0}\|\nabla \uu_n(t)\|_2^{q_0}\d t\bigg]
\\&\leq C\E\bigg[\int_0^T\|\nabla \uu_n(t)\|_{2}^{2q_0}\d t\bigg] \qquad
\mbox{for}\ 1\leq q_0\leq \frac{d}{d-1}\\& \leq CT\E\bigg[\sup_{t\in[0,T]}\|\nabla \uu_n(t)\|_2^{2q_0}\bigg]\leq C, \qquad
\mbox{for}\ \gamma \geq 2q_0.
	\end{align*}Combining the above estimates, we prove \eqref{est:E:un:div} for
\begin{equation}\label{Q0}
1\leq q_0\leq\frac{d}{d-1}\ \  \text{ and }\ \ \gamma\geq \frac{2d}{d-2}.
\end{equation}
The case of $d=2$ is easy. The first part follows immediately from \eqref{est:E:un2} and the second part is an easy consequence of Sobolev's inequality.
\end{proof}

Applying the Banach-Alaoglu theorem, we can pass the limit along a subsequence (still denoting by the same index) as follows:
\begin{alignat}{2}
	\label{5.43}
	\uu_n& \xrightharpoonup[n\to\infty]{} \uu,\quad &&\text{in} \quad \L^{\frac{\gamma p}{2}}\big({\Omega,\mathscr{F},\P};\L^{p}(0,T;{\W_0^{1,p}(\1)^d})\big),  \\
	\label{5.44}
	\uu_n& \xrightharpoonup[n\to\infty]{} \uu,\quad &&\text{in}\quad  \L^{\gamma}\big({\Omega,\mathscr{F},\P};\L^\rho(0,T;\bfV)\big), \quad \forall\ {\rho\geq 1}, \\
	 \label{5.044}\frac{1}{n}|\uu_n|^{q-2}\uu_n &\xrightarrow[n\to\infty]{} 0,\quad &&\text{in}\quad
\L^{\frac{\gamma q'}{2}}({\Omega,\mathscr{F},\P};\L^{q'}(0,T;{\L^{q'}(\1)^d})), \\	
	\label{5.45}
	\uu_n \otimes\uu_n& \xrightharpoonup[n\to\infty]{} \w, \quad &&\text{in}\quad  \L^{q_0}\big({\Omega,\mathscr{F},\P};\L^{q_0}(0,T;{\W^{1,q_0}(\1)^{d\times d}}),   \\
	\label{5.46}
	\bfA(\uu_n)& \xrightharpoonup[n\to\infty]{} \bfS,\quad &&\text{in}\quad  {\L^{p'}}\big({\Omega,\mathscr{F},\P};\L^{p'}(0,T;{\L^{p'}(\1)^{d\times d}})\big),  \\
	\label{5.47}
	\bfA(\uu_n)& \xrightharpoonup[n\to\infty]{} \bfS,\quad &&\text{in}\quad  \L^{p'}\big({\Omega,\mathscr{F},\P};\L^{p'}(0,T;{\W^{-1,p'}(\1)^{d\times d}})\big),  \\
	\label{5.48}
	\Phi(\uu_n)& \xrightharpoonup[n\to\infty]{} \widehat{\Phi},\quad &&\text{in}\quad  \L^{2}\big({\Omega,\mathscr{F},\P};\L^\rho(0,T;\mathcal{L}_2(\bfU,{\L^2(\1)^d}))\big),\quad \forall \ {\rho\geq 1}.
\end{alignat}
Moreover, we have
\begin{align*}	
		\uu&\in \L^{\gamma}({\Omega,\mathscr{F},\P};\L^\infty(0,T;\bfV)),\\
		\widehat{\Phi}&\in \L^\gamma({\Omega,\mathscr{F},\P};\L^\infty(0,T;\mathcal{L}_2(\bfU,{\L^2(\1)^d}))).
\end{align*}
Now, for the reconstruction of the pressure term, we set
\begin{align*}
		{\bcH}_1^n&:= \bfA(\uu_n),\\
		{\bcH}_2^n&:=\uu_n\otimes\uu_n+\nabla\Delta^{-1}\f_n+\nabla\Delta^{-1}\left(\frac{1}{n}|\uu_n|^{q-2}\uu_n\right)
		,\\
		\Phi^n&:=\Phi(\uu_n).
\end{align*}Applying Theorem \ref{pr:th:1}, Corollaries \ref{p:dc:cor:2} and \ref{p:dc:cor:3}, we get functions $\pi_h^n,\pi_1^n$ and $\pi_2^n$ adapted to $\{\overline{\mathscr{F}}_t\}_{t\in[0,T]}$ and $\Phi_\pi^n$ progressively measurable such that $\P-$a.s.,
\begin{align}\label{5.49}\nonumber
		&\int_\1(\uu_n-\nabla \pi_h^n)(t)\cdot \bphi\,\d\x +\kappa \int_\1\nabla\uu_n(t):\nabla\bphi\,\d\x \\&\nonumber=\int_\1\uu_0^n\cdot \bphi\,\d\x +\kappa \int_\1\nabla\uu_0^n:\nabla\bphi\,\d\x -\nu \int_{0}^{t}\int_\1({\bcH}_1^n-\pi_1^n\I):\nabla \bphi \,\d\x\d s\\&\quad+\int_{0}^{t}\int_\1\operatorname{div} \big({\bcH}_2^n-\pi_2^n\I \big): \bphi\,\d\x\d s +\int_0^t\int_\1\Phi^n \d\W(s)\cdot\bphi \,\d\x+\int_0^t\int_\1\Phi_\pi^n \d\W(s)\cdot \bphi \,\d\x.
\end{align}
Hence, {the following functions are uniformly bounded, with respect to $n$, in the below mentioned spaces}
\begin{alignat}{2}
	\label{5.50}
	{\bcH}_1^n &\in \L^{\frac{\gamma p'}{2}}({\Omega,\mathscr{F},\P};\L^{p'}(0,T;{\L^{p'}(\1)^{d\times d}})),\\ \label{5.51}
	{\bcH}_2^n &\in \L^{q_0}({\Omega,\mathscr{F},\P};\L^{q_0}(0,T;{\W^{1,q_0}(\1)^{d\times d}})),\\
	\label{5.52}
	\Phi^n&\in \L^\gamma ({\Omega,\mathscr{F},\P};\L^\infty(0,T;\mathcal{L}_2(\bfU,{\L^2(\1)^d})),
\end{alignat}where we have used the continuity of $\nabla\Delta^{-1}$ from $\L^{q_0}(\1)^d$ to $\W^{1,q_0}(\1)^{d\times d}$.
We know that the {corresponding pressure functions are also uniformly bounded in the scalar spaces.}
That is
\begin{alignat}{2}
	\label{5.53}
	\pi_h^n &\in \L^\gamma ({\Omega,\mathscr{F},\P};\L^\infty(0,T;\L^2(\1))),\\
	\label{5.54}
	\pi_1^n &\in \L^{\frac{\gamma p'}{2}}({\Omega,\mathscr{F},\P};\L^{p'}(0,T;\L^{p'}(\1))),\\
	\label{5.55}
	\pi_2^n& \in \L^{q_0}({\Omega,\mathscr{F},\P};\L^{q_0}(0,T;\W^{1,q_0}(\1))),\\
	\label{5.56}
	\Phi_\pi^n &\in \L^\gamma ({\Omega,\mathscr{F},\P};\L^\infty(0,T;\mathcal{L}_2(\bfU,{\L^2(\1)^d})),
\end{alignat}
where we have applied Corollaries \ref{p:dc:cor:2} and \ref{p:dc:cor:3}.
Using the regularity theory for the harmonic functions and Corollary \ref{p:dc:cor:3} to the harmonic pressure term, we find {that}
\begin{align}\label{5.57}
	\pi_h^n\in \L^\gamma ({\Omega,\mathscr{F},\P};\L^\rho(0,T;\W^{k,\infty}(\1))), \quad {\forall\ k\in\N,\quad \forall\ \rho\geq 1}.
\end{align}
Passing $n\to\infty$ along subsequences, we  obtain the following convergence results:
\begin{alignat}{2}
	\label{5.58}
	\pi_h^n &\xrightharpoonup[n\to\infty]{}  \pi_h, \quad &&\text{in}\quad \L^\gamma ({\Omega,\mathscr{F},\P};\L^\rho(0,T;\W^{k,\rho}(\1)), \quad {\forall \rho\geq 1},\\
	\label{5.59}
	\pi_1^n &\xrightharpoonup[n\to\infty]{}\pi_1, \quad &&\text{in}\quad \L^{\frac{\gamma p'}{2}}({\Omega,\mathscr{F},\P};\L^{p'}(0,T;\L^{p'}(\1))),\\
	\label{5.60}
	\pi_2^n&\xrightharpoonup[n\to\infty]{}   \pi_2, \quad&&  \text{in}\quad \L^{q_0}({\Omega,\mathscr{F},\P};\L^{q_0}(0,T;\W^{1,q_0}(\1))),\\
	\label{5.61}
	\Phi_\pi^n &\xrightharpoonup[n\to\infty]{}       \Phi_\pi,     \quad&&\text{in}\quad \L^\gamma ({\Omega,\mathscr{F},\P};\L^\rho(0,T;\mathcal{L}_2(\bfU,{\L^2(\1)^d})), \quad {\forall \rho\geq 1}.
\end{alignat}

Now, our goal is to show that in \eqref{5.45} and \eqref{5.48}, we have $\w=\uu\otimes\uu$ and $\widehat{\Phi} =\Phi(\uu)$.
{To prove this,} we use the compactness arguments and a version of Skorokhod's theorem (see \cite[Theorem 2]{AJ}) which help us in the construction of a new probability space, similar to the proof of Theorem \ref{thrmUE2}.
In the final part, we will prove {that} $\bfS = \bfA(\uu)$.

\vspace{2mm}
\noindent
\textbf{Step (2):} ({\bf Compactness}).
In this step, we prove the compactness of the sequence $\{\uu_n\}_{n\in\N}$. {Here we follow the approach used in \cite[Section 4]{MH} (see also \cite[Section 5]{Breit})}. {As the pressure needs to be included in the compact method, we have to work with weak convergence results.} {However, in this case, we cannot apply the classical Skorokhod theorem} due to the presence of the non-metric space (see \eqref{NMS} below). Even though there is a {generalization of this result that includes weak topologies in Banach spaces: the Jakubowski-Skorokhod theorem} (see \cite[Theorem 2]{AJ}), {which can be applied}  to quasi-Polish {spaces (see \cite{BreitMH})}.

From \eqref{5.40}-\eqref{est:E:un2}, we compute that
\begin{align*}
	\E\bigg[\bigg\|(\I-\kappa\Delta)\uu_n(t)-\int_{0}^{t}\Phi(\uu_n)\d\W(s)\bigg\|_{\W^{1,q_0}(0,T;{\W_{\sigma}^{-1,q_0}(\1)^d})}\bigg] \leq C.
\end{align*} For the stochastic integral term we use \eqref{STphi1} (see \cite[Lemma 4.6]{MH}), for some $\mu>0$
\begin{align*}
	\E\bigg[\bigg\|\int_{0}^{t}\Phi(\uu_n)\d\W(s)\bigg\|_{\mathrm{C}^{\mu}([0,T];{\L^2(\1)^d})}\bigg] \leq C,
\end{align*}which is a consequence of \eqref{5.41}, and the assumption \eqref{3.6a}. Using the above information and similar arguments to \eqref{5.19} and \eqref{5.019}, we find
\begin{align}\label{5.63}
	\E\bigg[\|\uu_n\|_{\W^{\eta,q_0}(0,T;{\W_{\sigma}^{1,q_0}(\1)^d})}\bigg]\leq C,\ \mbox{  for some  }\ \eta>0.
\end{align}
Using the embedding $ \W_{\sigma}^{1,q_0}(\1)^d\hookrightarrow\L_{\sigma}^{q_0}(\1)^d$, for any $q_0\in[1,\infty)$, we obtain
\begin{align}\label{5.64}
	\E\bigg[\|\uu_n\|_{\W^{\eta,q_0}(0,T;{\L_{\sigma}^{q_0}(\1)^d})}\bigg]\leq C.
\end{align}
Again, by \cite[Theorem 2.1]{FFDG}, for $p>\frac{2d}{d+2}$, we obtain the following compact embedding:
\begin{align*}
	\W^{\eta,q_0}(0,T;\L_{\sigma}^{q_0}(\1)^d)\cap \L^p(0,T;\W_{0,\sigma}^{1,p}(\1)^d) \hookrightarrow\hookrightarrow \L^\rho(0,T;\L_{\sigma}^\rho(\1)^d),
\end{align*}
for all $\rho<\min\big\{\frac{p(d+2)}{d},\frac{2d}{d-2}\big\}$.
Note that $\rho<\min\left\{\frac{p(d+2)}{d},p^\ast,2^\ast\right\}$ and for $p>\frac{2d}{d+2}$ we know that $\frac{p(d+2)}{d}<p^\ast$.

We use the above compact embedding for {the} compactness of $\{\uu_n\}_{n\in\N}$.
One can argue for {the} harmonic pressure  term  $\{\pi_h^n\}_{n\in\N}$ in a similar manner. Reasoning as in \cite[Eqn. (4.24)]{JW}, we can combine this with the regularity theory for harmonic functions and with the Lebesgue dominated convergence theorem to prove that the following compact embedding holds:
\begin{align*}
	\L^\infty(0,T;\L^2(\1))\cap \{\Delta u(t)=0, \text{ for a.e. } t \} {\hookrightarrow\hookrightarrow}  {\L^\rho(0,T;\L^\rho(\1))}.
\end{align*}

Now, we define a path space
\begin{align}\label{NMS}\nonumber
		\mathfrak{V}&:= {\L^\rho(0,T;\L_{\sigma}^\rho(\1)^d)\otimes \L^\rho(0,T;\L^\rho(\1))\otimes \L_{\rm w}^{p'}(0,T;\L^{p'}(\1))\otimes \L_{\rm w}^{q_0}(0,T;\W^{1,q_0}(\1))}\\&\qquad\otimes {\L_{\rm w}^\rho(0,T;\mathcal{L}_2(\bfU,\L^2(\1)^d))} \otimes \mathrm{C}([0,T];\bfU_0)\otimes \bfV\otimes \bfL^2(\1_T),	
\end{align}
where $\rm w$ denotes the weak topology.
The following notations are to be used next:
\begin{enumerate}
	\item $\varrho_{\uu_n}$ denotes the law of $\uu_n$ on ${\L^\rho(0,T;\L_{\sigma}^\rho(\1)^d)}$,
	\item $\varrho_{\pi_h^n}$ denotes the law of $\pi_h^n$ on ${\L^\rho(0,T;\L^\rho(\1))}$,
	\item $\varrho_{\pi_1^n}$ denotes the law of $\pi_1^n$ on $\L_{\rm w}^{p'}(0,T;\L^{p'}(\1))$,
	\item $\varrho_{\pi_2^n}$ denotes the law of $\pi_2^n$ on $\L_{\rm w}^{q_0}(0,T;\W^{1,q_0}(\1))$,
	\item $\varrho_{\Phi_\pi^n}$ denotes the law of $\Phi_\pi^n$ on $\L_{\rm w}^\rho(0,T;\mathcal{L}_2(\bfU,\L^2(
	\1)^d))$,
	\item $\varrho_{\W_n}$ denotes the law of $\W_n$ on $\mathrm{C}([0,T];\bfU_0)$,
	\item $\varrho_n$ denotes the joint law of $\uu_n,\pi_h^n,\pi_1^n,\pi_2^n,\Phi_\pi^n,\W_n,\uu_0^n,\f_n $ on $\mathfrak{V}$.
\end{enumerate}
Now, our main goal is to prove the tightness of the measure $\varrho_n$.
For that, we first consider the ball $B_N$ in the space $\W^{\eta,q_0}(0,T;{\L_{\sigma}^{q_0}(\1)^d})\cap \L^p(0,T;{\W_{\sigma}^{1,p}(\1)^d})$ and the complement of this ball is denoted by $B_N^c$. Thus, we have
\begin{align*}
	\varrho_{\uu_n}(B_N^c)&=\P \left(\|\uu_n\|_{\W_{\sigma}^{\eta,q_0}(0,T;{\L_{\sigma}^{q_0}(\1)^d})}+\|\uu_n\|_{\L^p(0,T;{\W_{0,\sigma}^{1,p}(\1)^d})}\geq N\right) \\& \leq \frac{1}{N}\E \left[\|\uu_n\|_{\W_{\sigma}^{\eta,q_0}(0,T;{\L_{\sigma}^{q_0}(\1)^d})}+\|\uu_n\|_{\L^p(0,T;{\W_{0,\sigma}^{1,p}(\1)^d})}\right] \leq \frac{C}{N},
\end{align*}
where we have used \eqref{5.41} and \eqref{5.64}. Thus, for any fixed $\xi>0$, we can find $N(\xi)$ such that \begin{align*}
	\varrho_{\uu_n}(B_{N(\xi)}) \geq 1-\frac{\xi}{8}.
\end{align*}
Using \eqref{5.57}, we can also prove that the law of $\pi_h^n$ is tight, that is, there exists a compact subset $K_\pi\subset {\L^\rho(0,T;\L^\rho(\1))}$ such that $\varrho_{\pi_h^n}(K_\pi) \geq 1-\frac{\xi}{8}$. As we know that our spaces are reflexive, therefore we can find compact sets for $\pi_1^n,\pi_2^n$ and $\Phi_\pi^n$ with measures $\geq 1-\frac{\xi}{8}$. And the law $\varrho_\W$ is tight as it is the same with the law of Brownian motion $\W$ which is a Radon measure on the space $\C([0,T];\bfU_0)$. So, there exists another compact set $K_\xi \subset \mathrm{C}([0,T];\bfU_0)$ with $\varrho_{\W_n} (K_\xi)\geq 1-\frac{\xi}{8}$. Using the similar arguments, we can find compact subsets of $\bfV$ and $\bfL^2(\1_T)$ with their measures $\big(\Lambda_0$ and $\Lambda_{\f_n}\big)$ $\geq 1-\frac{\xi}{8}$. Therefore, there exists a compact subset $\mathfrak{V}_\xi\subset \mathfrak{V}$ with $\varrho_n(\mathfrak{V}_\xi)\geq 1-\xi$. Hence, $\{\varrho_n\}_{n\in\N}$ is tight in the same space. Applying the Jakubowski version of Skorokhod Theorem (see \cite[Theorem 2]{AJ}), we ensure the existence of a probability space $(\overline{\Omega},\overline{\mathscr{F}},\overline{\P})$, and a sequence of random variables $(\overline{\uu}_n,\overline{\pi}_h^n,\overline{\pi}_1^n,\overline{\pi}_2^n,\overline{\Phi}_\pi^n,\overline{\W}_n,\overline{\uu}_0^n,\overline{\f}_n)$ and a random variable $(\overline{\uu},\overline{\pi}_h,\overline{\pi}_1,\overline{\pi}_2,\overline{\Phi}_\pi,\overline{\W},\overline{\uu}_0,\overline{\f})$  defined on the probability space $(\overline{\Omega},\overline{\mathscr{F}},\overline{\P})$ taking values in the space $\mathfrak{V}$ (see \eqref{NMS}) such that the following holds:
\begin{enumerate}
	\item The laws of $(\overline{\uu}_n,\overline{\pi}_h^n,\overline{\pi}_1^n,\overline{\pi}_2^n,\overline{\Phi}_\pi^n,\overline{\W}_n,\overline{\uu}_0^n,\overline{\f}_n)$ and $(\overline{\uu},\overline{\pi}_h,\overline{\pi}_1,\overline{\pi}_2,\overline{\Phi}_\pi,\overline{\W},\overline{\uu}_0,\overline{\f})$ are the same under $\overline{\P}$ and coincide with $\varrho_n$ and $\varrho:=\lim\limits_{n\to\infty}\varrho_n$, respectively;
	\item The following weak convergence holds
\begin{alignat*}{2}
			\overline{\pi}_1^n  &\xrightharpoonup[n\to\infty]{} \overline{\pi}_1, \quad &&\text{in}\quad \L^{p'}(0,T;\L^{p'}(\1)),\\
			\overline{\pi}_2^n  &\xrightharpoonup[n\to\infty]{} \overline{\pi}_2,  \quad &&\text{in}\quad \L^{q_0}(0,T;\W^{1,q_0}(\1)),\\
			\overline{\Phi}_\pi^n &\xrightharpoonup[n\to\infty]{} 	\overline{\Phi}_\pi,\quad && \text{in}\quad \L^\rho(0,T;\mathcal{L}_2(\bfU,{\L^2(\1)^d})),
		\end{alignat*}$	\overline{\P}-$a.s.;
	\item The following strong convergence holds
\begin{alignat*}{2}
			\overline{\uu}_n &\xrightarrow[n\to\infty]{} \overline{\uu},\quad&&\text{in}\quad {\L^\rho(0,T;\L_{\sigma}^\rho(\1)^d)},\\
			\overline{\pi}_h^n &\xrightarrow[n\to\infty]{} \overline{\pi}_h, \quad&&\text{in}\quad {\L^\rho(0,T;\L^\rho(\1))},\\
			\overline{\W}_n &\xrightarrow[n\to\infty]{}	\overline{\W},\quad&& \text{in}\quad \mathrm{C}([0,T];\bfU_0),\\
			\overline{\uu}_0^n &\xrightarrow[n\to\infty]{} 	\overline{\uu}_0,\quad &&\text{in}\quad \bfV,\\
			\overline{\f}_n &\xrightarrow[n\to\infty]{} 	\overline{\f},\quad &&\text{in}\quad \L^2(0,T;{\L^2(\1)^d}),
			\end{alignat*}$\overline{\P}-$a.s.;
	\item For all $\beta<\infty$, we have
	\begin{align*}
	\overline{\E}\bigg[\sup_{t\in[0,T]} \|\overline{\W}_n(t)\|_{\bfU_0}^\beta\bigg] = 	\E\bigg[\sup_{t\in[0,T]} \|\overline{\W}(t)\|_{\bfU_0}^\beta\bigg].
	\end{align*}
\end{enumerate}
Using the equivalency in distributions, we find the following weak convergence results:
\begin{alignat*}{2}	
		\overline{\pi}_1^n  &\xrightharpoonup[n\to\infty]{} \overline{\pi}_1,\quad&&  \text{in}\quad \L^{p'}({\overline{\Omega},\overline{\mathscr{F}},\overline{\P}};\L^{p'}(0,T;\L^{p'}(\1))),\\
		\overline{\pi}_2^n  &\xrightharpoonup[n\to\infty]{} \overline{\pi}_2, \quad&& \text{in}\quad \L^{q_0}({\overline{\Omega},\overline{\mathscr{F}},\overline{\P}};\L^{q_0}(0,T;\W^{1,q_0}(\1))),\\
		\overline{\Phi}_\pi^n &\xrightharpoonup[n\to\infty]{} 	\overline{\Phi}_\pi, \quad&&\text{in}\quad \L^{\gamma}({\overline{\Omega},\overline{\mathscr{F}},\overline{\P}};\L^\rho(0,T;\mathcal{L}_2(\bfU,{\L^2(\1)^d}))).
	\end{alignat*}
	Using Vitali's convergence theorem along subsequences, we obtain the following strong convergence results:
\begin{alignat}{2}
	\label{5.65}
	\overline{\W}_n &\xrightarrow[n\to\infty]{}  \overline{\W}, \quad&&\text{in}\quad \L^2({\overline{\Omega},\overline{\mathscr{F}},\overline{\P}};\mathrm{C}([0,T];\bfU_0)),\\
	\label{5.66}
	\overline{\uu}_n &\xrightarrow[n\to\infty]{} \overline{\uu},\quad&&\text{in}\quad {\L^\rho(\overline{\Omega},\overline{\mathscr{F}},\overline{\P};\L^\rho(0,T;\L_{\sigma}^\rho(\1)^d))},\\
	\label{5.67}
	\nabla^j \overline{\pi}_h^n&\xrightarrow[n\to\infty]{} \nabla^j \overline{\pi}_h,\quad&& \text{in}\quad {\L^\rho(\overline{\Omega},\overline{\mathscr{F}},\overline{\P};\L^\rho(0,T;\L^\rho(\1)))},\\
	\label{5.68}
	\overline{\uu}_0^n &\xrightarrow[n\to\infty]{} \overline{\uu}_0,\quad&&\text{in}\quad \L^2({\overline{\Omega},\overline{\mathscr{F}},\overline{\P}};\bfV),\\
	\label{5.69}
	\overline{\f}_n &\xrightarrow[n\to\infty]{} \overline{\f}, \quad&&\text{in}\quad \L^2({\overline{\Omega},\overline{\mathscr{F}},\overline{\P}};\L^2(0,T;{\L^2(\1)^d})),
\end{alignat}for all $\rho <\min\big\{ \frac{p(d+2)}{d},\frac{2d}{d-2}\big\}$.
{In the harmonic pressure term, we have used the regularity theory for harmonic maps.
As a consequence,  for any $\beta<\infty$,} we have
\begin{alignat}{2}
	\label{5.70}
	\overline{\uu}_n\otimes\overline{\uu}_n & \xrightharpoonup[n\to\infty]{} \overline{\uu}\otimes \overline{\uu},\quad&& \text{in}\quad  \L^{q_0}({\overline{\Omega},\overline{\mathscr{F}},\overline{\P}};\L^{q_0}(0,T;{\W^{1,q_0}(\1)^{d\times d}}))   ,\\
	\label{5.71}
	\Phi(\overline{\uu}_n) & \xrightharpoonup[n\to\infty]{} 	\Phi(\overline{\uu}) ,\quad&& \text{in}\quad \L^{\gamma}({\overline{\Omega},\overline{\mathscr{F}},\overline{\P}};\L^\beta(0,T;\mathcal{L}_2(\bfU,{\L^2(\1)^d})))   ,\\
	\label{5.72}
	\Phi_\pi(\overline{\uu}_n)& \xrightharpoonup[n\to\infty]{} 	\Phi_\pi(\overline{\uu}),\quad &&\text{in}\quad  \L^{\gamma}({\overline{\Omega},\overline{\mathscr{F}},\overline{\P}};\L^\beta(0,T;\mathcal{L}_2(\bfU,{\L^2(\1)^d}))) .
\end{alignat}
Now, we define
the $\overline{\P}-$augmented canonical filtration, {which is} denoted by $\{\overline{\mathscr{F}}_t\}_{t\geq 0}$, of the process $(\overline{\uu},\overline{\pi}_h,\overline{\pi}_1,\overline{\pi}_2,\overline{\Phi},\overline{\W},\overline{\f})$, that is,
\begin{align*}
	\overline{\mathscr{F}}_t = \sigma \big(\sigma(\bfh_t\overline{\uu},\bfh_t\overline{\pi}_h,\bfh_t\overline{\pi}_1,\bfh_t\overline{\pi}_2,\bfh_t\overline{\Phi}_\pi,\bfh_t\overline{\W},\bfh_t\overline{\f})\cup \{M\in\overline{\mathscr{F}};\overline{\P}(M)=0\}\big), \; t\in[0,T].
\end{align*}
{Proceeding as in the proof of Theorem \ref{thrmUE2}, choosing in the present case (non-divergence free) test functions from $\mathrm{C}_0^\infty(\1)^d$, we can show that the following equation holds on the new probability space.
That is,} we have $\P-$a.s.,
\begin{align}\label{5.0073}\nonumber	
		&	\int_\1(\overline{\uu}_n-\nabla \overline{\pi}_h^n)(t)\cdot\bphi \,\d\x+\kappa\int_\1\nabla\overline{\uu}_n(t):\nabla\bphi \,\d\x\\&\nonumber	= \int_\1\overline{\uu}_0^n\cdot\bphi \,\d\x+\kappa\int_\1\nabla\overline{\uu}_0^n:\nabla\bphi\,\d\x -\nu\int_{0}^{t}\int_\1\big(\overline{{\bcH}}_1^n-\overline{\pi}_1^n\I\big):\nabla\bphi \,\d\x\d s\\& \quad + \int_{0}^{t}\int_\1\operatorname{div}\big(\overline{{\bcH}}_2^n-\overline{\pi}_2^n\I\big)\cdot\bphi \,\d\x\d s  +\int_0^t\int_\1 \Phi(\overline{\uu}_n)\d\overline{\W}_n(s)\cdot\bphi \,\d\x +\int_0^t\int_\1\overline{\Phi}_\pi^n\d\overline{\W}_n(s)\cdot\bphi \,\d\x,
	\end{align} for all $t\in[0,T]$, and $\bphi \in {\mathrm{C}_0^\infty(\1)^d}$, where we {have} set \begin{align*}\overline{{\bcH}}_1^n&:=\bfA(\overline{\uu}_n),\\ \overline{{\bcH}}_2^n&:=\overline{\uu}_n\otimes \overline{\uu}_n+\nabla \Delta^{-1}\overline{\f}_n+\nabla\Delta^{-1}\left(\frac{1}{n}|\overline{\uu}_n|^{q-2}\overline{\uu}_n\right).
	\end{align*}
Passing $n \to \infty$, using the above convergence results, and \cite[Lemma 2.1]{ADNGRT} for the convergence of the stochastic integral term, we arrive at $\overline{\P}-$a.s.,
\begin{align}\label{5.73}	\nonumber
		& \int_\1(\overline{\uu}-\nabla \overline{\pi}_h)(t)\cdot\bphi \,\d\x+\kappa\int_\1\nabla\overline{\uu}(t):\nabla\bphi \,\d\x \\&\nonumber= \int_\1\overline{\uu}_0\cdot\bphi \,\d\x+\kappa\int_\1\nabla\overline{\uu}_0:\nabla\bphi \,\d\x -\nu\int_{0}^{t}\int_\1\big(\overline{{\bcH}}_1-\overline{\pi}_1\I\big):\nabla\bphi\,\d\x\d s\\& \quad+\int_{0}^{t}\int_\1\operatorname{div}\big(\overline{{\bcH}}_2-\overline{\pi}_2\I\big)\cdot\bphi\,\d\x\d s  +\int_0^t\int_\1 \Phi(\overline{\uu})\d\overline{\W}(s)\cdot\bphi\,\d\x+\int_0^t\int_\1\overline{\Phi}_\pi \d\overline{\W}(s)\cdot\bphi \,\d\x,
\end{align} {for all $\bphi \in \mathrm{C}_0^\infty(\1)^d$}, where $\overline{{\bcH}}_1:=\overline{\bfS}$ and $\overline{{\bcH}}_2:=\overline{\uu}\otimes \overline{\uu}+\nabla \Delta^{-1}\overline{\f}$.

 Now, our aim is to show that $\overline{\bfS}=\bfA(\overline{\uu})$. Let us set
\begin{align*}	
		\overline{\mathtt{G}}_1^n&:=\bfA(\overline{\uu}_n)-\overline{\bfS},\\
		\overline{\mathtt{G}}_2^n&:=\overline{\uu}_n\otimes \overline{\uu}_n-\overline{\uu}\otimes\overline{\uu}+ \nabla \Delta^{-1} (\overline{\f}_n-\overline{\f})+\nabla\Delta^{-1}\bigg(\frac{1}{n}|\overline{\uu}_n|^{q-2}\overline{\uu}_n\bigg),\\
		\overline{\theta}_h^n&:=\overline{\pi}_h^n-\overline{\pi}_h,\; \overline{\theta}_1^n:=\overline{\pi}_1^n-\overline{\pi}_1,\; \overline{\theta}_2^n:=\overline{\pi}_2^n-\overline{\pi}_2.
	\end{align*}
The following {convergence results hold true}:
\begin{alignat}{2}
	\label{5.74}
	\overline{\uu}_n -\overline{\uu}	&	 \xrightharpoonup[n\to\infty]{} 0, \quad&&\text{in}\quad
\L^{\frac{\gamma p}{2}}({\overline{\Omega},\overline{\mathscr{F}},\overline{\P}};\L^p(0,T;{\W_0^{1,p}(\1)^d})),\\
	\label{5.75}
	\overline{\uu}_n -\overline{\uu}	&	 \xrightharpoonup[n\to\infty]{} 0,\quad&& \text{in}\quad \L^\gamma({\overline{\Omega},\overline{\mathscr{F}},\overline{\P}};{\L^\rho(0,T;\bfV)}), \ \forall\ \rho<\infty,\\
	\label{5.76}
	\overline{\mathtt{G}}_1^n&\xrightharpoonup[n\to\infty]{} 0,\quad&& \text{in}\quad
\L^{\frac{\gamma p'}{2}}({\overline{\Omega},\overline{\mathscr{F}},\overline{\P}};\L^{p'}(0,T;{\L^{p'}(\1)^{d\times d}})),\\
	\label{5.77}
	\overline{\mathtt{G}}_2^n&\xrightharpoonup[n\to\infty]{} 0,\quad&& \text{in}\quad \L^{q_0}({\overline{\Omega},\overline{\mathscr{F}},\overline{\P}};\L^{q_0}(0,T;{\W^{1,q_0}(\1)^{d\times d}})),\\
	\label{5.78}
	\Phi(\overline{\uu}_n)-\Phi(\overline{\uu}) & \xrightharpoonup[n\to\infty]{} 0,\quad&& \text{in}\quad \L^{\gamma}({\overline{\Omega},\overline{\mathscr{F}},\overline{\P}};{\L^\rho}(0,T;\mathcal{L}_2(\bfU,{\L^2(\1)^d}))),\ \forall\ \rho<\infty.
\end{alignat}
For {the} pressure terms, we have the following {convergence results}:
\begin{alignat}{2}
	\label{5.79}
	\overline{\theta}_h^n&\xrightarrow[n\to\infty]{} 0, \quad&&\text{in}\quad \L^\gamma({\overline{\Omega},\overline{\mathscr{F}},\overline{\P}};{\L^\rho(0,T;\W^{j,\rho}(\1))}),\ \forall\ \rho<\infty,\\
	\label{5.80}
	\overline{\theta}_1^n&\xrightharpoonup[n\to\infty]{} 0, \quad&&\text{in}\quad  \L^{\frac{\gamma p'}{2}}({\overline{\Omega},\overline{\mathscr{F}},\overline{\P}};\L^{p'}(0,T;\L^{p'}(\1))),\\
	\label{5.81}
	\overline{\theta}_2^n&\xrightharpoonup[n\to\infty]{} 0, \quad&&\text{in}\quad \L^{q_0}({\overline{\Omega},\overline{\mathscr{F}},\overline{\P}};\L^{q_0}(0,T;\W
	^{1,q_0}(\1))),\\
	\label{5.82}
	\Phi_\pi(\overline{\uu}_n)-\Phi_\pi(\overline{\uu}) &\xrightharpoonup[n\to\infty]{} 0,\quad &&\text{in}\quad \L^{\gamma}({\overline{\Omega},\overline{\mathscr{F}},\overline{\P}};{\L^\rho}(0,T;\mathcal{L}_2(\bfU,{\L^2(\1)^d}))),\ \forall\  \rho<\infty.
\end{alignat}
Moreover, we have
\begin{alignat}{2}
	\label{5.83}
	\overline{\theta}_h^n&\in \L^\gamma({\overline{\Omega},\overline{\mathscr{F}},\overline{\P}};\L^\infty(0,T;\L^2(\1))),\\
	\label{5.84}
\Phi(\overline{\uu}_n),\Phi_\pi(\overline{\uu}_n)  &\in \L^{\gamma}({\overline{\Omega},\overline{\mathscr{F}},\overline{\P}};\L^\infty(0,T;\mathcal{L}_2(\bfU,{\L^2(\1)^d}))),
\end{alignat}uniformly in $n$.
We can rewrite the difference between the approximate equation \eqref{5.0073} and limit equation \eqref{5.73} as $\overline{\P}-$a.s.,
\begin{align}\label{5.86}	\nonumber
		&	\int_\1(\overline{\uu}_n-\overline{\uu}-\nabla \overline{\theta}_h^n )(t)\cdot\bphi\,\d\x+\kappa\int_\1(\nabla(\overline{\uu}_n-\overline{\uu})(t)):\nabla \bphi\,\d\x\\&\nonumber=	 	\int_\1(\overline{\uu}_0^n-\overline{\uu}_0 )\cdot\bphi\,\d\x+\kappa\int_\1(\nabla(\overline{\uu}_0^n-\overline{\uu}_0)):\nabla\bphi\,\d\x  -\nu\int_{0}^{t}\int_\1\big(\overline{\mathtt{G}}_1^n-\overline{\theta}_1^n\I\big):\nabla\bphi \,\d\x\d s\\&\nonumber\quad +\int_{0}^{t}\int_\1\operatorname{div}\big(\overline{\mathtt{G}}_2^n-\overline{\theta}_2^n\I\big)\cdot\bphi\,\d\x \d s +\int_0^t\int_\1\left( \Phi(\overline{\uu}_n)\d\overline{\W}_n(s)-\Phi(\overline{\uu})\d \overline{\W}(s)\right)\cdot\bphi\,\d\x \\&\quad+\int_0^t\int_\1 \left( \Phi_{\pi}(\overline{\uu}_n)\d\overline{\W}_n-\Phi_{\pi}(\overline{\uu}) \d\overline{\W}(s)\right)\cdot\bphi \,\d\x,
	\end{align}
for all $\bphi\in {\mathrm{C}_0^\infty(\1)^d}$.

	We introduce the sequence $\overline{\vv}_n:= \overline{\uu}_n-\nabla \overline{\pi}_h^n$, and the sequence $\overline{\vv}_{n,1}=\overline{\vv}_n-\overline{\vv}$, where $\overline{\vv}=\overline{\uu}-\nabla \overline{\pi}_h,$ for which we have
\begin{alignat}{2}
\label{5.87}
\overline{\vv}_{n,1} &\xrightharpoonup[n\to\infty]{} 0,\quad&& \text{in}\quad \L^p({\overline{\Omega},\overline{\mathscr{F}},\overline{\P}};\L^p(0,T;{\W_0^{1,p}(\1)^d})),\\
\label{5.88}
\overline{\vv}_{n,1} &\xrightarrow[n\to\infty]{} 0, \quad&&\text{in}\quad \L^\rho({\overline{\Omega},\overline{\mathscr{F}},\overline{\P}}; \L^\rho(0,T;\L_{\sigma}^\rho(\1)^d)).
\end{alignat}
Thus, for all $\bphi\in {\mathrm{C}_0^\infty(\1)^d}$, we arrive at $\overline{\P}-$a.s.,
\begin{align}\label{5.89}\nonumber
&\int_\1(\I-\kappa\Delta)\overline{\vv}_{n,1}(t)\cdot \bphi\,\d\x \\&\nonumber=\int_\1\overline{\vv}_0^{n,1}\cdot\bphi\,\d\x+\nu\int_{0}^{t}\int_\1\operatorname{div}  \big(\overline{\mathtt{G}}_1^{n}-\overline{\theta}_1^{n}\I\big)\cdot\bphi\,\d\x\d s\\&\quad\nonumber+\int_{0}^{t}\int_\1\operatorname{div} \big(\overline{\mathtt{G}}_2^{n}-\overline{\theta}_2^{n}\I\big)\cdot\bphi \,\d\x\d s +\int_0^t\int_\1\left( \Phi(\overline{\uu}_n)\d\overline{\W}_n(s)-\Phi(\overline{\uu})\d \overline{\W}(s)\right)\cdot\bphi\,\d\x \\&\quad+\int_0^t\int_\1 \left( \Phi_{\pi}(\overline{\uu}_n)\d\overline{\W}_n-\Phi_{\pi}(\overline{\uu}) \d\overline{\W}(s)\right)\cdot\bphi \,\d\x.
\end{align}Let us prove that if we test the above expression with $\overline{\vv}_{n,1}$ (validity of It\^o's formula), then the right hand side of \eqref{5.89} is well-defined. 
In view of the regularity of $\overline{\vv}_{n,1}$, we only need to work on the third term of the right hand side of \eqref{5.89}.
Assume that $d\neq 2$, as the case of $d=2$ is easier. By Sobolev's embedding, we have \begin{align*}
	\bfV\hookrightarrow\L^{q_0'}(\1)^d,\ \  \text{ for } \ \ q_0' \leq \frac{2d}{d-2}\ \  \text{ or }\ \  q_0\geq \frac{2d}{d+2}.
\end{align*}From \eqref{Q0}, we obtain
\begin{align*}
	\frac{2d}{d+2} \leq q_0\leq \frac{d}{d-1}, \quad \mbox{for} \quad d\leq 4.
\end{align*}Combining the above facts, we conclude that
\begin{align}\label{4110}
	&\overline{\E}\bigg[\int_{0}^{t}\int_\1\operatorname{div} \big(\overline{\mathtt{G}}_2^{n}-\overline{\theta}_2^{n}\I\big)\cdot\overline{\vv}_{n,1} \,\d\x\d s \bigg] \nonumber\\&\leq \overline{\E}\bigg[\int_0^t\int_{\1}\big\|\operatorname{div} \big(\overline{\mathtt{G}}_2^{n}-\overline{\theta}_2^{n}\big)\big\|_{q_0}\| \overline{\vv}_{n,1}\|_{q_0'}\d \x\d s\bigg]\nonumber\\&\leq \overline{\E}\bigg[\left(\int_0^t\int_{\1}\| \overline{\vv}_{n,1}\|_{q_0'}^{q_0'}\d \x\d s\right)^{\frac{1}{q_0'}} \left(\int_0^t\int_{\1} \|\operatorname{div} \big(\overline{\mathtt{G}}_2^{n}-\overline{\theta}_2^{n}\big)\big\|_{q_0}^{q_0}\d \x\d s\right)^{\frac{1}{q_0}}\bigg] \\&\leq T^{\frac{1}{q_0'}} \bigg\{\overline{\E}\bigg[\sup_{s\in[0,t]}\|\nabla \overline{\vv}_{n,1}(s)\|_2^{q_0'}\bigg]\bigg\}^{\frac{1}{q_0'}}\bigg\{\overline{\E}\bigg[\int_0^t\int_{\1} \|\operatorname{div} \big(\overline{\mathtt{G}}_2^{n}-\overline{\theta}_2^{n}\big)\big\|_{q_0}^{q_0}\d \x\d s\bigg] \bigg\}^{\frac{1}{q_0}}\nonumber\\& <\infty,\nonumber
\end{align}
where we have used  H\"older's inequality.

\begin{claim}\label{claim:5}
\begin{align}\label{5.36B}
	\overline{\bfS}=\bfA(\overline{\uu}).
\end{align}
	\end{claim}
\begin{proof}[Proof of Claim~\ref{claim:5}]
 Using density arguments, we can choose our test functions $\bphi \in\W_0^{1,p}(\1)^d\cap \W_0^{1,2}(\1)^d\cap\L^{q_0'}(\1)^d$. In view of the regularity of $\overline{\vv}_{n,1}(\cdot)$, we can apply the infinite-dimensional It\^o formula to the process $\|(\I-\kappa\Delta)^{\frac{1}{2}}\overline{\vv}_{n,1}(\cdot)\|_{\bfH}^2$, to find (see Subsection \ref{PU} for a justification)
\begin{align}\label{MT01}\nonumber
	&\|\overline{\vv}_{n,1}(t)\|_{2}^2+\kappa\|\nabla \overline{\vv}_{n,1}(t)\|_{2}^2\\&\nonumber=	\|\overline{\vv}^{n,1}_0\|_{2}^2+\kappa\|\nabla \overline{\vv}^{n,1}_0\|_{2}^2+2\nu \int_0^t\int_{\1} \operatorname{div}\big(\overline{\mathtt{G}}_1^{n}-\overline{\theta}_1^{n}\I\big)\cdot\overline{\vv}_{n,1}\d \x\d s\\&\nonumber \quad +2\int_0^t\int_{\1} \operatorname{div} \big(\overline{\mathtt{G}}_2^{n}-\overline{\theta}_2^{n}\big)\cdot\overline{\vv}_{n,1} \d \x\d s+ 2\int_0^t\int_\1 \big(\Phi(\overline{\uu}_n)\d \overline{\W}_n(s)-\Phi(\overline{\uu})\d\overline{\W}(s)\big) \cdot\overline{\vv}_{n,1}\d \x \\&\nonumber\quad + 2\int_0^t\int_\1 \big(\Phi_{\pi}(\overline{\uu}_n)\d \overline{\W}_n(s)-\Phi_{\pi}(\overline{\uu})\d\overline{\W}(s)\big)\cdot\overline{\vv}_{n,1}\d \x \\&\nonumber\quad +\int_0^t\int_\1\d \bigg\langle \int_0^{\cdot} \Phi (\overline{\uu}_n)\d\overline{\W}_n-\int_0^{\cdot}\Phi(\overline{\uu})\d\overline{\W}\bigg\rangle_s\d s\\&\nonumber\quad +\int_0^t\int_\1\d \bigg\langle \int_0^{\cdot} \Phi_{\pi} (\overline{\uu}_n)\d\overline{\W}_n-\int_0^{\cdot}\Phi_{\pi}(\overline{\uu})\d\overline{\W}\bigg\rangle_s\d s\\&	=: \sum_{j=1}^{8} I_j^n.
\end{align}
Using the strong convergence \eqref{5.68} and Theorem \ref{pr:th:1} (2), we obtain $\overline{\E}\big[I_1^n\big],\ \overline{\E}\big[I_2^n\big]\xrightarrow[n\to\infty]{}   0$. Now, our aim is to show that the expected values of the terms $I_j^n$ for $j=4,\ldots, 8$ goes to $0$. We know that the expectations of the terms $I_5^n$ and $I_6^n$ are equal to $0$, being  local martingale terms.

The strong convergence \eqref{5.88} gives
\begin{align*}
	\overline{\vv}_{n,1} \xrightarrow[n\to\infty]{} 0, \quad\text{in}\quad \L^{q_0'}({\overline{\Omega},\overline{\mathscr{F}},\overline{\P}}; \L^{q_0'}(0,T;\L_{\sigma}^{q_0'}(\1)^d)), \mbox{ for } \frac{pd}{pd-d+2}\leq  q_0\leq \frac{d}{d-1}.
\end{align*}
Using \eqref{5.77}, \eqref{5.81} and the above consequence of strong convergence \eqref{5.88}, we conclude that $\overline{\E}\big[I_4^n\big]\xrightarrow[n\to\infty]{} 0$ (cf. \eqref{4110}).

Let us consider the term $I_7^n$, and estimate it using the Cauchy-Schwarz inequality as
\begin{align*}
	I_7^n &=\int_0^t\int_\1\d\bigg\langle \int_{0}^{\cdot}\big(\Phi(\overline{\uu}_n)-\Phi(\overline{\uu})\big)\d\W_n\bigg\rangle_s \,\d\x+\int_0^t\int_\1\d\bigg\langle \int_{0}^{\cdot}\Phi(\overline{\uu})\d\big(\overline{\W}_n-\overline{\W}\big)\bigg\rangle_s \,\d\x\\&\quad+\int_0^t\int_\1\d\bigg\langle \int_{0}^{\cdot}\big(\Phi(\overline{\uu}_n)-\Phi(\overline{\uu})\big)\d\overline{\W}_n, \int_{0}^{\cdot}\Phi(\overline{\uu})\d\big(\overline{\W}_n-\overline{\W}\big)\bigg\rangle_s \,\d\x
	\\&
	\leq C \int_0^t\int_\1\d\bigg\langle \int_{0}^{\cdot}\big(\Phi(\overline{\uu}_n)-\Phi(\overline{\uu})\big)\d\overline{\W}_n\bigg\rangle_s \,\d\x+C \int_0^t\int_\1\d\bigg\langle \int_{0}^{\cdot}\Phi(\overline{\uu})\d\big(\overline{\W}_n-\overline{\W}\big)\bigg\rangle_s \,\d\x\\&
	=: C(I_{71}^{n}+I_{72}^{n}).
\end{align*}Using \eqref{3.6a} and  \eqref{5.66}, we find
\begin{align*}
	\overline{\E}[I_{71}^{n}]
	\leq C\overline{\E}\bigg[\int_{0}^{t} \big\|\Phi(\overline{\uu}_n)-\Phi(\overline{\uu})\big\|_{\mathcal{L}_2}^2\d s\bigg]\leq C\overline{\E}\bigg[\int_{0}^{t}\|\overline{\uu}_n-\overline{\uu}\|_2^2\d s\bigg] \xrightarrow[n\to\infty]{} 0.
\end{align*}For the term $I_{72}^{n}$, we {use the fact that} $\overline{\uu}\in \L^2({\overline{\Omega},\overline{\mathscr{F}},\overline{\P}};\L^2(0,T;\bfV))$, \eqref{3.6b} and \eqref{5.65}, to get
\begin{align*}
\overline{\E}[I_{72}^{n}] &=  \overline{\E}\bigg[\int_{0}^{T}\sum_i\bigg(\int_\1 \big|\phi_i(\overline{\uu})\big|^2\Var \big(\overline{\beta}_i^n(1)-\overline{\beta}_i(1)\big)\,\d\x\bigg)\d t\bigg] \\& \leq \overline{\E}\bigg[\int_{0}^{T}\bigg(\int_{\1} \sup_i i^2|\phi_i(\overline{\uu})|^2\,\d\x\bigg)\d t\bigg]\sum_i\frac{1}{i^2}\Var \big(\overline{\beta}_i^n(1)-\overline{\beta}_i(1)\big)
\\& \leq C \overline{\E}\bigg[\int_{0}^{T}\bigg(\int_{\1} \big(1+|\overline{\uu}|^2\big)\,\d\x\bigg)\d t\bigg]\sum_i\frac{1}{i^2}\Var \big(\overline{\beta}_i^n(1)-\overline{\beta}_i(1)\big)
\\&
\leq  C \overline{\E}\bigg[\int_{0}^{T}\bigg(\int_{\1} \big(1+|\overline{\uu}|^2\big)\,\d\x\bigg)\d t\bigg]\E\big[\|\overline{\W}_n(1)-\overline{\W}(1)\|_{\bfU_0}^2\big]
\\&  \leq C\overline{\E}\bigg[\int_{0}^{T} \big(1+\|\overline{\uu}\|_{2}^2\big)\,\d t\bigg]\overline{\E}\left[\|\overline{\W}_n-\overline{\W}\|_{\mathrm{C}([0,T];\bfU_0)}^2\right]\xrightarrow[n\to\infty]{} 0.
\end{align*}We know that $\Phi_\pi$ inherits the properties of $\Phi$. Therefore, $I_8^n$ can be estimated in a similar manner.

Taking expectation in \eqref{MT01}, we find
\begin{align*}
&	2\nu \overline{\E}\bigg[\int_0^t\int_{\1}  \overline{\mathtt{G}}_1^{n}:\bfD(\overline{\vv}_{n,1})\d \x\d s\bigg]\\&=-\overline{\E}\big[\|\overline{\vv}_{n,1}(t)\|_{2}^2+\kappa\|\nabla \overline{\vv}_{n,1}\|_{2}^2\big] +\sum_{j=1}^{2}\overline{\E} \big[I_j^n\big]+2\nu\overline{\E}\bigg[\int_0^t\int_{\1} \nabla \overline{\theta}_1^n\cdot \overline{\vv}_{n,1}\d\x\d s\bigg] +\sum_{j=4}^{8}\overline{\E}\big[I_j^n\big].
\end{align*}
 Combining the above estimates, we arrive at
 \begin{align}\label{4113}
	\limsup_{n\to\infty}\overline{\E}\bigg[\int_0^t\int_{\1}  \big(\bfA(\overline{\uu}_n)-\overline{\bfS}\big):\big(\bfD(\overline{\uu}_{n})-\bfD(\overline{\uu})\big)\d \x\d s\bigg]\leq0,
 \end{align}where we have used the fact that $\operatorname{div}\overline{\vv}_{n,1}=0$.
 Let us now consider
 \begin{align}\label{4114}
 	&	\limsup_{n\to\infty}\overline{\E}\bigg[\int_0^t\int_{\1}  \big(\bfA(\overline{\uu}_n)-\bfA(\overline{\uu})\big):\big(\bfD(\overline{\uu}_{n})-\bfD(\overline{\uu})\big)\d \x\d s\bigg]\nonumber\\&\leq \limsup_{n\to\infty}\overline{\E}\bigg[\int_0^t\int_{\1}  \big(\bfA(\overline{\uu}_n)-\overline{\bfS}:\big(\bfD(\overline{\uu}_{n})-\bfD(\overline{\uu})\big)\d \x\d s\bigg]\nonumber\\&\quad +\limsup_{n\to\infty}\overline{\E}\bigg[\int_0^t\int_{\1}  \big(\overline{\bfS}-\bfA(\overline{\uu})\big):\big(\bfD(\overline{\uu}_{n})-\bfD(\overline{\uu})\big)\d \x\d s\bigg]\leq 0,
 \end{align}
 where we have used \eqref{4113} and also the weak convergence \eqref{5.74}.
 Using the monotonicity of the operator $\bfA(\cdot)$ (see \eqref{2.4}), we  also obtain
 \begin{align}\label{4115}
 	&	\limsup_{n\to\infty}\overline{\E}\bigg[\int_0^t\int_{\1}  \big(\bfA(\overline{\uu}_n)-\bfA(\overline{\uu})\big):\big(\bfD(\overline{\uu}_{n})-\bfD(\overline{\uu})\big)\d \x\d s\bigg]\geq 0.
 \end{align}
 Combining \eqref{4114}, \eqref{4115}, \eqref{2.2} and \eqref{2.3}, one can deduce that (cf. \cite[Eqn. (1.6)]{LBJMMS})
\begin{align*}
	\bfD(\overline{\uu}_n)\xrightarrow[n\to\infty]{}\bfD(\overline{\uu}), \ \ \overline{\P}\otimes \lambda^{d+1}-\text{a.e.}
\end{align*}The above convergence justifies the limit procedure in the energy estimate, that is, $\overline{\bfS}=\bfA(\overline{\uu})$, which completes the proof of Claim~\ref{claim:5} (for more details, see \cite[Appendix A]{JW}).
\end{proof}
This completes the proof of Theorem \ref{thm:exist}.
\end{proof}

\begin{remark}
	Since  $\overline{\uu}\in\L^\infty(0,T;\bfV)\cap \L^p(0,T;{\W_{0}^{1,p}(\1)^d})$ is regular,  we note that in order to obtain the Claim \ref{claim:5}, neither Lipschitz truncation (\cite{HBG}) nor $\L^{\infty}$-truncation (\cite{Breit,JW}) are required.
\end{remark}

\subsection{Pathwise uniqueness of solution}\label{PU}
We rewrite the equation \eqref{1.1} as follows:
	\begin{equation}\label{UU1}
	\begin{aligned}
		\d(\I-\kappa\Delta)^{\frac{1}{2}} \uu & =(\I-\kappa\Delta)^{-\frac{1}{2}}\big\{\operatorname{div} \big(\nu \bfA(\uu)- (\uu\otimes\uu)-\pi \I
		-\bfF \big)\big\}\d t\\&\quad+(\I-\kappa\Delta)^{-\frac{1}{2}}\Phi(\uu)\d \W(t),
	\end{aligned}
\end{equation}
where $\bfF\in\L^2(0,T;{\W^{1,2}(\1)^{d\times d}})$ such that $\operatorname{div} \bfF =-\f$.

The existence of a solution $\uu$ is already established in Theorem \ref{thm:exist} and  $\uu\in \L^\infty(0,T;\bfV)\cap \L^p(0,T;{\W_0^{1,p}(\1)^d})$, $\P-$a.s., for $p>\frac{2d}{d+2}$. Let us set  $\vv(\cdot):=(\I-\kappa\Delta)^{\frac{1}{2}}\uu(\cdot)$. In order to establish the energy equality (It\^o's formula), in view of \cite[Theorem 2.1, Eqn. (1.2)]{IGDS}, we only need to show that
 \begin{align*}
	\vv^*=(\I-\kappa\Delta)^{-\frac{1}{2}}\big\{\operatorname{div} \big( \nu\bfA(\uu)- (\uu\otimes\uu)-\pi \I
	-\bfF \big)\big\}\in\L^1(0,T;{\L^2(\1)^d}).
\end{align*}
In order to  verify $\vv^*\in\L^1(0,T;{\L^2(\1)^d})$, we consider
\begin{align*}
&	\int_0^t\|\vv^*(s)\|_{2}\d s\\&\leq \int_0^t \|(\I-\kappa\Delta)^{-\frac{1}{2}}\operatorname{div} \bfF (s)\|_{2}\d s+\int_0^t\|(\I-\kappa\Delta)^{-\frac{1}{2}}\operatorname{div} \pi \I \|_{2}\d s\\& \quad +\nu\int_0^t \|(\I-\kappa\Delta)^{-\frac{1}{2}}\operatorname{div} \bfA(\uu(s))\|_{2}\d s+ \int_0^t \|(\I-\kappa\Delta)^{-\frac{1}{2}}\operatorname{div} (\uu(s)\otimes\uu(s))\|_{2}\d s\\&=:\sum_{i=1}^{4}I_i.
\end{align*}We consider the term $I_1$ and estimate it using H\"older's inequality, to find
\begin{align*}
	I_1 \leq C\int_0^t\|\bfF(s)\|_{2}\d s \leq C\|\bfF\|_{\L^2(0,T;{\W^{1,2}(\1)^{d\times d}})}<\infty.
\end{align*}Similarly, we can estimate the term $I_2$ as
\begin{align*}
	I_2 \leq C\int_0^t\|\pi(s)\|_{\L^2}\d s\leq C \|\pi\|_{\L^2(0,T;\L^2(\1))}<\infty.
\end{align*}
Now, we consider the term $I_3$ and estimate it in the following way:
\begin{align*}
	I_3\leq C\nu\int_0^t\||\D(\uu(s))|^{p-2}\D(\uu(s))\|_{2}\d s \leq C \int_0^t\|\nabla \uu(s)\|_{2(p-1)}^{p-1}\d s \leq C\sup_{t\in[0,T]}\|\nabla\uu(t)\|^{p-1}_{2}<\infty,
\end{align*}for $p\leq2$.

We estimate the final term $I_4$, with the help of the embedding {$\W^{1,2}_0(\1)^d\hookrightarrow \L^4(\1)^d$ for $2\leq d\leq 4$},  and H\"older's inequality as
\begin{align*}
	I_4  \leq C \int_0^t\|\uu(s)\otimes \uu(s)\|_{2}\d s \leq C\int_0^t \|\uu(s)\|^2_{4}\d s\leq C \sup_{t\in[0,T]}\|\nabla\uu(s)\|^2_{2}<\infty.
\end{align*}

Combining the estimates of $I_1-I_4$, we deduce that $\vv^*\in\L^1(0,T;{\L^2(\1)^d})$ and hence we can use It\^o's formula (see \cite[Theorem 2.1]{IGDS}) to obtain the following energy equality, $\P$-a.s.,
\begin{align}\label{EQ}\nonumber
&	\|\uu(t)\|^2_{2}+{\kappa\|\nabla\uu(t)\|_{2}^2}+2\nu \int_0^t\|\D(\uu(s))\|^p_{p}\d s\\&\nonumber=	\|\uu_0\|^2_{2}+{\kappa\|\nabla\uu_0\|_{2}^2}+2\int_0^t\big(\f(s),\uu(s)\big)\d s +\int_0^t\|\Phi(\uu(s))\|^2_{2}\d s\\&\quad+2\int_0^t\big(\Phi(\uu(s))\d\W(s),\uu(s)\big),
\end{align}for all $t\in[0,T]$, where $p\in \big(\frac{2d}{d+2},2\big]$.

For $p\geq 2$, we consider the Gelfand triplet ${\L^p(\1)^d}\subset \bfH \subset {\L^{p'}(\1)^d}$. In view of \cite[Theorem 2.1, Eqn. (1.2)]{IGDS}, we only need to show the following:
\begin{align}\label{AS2}
\nu	\int_0^t \|(\I-\kappa\Delta)^{-\frac{1}{2}}\operatorname{div} \bfA(\uu(s))\|^{p'}_{p'}\d s<\infty,
\end{align}
and the remaining terms can be estimated with the help of the embedding  $\L^{p'}(0,T;{\L^2(\1)^d})\subset \L^{p'}(0,T;{\L^{p'}(\1)^d})$, for $p\geq 2$. Let us verify $\eqref{AS2}$ in the following way:
\begin{align*}
	\nu\int_0^t \|(\I-\kappa\Delta)^{-\frac{1}{2}}\operatorname{div} \bfA(\uu(s))\|^{p'}_{p'}\d s\leq
C\int_0^t\|\bfA(\uu(s))\|^{p'}_{p'}\d s \leq C\int_0^t\|\nabla \uu(s)\|^p_{p}\d s<\infty.
\end{align*}

Combining  the estimates \eqref{AS2}, $I_1,I_2$ and $I_4$ imply the required energy equality \eqref{EQ} for the case $p\in[2,\infty)$.

Now, we discuss the pathwise uniqueness of the  solution to the system \eqref{1.1}-\eqref{1.4}.
\begin{theorem}[Uniqueness]\label{thrm:uni}
	Under the assumptions of Theorem \ref{thm:exist}, solution of the system \eqref{1.1}-\eqref{1.4} is \emph{pathwise unique}.
\end{theorem}
\begin{proof}
	Let $\uu_1(\cdot),\uu_2(\cdot)$ be any two solutions of the system \eqref{UU1} with the initial data $\uu_0^1$ and $\uu_0^2$, respectively. For $M>0$, let us define
	\begin{align*}
		\tau_M^1=\inf_{t\in[0,T]}\big\{t: \   {\|\nabla\uu_1(t)\|_{2}}\geq M\big\},\ \text{ and }\
			\tau_M^2=\inf_{t\in[0,T]}\big\{t:  \  {\|\nabla\uu_2(t)\|_{2}}\geq M\big\}.
	\end{align*}Set $\tau_M:=\tau_M^1\wedge \tau_M^2$. Let us define $\w(\cdot):=\uu_1(\cdot)-\uu_2(\cdot)$ and $\tilde{\Phi}(\cdot):=\Phi(\uu_1(\cdot))-\Phi(\uu_2(\cdot))$. Then, $\w(\cdot)$ satisfies the following system:
	\begin{equation}\label{UU2}
	\left\{
	\begin{aligned}
		\d(\I-\kappa\Delta)^{\frac{1}{2}} \w & =(\I-\kappa\Delta)^{-\frac{1}{2}}\left[\operatorname{div} \big( \bfA(\uu_1)-\bfA(\uu_2)\big)-\operatorname{div} \big((\uu_1\otimes\uu_1)-(\uu_2\otimes\uu_2)\big)\right.\\&\qquad \left.-\operatorname{div} \big((\pi_1-\pi_2)\I \big)\right]\d t+(\I-\kappa\Delta)^{-\frac{1}{2}}\tilde{\Phi}\d \W,\\
		\w(0)& =\uu_0^1-\uu_0^2,
	\end{aligned}
	\right.
\end{equation}
for all $t\in[0,T]$ in $\mathbf{L}^{p'}$.

Let us define
\begin{align*}
	\varphi(t):=\exp\left(-C_1\int_0^t{\|\nabla\uu_2(s)\|_{2}}\d s\right),
\end{align*}
where $C_1$ is the constant appearing in \eqref{417} below. Applying infinite-dimensional It\^o's formula (see \cite[Theorem 2.1]{IGDS}) to the process $\varphi(\cdot)\|(\I-\kappa\Delta)^{\frac{1}{2}}\w(\cdot)\|^2_{2}$, we find  $\P-$a.s.,
		\begin{align}\label{UU3}\nonumber
	&	\varphi(t\wedge\tau_M)\big\{	\|\w(t\wedge\tau_M)\|^2_{2}+{\kappa\|\nabla\w(t\wedge\tau_M)\|_{2}^2}\big\}\\&\nonumber=	\|\w(0)\|^2_{2}+{\kappa\|\nabla\w(0)\|_{2}^2}-2\nu\int_0^{t\wedge\tau_M}\varphi(s)\big\langle \bfA(\uu_1(s))-\bfA(\uu_2(s)),\bfD (\w)\big\rangle\d s\\&\nonumber\quad +2\int_0^{t\wedge\tau_M}\varphi(s)\bigg(\big\langle \operatorname{div}  \big\{(\uu_1(s)\otimes\uu_1(s))-(\uu_2(s)\otimes\uu_2(s))\big\},\w(s)\big\rangle\\&\nonumber\qquad-C_1{\|\nabla\uu_2(s)\|_{2}}{\|\w(s)\|_{2}^2}\bigg)\d s +\int_0^{t\wedge\tau_M}\varphi(s)\|\tilde{\Phi}(s)\|_{\mathcal{L}_2}^2\d s\\&\quad+2\int_0^{t\wedge\tau_M}\varphi(s) \big(\tilde{\Phi}(s)\d\W(s),\w(s)\big).
		\end{align}
	Using  the following facts:
	\begin{align}\label{417}
		\big|\big\langle \operatorname{div}  \big\{(\uu_1\otimes\uu_1)-(\uu_2\otimes\uu_2)\big\},\w\big\rangle\big|\leq {\|\nabla\uu_2\|_{2}\|\w\|_{4}^2\leq C_1\|\nabla\uu_2\|_{2}\|\nabla\w\|_{2}^2}, 
	\end{align}
for  $2\leq d\leq 4$,
and
	\begin{align*}
\big\langle \bfA(\uu_1)-\bfA(\uu_2),\bfD(\w)\big\rangle \geq 0,
\end{align*}and $\operatorname{div} \w=0,$
 we arrive at $\P-$a.s.,
	\begin{align}\label{UU4}\nonumber
&	\varphi(t\wedge\tau_M)\big\{\|\w(t\wedge\tau_M)\|^2_{2}+{\kappa\|\nabla\w(t\wedge\tau_M)\|_{2}^2}\big\}\\&\nonumber\leq 	\|\w(0)\|^2_{2}+{\kappa\|\nabla\w(0)\|_{2}^2} 
 +\int_0^{t\wedge\tau_M}\varphi(s)\|\tilde{\Phi}(s)\|_{\mathcal{L}_2}^2\d s+2\int_0^{t\wedge\tau_M}\varphi(s) \big(\tilde{\Phi}(s)\d\W(s),\w(s)\big).
\end{align}
Taking expectation, using Hypothesis \ref{3.6a}, Poincar\'e's inequality and the fact that the final term is a martingale, we deduce
\begin{align*}
&	\E\bigg[\varphi(t\wedge\tau_M)\big\{\|\w(t\wedge\tau_M)\|^2_{2}+{\kappa\|\nabla\w(t\wedge\tau_M)\|_{2}^2}\big\}\bigg] \\&\leq \bigg(\frac{1}{\eta_1}+\kappa\bigg){\|\nabla\w(0)\|_{2}^2}+\frac{C(K)}{\eta_1}\E\bigg[\int_0^{t\wedge\tau_M}\varphi(s){\|\nabla\w(s)\|_{2}^2}\d s\bigg].
\end{align*}
An application of Gronwall's inequality yields
\begin{align*}
	\E\bigg[\varphi(t\wedge \tau_M){\|\nabla\w(t\wedge\tau_M)\|_{2}^2}\bigg] \leq C(\eta_1,\kappa,K){\|\nabla\w(0)\|_{2}^2},
\end{align*}where we have used that fact that $\int_0^t{\|\nabla\uu_2(s)\|_{2}}\d s<\infty, \ \P-$a.s.

Thus the initial data $\uu_0^1=\uu_0^2=\uu_0$ leads to
 $\w(t\wedge\tau_M)=0,\ \P-$a.s.  But using the fact that  $\tau_M\to T$, $\P-$a.s., implies $\w(t)=0$ and hence  $\uu_1(t)=\uu_2(t)$, $\P-$a.s., for all $t\in[0,T]$, which completes the proof.\end{proof}

 \begin{proof}[Proof of Theorem \ref{thm2.7}]
We have already established the existence of a probabilistically weak solution and pathwise uniqueness in Theorems \ref{thm:exist} and \ref{thrm:uni}, respectively. Therefore, combining Theorems \ref{thm:exist}, \ref{thrm:uni} and an application of Yamada-Watanabe theorem (see \cite[Theorem 2.1]{MRBSXZ}) leads to the proof of this theorem.
 \end{proof}

	\medskip\noindent
\textbf{Acknowledgments:} The first author would like to thank Ministry of Education, Government of India - MHRD for financial assistance. 
H. B. de Oliveira would like to thank to the Portuguese Foundation for Science and Technology under the project UIDP/04561/2020.  M. T. Mohan would  like to thank the Department of Science and Technology (DST) Science $\&$ Engineering Research Board (SERB), India for a MATRICS grant (MTR/2021/000066).

\medskip\noindent	{\bf  Declarations:}

\noindent 	{\bf  Ethical Approval:}   Not applicable

\noindent  {\bf   Competing interests: } The authors declare no competing interests.

\noindent 	{\bf   Authors' contributions: } All authors have contributed equally.

\noindent 	{\bf   Funding: } DST-SERB, India, MTR/2021/000066 (M. T. Mohan).\\
\url{https://doi.org/10.54499/UIDP/04561/2020} (H. B. de Oliveira).

\noindent 	{\bf   Availability of data and materials: } Not applicable.

\end{document}